\documentclass[3p]{elsarticle}

\pdfoutput=1

\makeatletter
\def\ps@pprintTitle{
    \let\@oddhead\@empty
    \let\@evenhead\@empty
    \def\@oddfoot{}
    \let\@evenfoot\@oddfoot}
\makeatother

\journal{}

\usepackage{graphicx}
\usepackage{amsmath,amsfonts,amsthm,amscd,mathrsfs,amssymb,upref,multirow,color}

\usepackage[utf8]{inputenc}
\usepackage{comment}
\usepackage[dvipsnames]{xcolor}
\usepackage{textcomp}

\newcommand{\R}{\mathbb{R}}
\newcommand{\N}{\mathbb{N}}

\newcommand{\ii}{\mathrm{i}}
\def\d{\,\mathrm{d}}

\newcommand{\vertiii}[1]{{\left\vert\kern-0.25ex\left\vert\kern-0.25ex\left\vert #1 \right\vert\kern-0.25ex\right\vert\kern-0.25ex\right\vert}}

\newcommand{\dd}{\mathrm{d}}
\newcommand{\vpt}{\tilde{v}_{\texttt{pw}}}
\newcommand{\vnt}{\tilde{v}_{\texttt{nw}}}

\newcommand{\fucik}{Fu\v{c}\'{\i}k}

\newtheorem{theorem}{Theorem}
\newtheorem{lemma}[theorem]{Lemma}
\newtheorem{proposition}[theorem]{Proposition}
\newtheorem{conjecture}[theorem]{Conjecture}
\newtheorem{corollary}[theorem]{Corollary}
\theoremstyle{definition} 
\newtheorem{remark}[theorem]{Remark}

\newtheorem*{acknowledgement}{Acknowledgement}

\begin{document}
	
\begin{frontmatter}
	
	\title{Lower Bounds for Admissible Values of the Travelling Wave Speed in~Asymmetrically Supported Beam}
	
	\author{Hana Form\'{a}nkov\'{a} Lev\'{a}}
	\ead{levah@kma.zcu.cz}
	\author{Gabriela Holubov\'{a}}
	\ead{gabriela@kma.zcu.cz}
	\author{Petr Ne\v{c}esal}
	\ead{pnecesal@kma.zcu.cz}
	
	\address{Department of Mathematics and NTIS, Faculty of Applied Sciences, University of~West Bohemia, Univerzitn\'{\i}~8, 301~00~Plze\v{n}, Czech Republic}
	
	\begin{abstract}
		We study the admissible values of the wave speed $c$ for which the beam equation with jumping nonlinearity possesses a travelling wave solution. In contrast to previously studied problems modelling suspension bridges, the presence of the term with negative part of the solution in the equation results in restrictions of $c$. In this paper, we provide the maximal wave speed range for which the existence of the travelling wave solution can be proved using the Mountain Pass Theorem. We also introduce its close connection with related Dirichlet problems and their \fucik~spectra. Moreover, we present several analytical approximations of the main existence result with assumptions that are easy to verify. Finally, we formulate a conjecture that the infimum of the admissible wave speed range can be described by the \fucik~spectrum of a simple periodic problem.
	\end{abstract}
		

\end{frontmatter}


\section{Motivation}
\label{sec:motivation}

When studying the existence of homoclinic travelling waves in asymmetrically supported vibrating beams, i.e., solutions of the problem
\begin{equation}
\label{eq:original_problem}
\left\{ \begin{array}{l}
u_{tt}+u_{xxxx}+ a u^+ - b u^- = 1, \quad x \in \R, \, t>0,\\[3pt]
u(x,t) \rightarrow 1/a, ~u_{x}(x,t) \rightarrow 0~~\textup{for}~|x| \rightarrow +\infty,
\end{array} \right.
\end{equation}
which can be expressed in the form $u(x,t) = y(x-ct)$,
the following basic question naturally arises:

\begin{center}
\emph{$Q_c$: What are the admissible values of wave speed $c$ for which the travelling wave solutions exist?}
\end{center}

\noindent
In \eqref{eq:original_problem}, 
$a$, $b$ are positive constants and $u^+$ and $u^-$ denote the positive and negative parts of $u$, i.e., $u^{\pm}(x,t) = \max\{\pm u(x,t),0\}$, $u = u^+-u^-$. 
Substituting $u(x,t) = y(x-ct) = y(s)$, the PDE in \eqref{eq:original_problem} is transformed into an ODE
\begin{equation}
\label{eq:transformed_ode1}
y^{(4)}+c^2 y''+ a y^{+}-b y^{-} = 1.
\end{equation}
Further, putting $z(s) := y(s) - 1/a$, i.e., translating the equilibrium $1/a$ of $y$ to zero,  we come to the problem
\begin{equation}
\label{eq:transformed_ode2}
\left\{ 
\begin{array}{ll} 
z^{(4)}+c^2 z''+ a (z+1/a)^{+}-b (z+1/a)^{-} - 1 = 0,\\[3pt]
z, \, z' \rightarrow 0~~\textup{for}~|s| \rightarrow +\infty.  
\end{array} 
\right.
\end{equation}
A necessary condition for the existence of $z$ satisfying \eqref{eq:transformed_ode2} 
follows from \cite{karageorgis} and can be written as
\begin{equation}
\label{eq:necessary_con}
|c| \in \left(\sqrt[4]{4b},\sqrt[4]{4a}\right).
\end{equation}
In particular, Corollary 2 in \cite{karageorgis} states that \eqref{eq:transformed_ode2}
has only a trivial solution for $c$ that does not fulfill $\eqref{eq:necessary_con}$. 
It was shown in \cite{chen_mckenna} that the condition \eqref{eq:necessary_con} 
for $b = 0$, i.e., $|c| \in \left(0,\sqrt[4]{4a}\right)$, 
is also the sufficient one and that this interval stays preserved even if we include an additional nontrivial nonlinearity $g(u)$ satisfying suitable assumptions into the original model (for details see \cite{chen_mckenna}, \cite{chen} or \cite{holubova_leva}).

As for the sufficient condition in the case $b > 0$, 
it was proved in \cite{holubova_leva} 
that the nontrivial solutions of \eqref{eq:transformed_ode2} and hence the travelling wave solutions of \eqref{eq:original_problem} exist for any $c$ satisfying
$$
|c| \in \left(\sqrt[4]{100b/9},\sqrt[4]{4a}\right).
$$
The authors in \cite{holubova_leva} however admit that the lower bound $\sqrt[4]{100b/9}$ is probably too restrictive and the range of $|c|$ should be extended below this value.

The aim of this paper is to describe the optimal lower bound for $|c|$, i.e.,  to enlarge as much as possible the interval of admissible wave speeds $c$ for which the travelling wave solution of \eqref{eq:original_problem} exists. The paper is organized as follows.
The next section provides the first (theoretical) answer to $Q_c$, namely, the maximal wave speed range for which the existence of the travelling wave solution can be proved using the mountain pass technique (cf. Theorem \ref{th:main1}). 
Sections \ref{sec:omega_minus} and \ref{sec:connection} concretize the description, 
show the connection of the wave speed bounds with the \fucik\ spectrum of related Dirichlet problems and 
reformulate the result of Theorem \ref{th:main1} into a more transparent version (cf. Theorem \ref{th:main2}). 
In particular, we prove that the existence of a travelling wave solution is guaranteed for any 
$|c| \in \left(c^{\ast},\sqrt[4]{4a}\right)$ with $c^{\ast} > \sqrt[4]{4b}$ determined by the envelope 
of all related Dirichlet \fucik\ spectra. 
The obtained wave speed range is visualized in Figure \ref{fig:main} and an open question left 
is whether it can be further extended up to $|c| \in \left(\sqrt[4]{4b},\sqrt[4]{4a}\right)$ using 
some other (nonvariational) techniques or whether there are really no nontrivial solutions 
for $|c|$ close enough to $\sqrt[4]{4b}$.

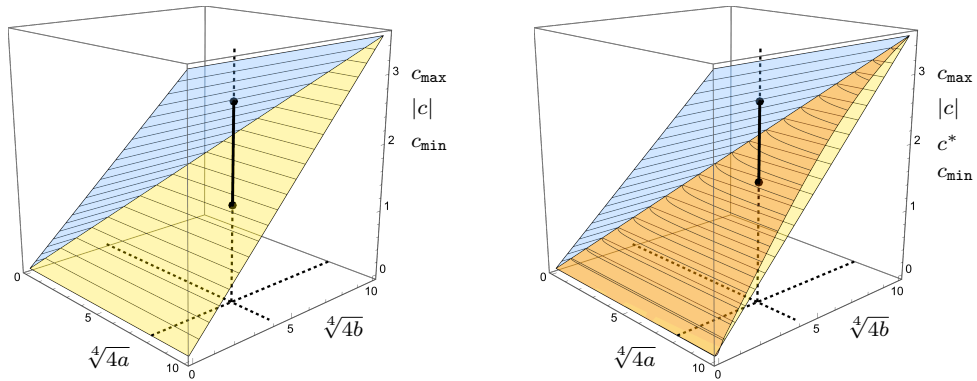
\begin{figure}[h]
\centerline{
  \setlength{\unitlength}{1mm}
  \begin{picture}(50, 50)(0,0)
    \put(0,0){\includegraphics[height=5.0cm]{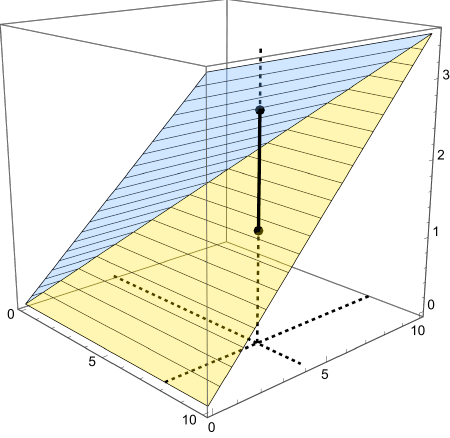}}
    \put(10,2){\makebox(0,0)[lb]{\footnotesize$\sqrt[4]{4a}$}}         
    \put(41,6){\makebox(0,0)[lb]{\footnotesize$\sqrt[4]{4b}$}}         
    \put(53,40){\makebox(0,0)[lb]{\footnotesize$c_{\texttt{max}}$}}                     
    \put(53,35){\makebox(0,0)[lb]{\footnotesize$|c|$}}                 
    \put(53,31){\makebox(0,0)[lb]{\footnotesize$c_{\texttt{min}}$}}                         
  \end{picture}    
  \hspace{48pt}
  \begin{picture}(50, 50)(0,0)
    \put(0,0){\includegraphics[height=5.0cm]{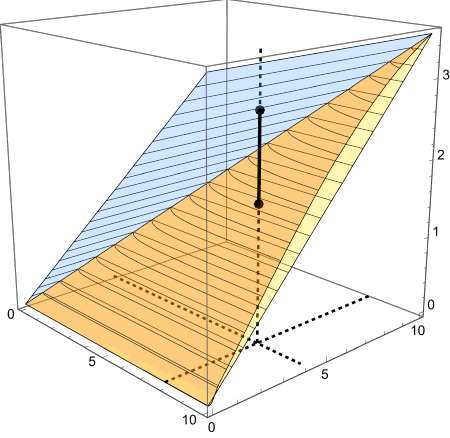}}
    \put(10,2){\makebox(0,0)[lb]{\footnotesize$\sqrt[4]{4a}$}}         
    \put(41,6){\makebox(0,0)[lb]{\footnotesize$\sqrt[4]{4b}$}}             
    \put(53,40){\makebox(0,0)[lb]{\footnotesize$c_{\texttt{max}}$}}                     
    \put(53,35){\makebox(0,0)[lb]{\footnotesize$|c|$}}                 
    \put(53,31){\makebox(0,0)[lb]{\footnotesize$c^{\ast}$}}                         
    \put(53,27){\makebox(0,0)[lb]{\footnotesize$c_{\texttt{min}}$}}                             
  \end{picture}    
  }
\caption{
Dependence of the lower and upper bounds of the necessary condition $|c|\in(c_{\texttt{min}},c_{\texttt{max}})$  
and the sufficient condition  $|c|\in(c^{*},c_{\texttt{max}})$  on parameters $a$ and $b$ ($0 < b < a$). 
Cf. the main existence
result in Theorem \ref{th:main2}.
Here, $c_{\texttt{min}}$ is in yellow,  $c_{\texttt{max}}$ in blue and $c^{\ast}$ is in orange color.}
\label{fig:main}
\end{figure}
 
In Sections \ref{sec:approx1} and \ref{sec:approx2}, we prove several analytical approximations 
of results from the previous sections. In particular, we construct two different families 
of subintervals of $\left(c^{\ast},\sqrt[4]{4a}\right)$ which are easy to evaluate
(cf. Theorems \ref{veta11} and \ref{veta12}) and find their envelopes in a reversed setting 
meaning that we fix the wave speed $c$ and characterize 
the admissible values of parameters $a$ and $b$
(cf. Theorems \ref{veta22} and \ref{veta33}).
Finally, in Section \ref{sec:conjecture}, we introduce
another way how to describe the admissible wave speed range, however, all assertions stated 
in Conjecture \ref{con:periodic1} remain in the form of hypotheses, and their proofs or disproofs are still open. 
These and also other open questions are summarized in Section~\ref{sec:final}.
Some auxiliary technical derivations are intentionally moved to Appendices.


\section{First answer to $Q_c$}
\label{sec:answer1}

Let us denote $\Omega:=(1,+\infty)\times(0,1)$, consider the Hilbert space $H^2(\R) = W^{2,2}(\R)$ with the standard norm
$\|z\|_{H^2}^2 =\int_{\R}\left((z'')^2 + (z')^2 + z^2\right) \d s$ and introduce 
the functional
$J: ~\R^2 \times H^2(\R) \to \R$ in the form
\begin{equation}
\label{eq:J}
J(\alpha,\beta,v) := \int\limits_{\R}\left((v'')^2-2(v')^2 + \alpha (v^+)^2  
+ \beta(v^-)^2 \right) \d x
\end{equation}
and the set
\begin{equation}
\label{eq:omega_minus}
\Omega^{\texttt{-}} := \left\{ (\alpha,\beta)\in \Omega: ~ \exists v \in H^2(\R): ~ J(\alpha,\beta,v) < 0 \right\}.
\end{equation}
The first answer to $Q_c$ (in the sense of a sufficient condition) can be formulated as follows.

\begin{theorem}
\label{th:main1}
Let $a > b > 0$ be arbitrary but fixed.
Then the travelling wave solution of \eqref{eq:original_problem} exists for any wave speed $c$ satisfying
\begin{equation}
\label{eq:c_in_omega_minus}
\left(\frac{4a}{c^4},\frac{4b}{c^4}\right) \in \Omega^{\texttt{-}}.
\end{equation}
\end{theorem}

\begin{proof}
All weak (and due to regularity also classical) solutions $z \in H^2(\R)$ of \eqref{eq:transformed_ode2} correspond to critical points of the functional $I \in C^1(H^2(\R),\R)$,
\begin{equation}
\label{eq:functional}
I(z) := 
\dfrac{1}{2}\int\limits_{\R}\left((z'')^2-c^2(z')^2 + az^2 - (a-b)\left((z + 1/a)^-\right)^2 \right)\d s. 
\end{equation}
Trivially, $I(0) = 0$. Exploiting the results of \cite{holubova_leva}, we can see that $I$ has a local minimum at $0$ and another nontrivial critical point exists, if $I$ has the mountain pass geometry. This requires the existence of some $e \in H^2(\R) \setminus \{0\}$ such that $I(e) < 0$. Considering $e = Av$ with $A$ being a positive real number and $v$ some fixed element in $H^2(\R)$, we can write
\begin{equation}
\label{eq:I(e)}
I(Av) = \dfrac{A^2}{2} \int\limits_{\R}\left((v'')^2-c^2(v')^2 + a v^2 - (a-b)\left(\left(v + 1/(aA)\right)^-\right)^2 \right)\d s .
\end{equation}
Hence, for $A$ and $v$
such that the integral in \eqref{eq:I(e)} is negative, we obtain the required condition $I(e) < 0$. Moreover, 
\begin{equation}
\label{eq:functional_reduced}
\lim_{A \to +\infty} \frac{2I(Av)}{A^2} = \int\limits_{\R}\left((v'')^2-c^2(v')^2 +a (v^+)^2  
+ b(v^-)^2 \right) \d s =: {\cal J}(a,b,c,v).
\end{equation}
Transforming $x = s |c|/\sqrt{2}$ and denoting
\begin{equation}
\label{eq:alpha_beta}
\alpha := \frac{4a}{c^4}, \qquad \beta := \frac{4b}{c^4}
\end{equation}
we obtain ${\cal J}(a,b,c,v) = {{|c|^3}/(2\sqrt{2}}) \, J(\alpha,\beta,v)$ 
with $J$ given by \eqref{eq:J}.

Now, for any $c$ 
satisfying \eqref{eq:c_in_omega_minus}, 
we have $(\alpha,\beta) \in \Omega^{\texttt{-}}$.
That is, 
there exists $v \in H^2(\R)$ such that $J(\alpha,\beta,v)<0$.
Consequently, ${\cal J}(a,b,c,v) < 0$ and 
hence $I(e) < 0$ with $e = Av$ for $A>0$ large enough. This implies that the functional $I$ has the required geometry and the existence of the travelling wave solution of \eqref{eq:original_problem} then follows directly from the main result in \cite{holubova_leva}.
\end{proof}

\begin{remark}~
\label{rem:omega}
\begin{enumerate}
\item 
Due to \eqref{eq:alpha_beta} and the necessary condition for 
the existence of the travelling wave solution $\sqrt[4]{4b} < |c|< \sqrt[4]{4a}$, 
we have $\beta < 1 < \alpha$. Hence, the restriction to $\Omega$ in \eqref{eq:omega_minus} makes sense.
\item
Let us note that for $\alpha > 1$ and $\beta \geq 1$, we obtain 
\begin{equation}  \label{mrce1}
J(\alpha,\beta,v) \geq  \int\limits_{\R} \left((v'')^2-2(v')^2 + v^2 \right) \d x 
=  \int\limits_{\R} \left(\omega^2-1 \right)^2\hat{v}^2 \d \omega
> 0 
\end{equation}
for any nontrivial $v \in H^2(\R)$. Here, $\hat{v}(\omega) = \mathcal{F}(v(x))$ denotes the Fourier transform of $v=v(x)$.
Further, we have 
$I(Av) \geq A^2 |c|^3/(4\sqrt{2}) J(\alpha,\beta,v)$, 
since $(v + 1/(aA))^- \leq v^-$ for any $aA > 0$. Hence, if $J(\alpha,\beta,v) \geq 0$ for all $v$, the functional $I$ loses the mountain pass geometry around zero.
\end{enumerate}
\end{remark}
 
In the following sections we provide several characterizations of the set $\Omega^{\texttt{-}}$ leading to more explicit descriptions of the admissible wave speed range.

\section{Characterization of $\Omega^{\texttt{-}}$}
\label{sec:omega_minus}

One partial characterization of $\Omega^{\texttt{-}}$, and hence answer to $Q_c$, was already provided in \cite{holubova_leva}. 

\begin{lemma}[cf. \cite{holubova_leva}]
\label{lemma:9/25}
Let $\Omega^{\texttt{-}}$ be defined by \eqref{eq:omega_minus}. Then 
$$
(1,+\infty)\times \big(0, \tfrac{9}{25}\big) \subset \Omega^{\texttt{-}}.
$$
\end{lemma}

\begin{proof}
Taking $v_0(x) =  -4 \cos^3(x/\sqrt{5}) / \sqrt{5\sqrt{5}\pi}$ for 
$x \in [-\sqrt{5}\pi/2, \sqrt{5}\pi/2]$ and $v_0(x) = 0$ elsewhere, we can directly compute 
$J(\alpha,\beta,v_0) = \beta - \frac{9}{25}$. Hence the assertion follows.
\end{proof}

To provide the description of $\Omega^{\texttt{-}}$,
let us set $\alpha = q \beta$ with $q \in (1,+\infty)$ arbitrary but fixed, i.e., 
$$
J(\alpha,\beta, v) = J(q\beta,\beta,v) = \int\limits_{\R}\left((v'')^2-2(v')^2 + \beta \left( q(v^+)^2  
+ (v^-)^2 \right) \right)\d x
$$
and denote
\begin{equation}
\label{eq:beta_star_orig}
\beta^{\ast}(q)  := \sup\left\{  \beta>0: ~ \exists v \in H^2(\R): ~ J(q\beta,\beta,v) = 0 \right\},\quad q > 1.
\end{equation}
Notice that Lemma \ref{lemma:9/25} and Remark \ref{rem:omega} immediately imply 
$\frac{9}{25} \leq \beta^{\ast}(q) \leq 1$ for any $q > 1$.
Using $\beta^{\ast}$, we can characterize $\Omega^{\texttt{-}}$ in the following way.

\begin{lemma}
\label{lemma:omega_minus}
Let $\beta^{\ast}$ be given by \eqref{eq:beta_star_orig}. Then
\begin{equation}
\label{eq:omega_minus_2}
\Omega^{\texttt{-}}  = \left\{(q\beta, \beta)\in\Omega: q > 1,\ \beta < \beta^{\ast}(q)\right\}.
\end{equation}
\end{lemma}

\begin{proof}
Indeed, if $\beta_0$ and $v_0$ are such that $J(q\beta_0,\beta_0,v_0) = 0$, then $J(q\beta,\beta,v_0) < 0$ for any $\beta < \beta_0$. The assertion thus follows directly from the definitions of $\Omega^{\texttt{-}}$ and $\beta^{\ast}$ (cf. \eqref{eq:omega_minus}, \eqref{eq:beta_star_orig}).
\end{proof}

By Lemma \ref{lemma:omega_minus}, we have converted the problem of describing $\Omega^{\texttt{-}}$ to the problem of describing~$\beta^{\ast}$.

\section{Connection of $\beta^{\ast}$ to the \fucik\ spectrum and another answer to~$Q_c$}
\label{sec:connection}

First of all, since 
$\int_{\R} (q(v^+)^2 + (v^-)^2) \d x \geq \int_{\R} v^2 \d x > 0$ for any nontrivial $v \in H^2(\R)$ and $q>1$, we can set
$$
S_q := \left\{ v \in H^2(\R): ~\int_{\R} (q(v^+)^2 + (v^-)^2) \d x = 1 \right\}
$$
and express
\begin{equation*}
 \beta^{\ast} (q) = \sup_{v\in H^2(\R)\setminus\{0\}} \frac{\int_{\R} (2(v')^2 - (v'')^2) \d x}{\int_{\R} (q(v^+)^2 + (v^-)^2) \d x} = \sup_{v \in S_q} \int_{\R} (2(v')^2 - (v'')^2) \d x.
\end{equation*}
For an arbitrary $r > 0$, let us now consider the space
$
H_0^2(-r,r) := \left\{v \in H^2(-r,r): ~v(\pm r) = v'(\pm r) = 0 \right\}
$
and define the function $\beta_r^{\ast}$ as
\begin{equation}
\label{eq:beta_star_a}
\beta_r^{\ast}(q) := \sup_{v \in S_{r,q}} \int_{-r}^r (2(v')^2 - (v'')^2) \d x,\quad q > 1,
\end{equation}
where the manifold $S_{r,q}$ is given by
$$
S_{r,q} := \left\{ v \in H_0^2(-r,r): ~\int_{-r}^{r} (q(v^+)^2 + (v^-)^2) \d x = 1 \right\}.
$$
Similarly as in the case of $\beta^{\ast}$, we easily obtain 
that $\beta^{\ast}_r(q) \geq \frac{9}{25}$ for any $r \geq \sqrt{5}\pi/2$ and $q > 1$
(we can choose the same $v_0 \in S_{r,q}$ as in the proof of Lemma \ref{lemma:9/25}), 
or that $\beta^{\ast}_r(q) > 0$ for any $r > \sqrt{2}\pi/2$ and $q > 1$
(for details and suitable choice of $v$, see Remark \ref{rem:eigen} in \ref{sec:eigen}). 
For these reasons, we limit our considerations to $r > \sqrt{2}\pi/2$ in the following assertions.

\begin{proposition}
\label{prop:beta_lim}
For any $q > 1$ fixed, we have
\begin{equation}
\label{eq:beta_lim}
\beta^{\ast}(q) = \sup_{r > \sqrt{2}\pi/2} \beta_r^{\ast}(q) = \lim_{r \to +\infty} \beta_r^{\ast}(q).
\end{equation}
\end{proposition}

\begin{proof}
Considering an arbitrary pair $\sqrt{2}\pi/2 < r_1 < r_2$, any $v \in H_0^2(-r_1,r_1)$ extended by zero to $(-r_2,r_2)$ belongs to $H_0^2(-r_2,r_2)$, similarly any $v \in S_{r_1,q}$ extended by zero to $(-r_2,r_2)$ belongs to $S_{r_2,q}$. Hence, we obtain $0< \beta_{r_1}^{\ast}(q) \leq \beta_{r_2}^{\ast}(q) \leq \beta^{\ast}(q)$.
Moreover, any $v\in H^2(\R)$ can be approximated by functions with compact support, i.e., by functions from $H_0^2(-r,r)$ with sufficiently large $r$.
Thus, we can conclude that \eqref{eq:beta_lim} holds true.
\end{proof}

For another fundamental characterization of $\beta^{\ast}$, let us consider the Dirichlet boundary value problem
\begin{equation}
\label{eq:fucik}
\left\{ \begin{array}{l}
v^{(4)} + 2v'' + \alpha v^+ - \beta v^- = 0, \quad x \in (-r,r),\\[3pt]
v(\pm r) = v'(\pm r) = 0,
\end{array} \right.
\end{equation}
and denote by $\Sigma^{\texttt{D}}_r$ the corresponding \fucik\ spectrum of \eqref{eq:fucik}, i.e.,
$$
\Sigma^{\texttt{D}}_r :=  
\left\{ (\alpha,\beta) \in \R^2: ~ \mbox{\eqref{eq:fucik} has a nontrivial solution} \right\}.
$$
\begin{remark}
Let us note that the linear term $v^{(4)} + 2v''$ in \eqref{eq:fucik} 
can be formally written as $(\partial^2 + 1)^2 v - v$ and the operator 
$(\partial^2 + 1)^2$ is sometimes called the (stationary) \emph{Swift-Hohenberg operator}. 
Its origin comes from the fourth order term in Swift--Hohenberg equation 
in \cite{Swift_1977} which was used to study hydrodynamic behaviour of fluids. 
Nowadays, it finds many other applications, e.g., in models of pattern formation.
\end{remark}
We claim the following.

\begin{proposition}
\label{prop:connection}
For any $r > \sqrt{2}\pi/2$ and $q > 1$, we have 
$$
\beta^{\ast}_r(q) = \max \left\{ \beta > 0: \, (q\beta,\beta) \in \Sigma^\textup{\texttt{D}}_r \right \}.
$$
Consequently,
\begin{equation}
\label{eq:beta_sup_max}
\beta^{\ast}(q) = \sup_{r > \sqrt{2}\pi/2} \max \left\{ \beta > 0: \, (q\beta,\beta) \in \Sigma^\textup{\texttt{D}}_r \right \}. 
\end{equation}
\end{proposition}

\begin{proof}
First of all, we show that the supremum in \eqref{eq:beta_star_a} is really attained. Indeed, there must exist a sequence $(v_n)_{n=1}^{+\infty} \subset S_{r,q}$ such that $\int_{-r}^r (2(v_n')^2 - (v_n'')^2) \d x \to \beta_r^{\ast}(q)$. 
Using the fact that $\vertiii{v} = \left(\int_{-r}^r (v'')^2 \d x\right)^{1/2}$ is an equivalent norm on $H_0^2(-r,r)$, 
$\int_{-r}^r v^2 \d x \leq 1$ on $S_{r,q}$ and applying the estimate
$$
2\int_{-r}^r (v')^2 \d x = -2\int_{-r}^r v''v \d x \leq p^2 \int_{-r}^r (v'')^2 \d x + \frac{1}{p^2}  \int_{-r}^r v^2 \d x
\leq p^2 \int_{-r}^r (v'')^2 \d x + \frac{1}{p^2} 
$$
with $p \in (0,1)$,
we can write
$$
(1-p^2) \vertiii{v_n}^2 \leq \frac{1}{p^2}  - \int_{-r}^r (2(v_n')^2 - (v_n'')^2)
 \d x  \quad \to \quad \frac{1}{p^2} - \beta_r^{\ast}(q) \in \R.
$$
Hence $(v_n)$ is bounded in $H_0^2(-r,r)$. Due to the reflexivity of $H_0^2(-r,r)$ and its compact embedding into $H_0^1(-r,r)$, there exists a subsequence (denoted again by $(v_n)$) and $v \in H_0^2(-r,r)$ such that
$$
v_n \rightharpoonup v \quad \mbox{in } H_0^2(-r,r), \qquad v_n \to v \quad \mbox{in } L^2(-r,r), \qquad v'_n \to v' \quad \mbox{in } L^2(-r,r).
$$
This implies $\int_{-r}^r (q(v^+)^2 + (v^-)^2) \d x = \lim_{n \to +\infty} \int_{-r}^r (q(v^+_n)^2 + (v^-_n)^2) \d x = 1$, i.e., $v \in S_{r,q}$. Using the definition of $\beta_r^{\ast}(q)$ and the weak semicontinuity of the norm, we finally obtain
\begin{eqnarray*}
\beta_r^{\ast}(q) & \geq & \int_{-r}^r \left(2(v')^2 - (v'')^2\right) \d x = 2\|v'\|^2_{L^2} - \vertiii{v}^2 \geq 2 \lim\limits_{n \to +\infty} \|v'_n\|^2_{L^2} - \liminf\limits_{n \to +\infty} \vertiii{v_n}^2 \\
& = & \lim\limits_{n \to +\infty}  \int_{-r}^r \left(2(v_n')^2 - (v_n'')^2\right) \d x = \beta_r^{\ast}(q),
\end{eqnarray*}
that is, $\beta_r^{\ast}(q)$ is attained at $v \in S_{r,q}$.

Now, since $v \not\equiv 0$, we can apply the standard Lagrange multiplier method (cf. \cite[Theorem 7.8.2]{drabek_milota}) and conclude that there exist $\lambda \in \R$ such that
\begin{equation}
\label{eq:weak_fucik}
\int_{-r}^{r} \left(v''\varphi'' - 2v'\varphi' + \lambda(q v^+ - v^-)\varphi\right) \d x = 0 \quad \mbox{for all } \varphi \in H_0^2(-r,r).
\end{equation} 
Taking $\varphi = v$ in \eqref{eq:weak_fucik}, we get $\lambda = \beta_r^{\ast}(q)$. Moreover, \eqref{eq:weak_fucik} is a weak formulation of \eqref{eq:fucik} with $(\alpha,\beta) = (q\beta_r^{\ast}(q),\beta_r^{\ast}(q))$. Due to standard regularity arguments, the nontrivial weak solution of \eqref{eq:fucik} coincides with the classical one and $(q\beta_r^{\ast}(q),\beta_r^{\ast}(q)) \in \Sigma^{\texttt{D}}_r$.

Relation \eqref{eq:beta_sup_max} is then a direct consequence of Proposition \ref{prop:beta_lim}.
\end{proof}

Since our primary goal is the existence of travelling wave solutions of \eqref{eq:original_problem}, 
we assume $a > b > 0$ and thus, for $\alpha$ and $\beta$ given by \eqref{eq:alpha_beta},
we have $\alpha/\beta = q > 1$. However, Proposition \ref{prop:connection} could be formulated and proved to cover more general cases. First, it is not necessary to limit ourselves to $r > \sqrt{2}\pi/2$. 
But, if we consider the case of $0 < r \leq \sqrt{2}\pi/2$ then
we have to admit that
$\beta^{\ast}_r(q) = \max \left\{ \beta \in \R: \, (q\beta,\beta) \in \Sigma^{\texttt{D}}_r \right \}$
could be negative or zero.
Next, we can include also the case $q = 1$. 
This results in the following corollary.

\begin{corollary}
\label{cor:r=1}
For $q \to 1+$, we have
$\beta_r^{\ast}(q) \to \lambda_{1,r}^\textup{\texttt{D}}$ being the first (i.e., the largest) eigenvalue of 
\begin{equation}
\label{eq:eigen}
\left\{ \begin{array}{l}
v^{(4)} + 2v'' + \lambda v = 0, \quad x \in (-r,r),\\[3pt]
v(\pm r) = v'(\pm r) = 0. 
\end{array} \right.
\end{equation}
\end{corollary}

A detailed description of $\lambda_{1,r}^{\texttt{D}}$ can be found in \ref{sec:eigen}. 
As we have stated above, $\frac{9}{25} \leq \beta^{\ast}(q) \leq 1$ for any $q > 1$.
Moreover, the following holds true.

\begin{lemma}
\label{lemma:nonempty}
For any $q > 1$ we have $ \beta^{\ast}(q) > 1/q $.
\end{lemma}

\begin{proof}
For simplicity, let us denote 
$$
R_{q,r}(v) := \frac{\int_{-r}^{r}(2(v')^2 - (v'')^2) \d x}{\int_{-r}^r (q(v^+)^2 + (v^-)^2) \d x}.
$$
Due to Proposition \ref{prop:connection} and \eqref{eq:beta_star_a}, we can write
$$
\beta^{\ast}_r(q) = \max_{v \in H^2_0(-r,r) \setminus\{0\}} R_{q,r}(v) = R_{q,r}(v_{q,r}) \geq R_{q,r}(v_{1,r}),
$$
where $v_{q,r}$ solves \eqref{eq:fucik} with 
$(\alpha,\beta) = (q\beta^{\ast}_r(q), \beta^{\ast}_r(q))$ and 
$v_{1,r}$ is the eigenfunction of \eqref{eq:eigen} corresponding 
to the first (largest) eigenvalue $\lambda^{\texttt{D}}_{1,r}$ such that 
$\int_{-r}^r (v^-_{1,r})^2 \d x \geq \frac{1}{2} \int_{-r}^r v^2_{1,r} \d x$. 
Notice that for any function $v$, $v^- = (-v)^+$ and 
$\int_{-r}^r (v^+)^2 \d x + \int_{-r}^r (v^-)^2 \d x = \int_{-r}^r v^2 \d x$, 
that is either $v_{1,r}$ or $-v_{1,r}$ must fulfill our requirements.

Trivially, using the fact that $R_{1,r}(v_{1,r}) = \lambda^{\texttt{D}}_{1,r} > 0$ for $r > \sqrt{2}\pi/2$, we have
$$
R_{q,r}(v_{1,r}) = R_{1,r}(v_{1,r}) \frac{\int_{-r}^r v^2_{1,r} \d x}{\int_{-r}^r (q(v^+_{1,r})^2 + (v^-_{1,r})^2) \d x} = 
\lambda^{\texttt{D}}_{1,r} \left({q - (q-1)\frac{\int_{-r}^r (v^-_{1,r})^2 \d x}{\int_{-r}^r (v_{1,r})^2\d x}}\right)^{-1} 
\geq  \frac{\lambda^{\texttt{D}}_{1,r}}{q - \frac{q-1}{2}}.
$$
That is, for any $r > \sqrt{2}\pi/2$ and $q > 1$, $\beta^{\ast}_r(q) \geq {\lambda^{\texttt{D}}_{1,r}}/(q - \frac{q-1}{2})$. Passing to the limit for $r \to + \infty$,  we obtain $\lambda^{\texttt{D}}_{1,r} \to 1$ (cf. \ref{sec:eigen}) and hence
$\beta^{\ast}(q) \geq 1/(q - \frac{q-1}{2}) > 1/q$, what we wanted to prove.
\end{proof}

Proposition \ref{prop:connection} says, in other words, that $(q\beta_r^{\ast}(q),\beta_r^{\ast}(q))$ 
coincides with the ``highest point'' in the intersection of the \fucik\ spectrum $\Sigma^{\texttt{D}}_r$ 
of \eqref{eq:fucik} and the ray $\alpha = q\beta$. 
Due to Lemma \ref{lemma:nonempty}, $q\beta_r^{\ast}(q) > 1$, i.e., 
$\Omega^{\texttt{-}}$ is nonempty on any ray $\alpha = q\beta$, $q > 1$.
If we further introduce 
$\Omega^{\texttt{+}}  := \left\{(q\beta, \beta)\in\Omega: q > 1,\ \beta > \beta^{\ast}(q)\right\}$ and
\begin{equation}
\label{eq:envelope}
\mathscr{E} := \left\{(q\beta^{\ast}(q),\beta^{\ast}(q)): \, q > 1 \right\},
\end{equation}
then $\mathscr{E}$ can be understood as the \emph{envelope} of the collection
of all \fucik\ spectra $\Sigma^{\texttt{D}}_r$ with $r > \sqrt{2}\pi/2$ and 
represents the border line under which $J(\alpha,\beta,v)$ possesses negative
values for some $v \in H^2(\R)$.

Using the above results, we are now ready to state our main theorem providing
another version of the answer to $Q_c$. It is visualized in Figures \ref{fig:main} and \ref{fig:main_2} and reads as follows.

\begin{figure}[h]
\centerline{
  \setlength{\unitlength}{1mm}
  \begin{picture}(66, 61)(-3,-1)
    \put(0,0){\includegraphics[height=6.0cm]{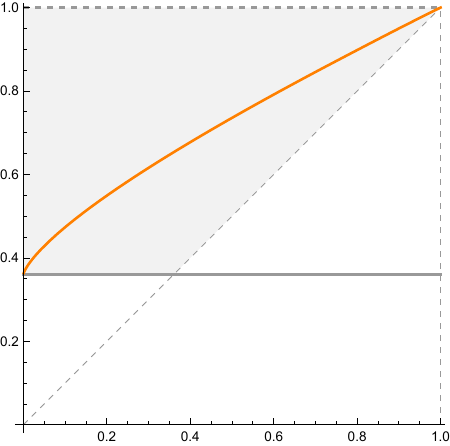}}
    \put(52,-3){\makebox(0,0)[lb]{\footnotesize$1/q$}}         
    \put(-3,52){\makebox(0,0)[lb]{\footnotesize$\beta$}}        
    \put(31,52){\makebox(0,0)[lb]{\footnotesize$\beta = \beta^{*}(q)$}}            
    \put(31,27){\makebox(0,0)[lb]{\footnotesize $\beta = 1/q$}}                
    \put(31,18){\makebox(0,0)[lb]{\footnotesize $\beta = \frac{9}{25}$}}                    
  \end{picture}    
\hspace{12pt}  
  \begin{picture}(73, 61)(-7, -1)
    \put(0,0){\includegraphics[height=6.0cm]{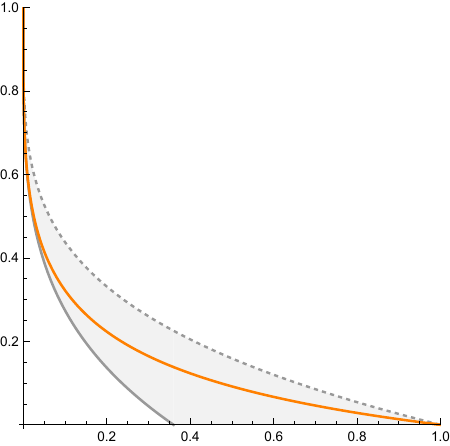}}
    \put(52,-3){\makebox(0,0)[lb]{\footnotesize$b/a$}}         
    \put(-7,50){\makebox(0,0)[lb]{\footnotesize$\cfrac{l}{\sqrt[4]{4a}}$}}             
  \end{picture}    
  }
\caption{
The graph of $\beta^{*}$ (orange curve, left) located in the gray area 
and the dependence of the length $l$ of the interval in \eqref{eq:c}
on $a$ and $b$ (orange curve, right).
Both gray curves (right) represent cases when we replace 
the function $\beta^{*}$ in \eqref{eq:c} by a constant:
$\frac{9}{25}$ (gray solid curve, right), $1$ (gray dashed curve, right).
}
\label{fig:main_2}
\end{figure}
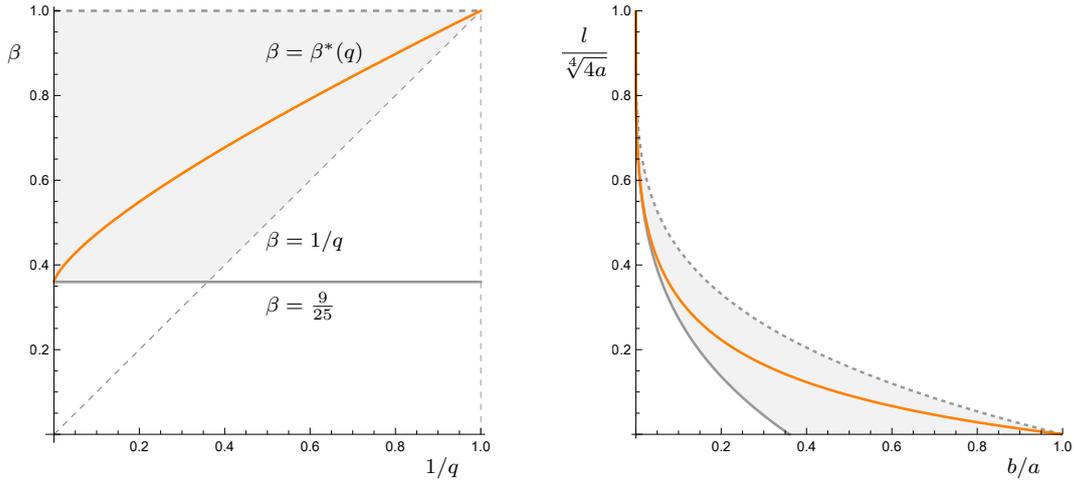

\begin{theorem}
\label{th:main2}
Let $a > b > 0$ be arbitrary but fixed and let $\beta^{\ast}$ be given by \eqref{eq:beta_sup_max}.
Then the travelling wave solution of \eqref{eq:original_problem} exists for any wave speed $c$ satisfying
\begin{equation}
\label{eq:c}
|c| \in \left( \sqrt[4]{\frac{4b}{\beta^{\ast}(a/b)}}, \sqrt[4]{4a} \right).
\end{equation}
\end{theorem}

\begin{proof}
Due to Lemma \ref{lemma:nonempty}, $\beta^{\ast}(a/b) > b/a$ and the interval in \eqref{eq:c} is nonempty.
Now, for any $c$ satisfying \eqref{eq:c}, we have
$$
\alpha = \frac{4a}{c^4} 
> 1 \quad \mbox{and} \quad 
\beta = \frac{4b}{c^4} 
< \beta^{\ast}(a/b).
$$
That is, denoting $q=a/b$, we can write
$
\left({4a}/{c^4},{4b}/{c^4}\right) = (\alpha,\beta) = (q\beta,\beta) \in \Omega^{\texttt{-}}
$
and the proved assertion follows directly from Theorem \ref{th:main1}.
\end{proof}


\section{Analytical approximation of $\beta^{\ast}$}
\label{sec:approx1}

In spite of the characterization \eqref{eq:beta_sup_max} of $\beta^{\ast}$ provided
by Proposition \ref{prop:connection}, it is not easy to compute the precise value
of $\beta^{\ast}(q)$ for a given $q > 1$, since we do not have available 
any analytical description of the \fucik~spectrum $\Sigma^{\texttt{D}}_{r}$. 
In this section, we present one possible way how to analytically 
construct lower bounds for $\beta^{\ast}$ using both its characterizations 
\eqref{eq:beta_star_orig} and \eqref{eq:beta_sup_max}. In particular, we 
construct a function $\mathcal{B} = \mathcal{B}(q)$ for $q > 1$ such that  
$\mathcal{B}(q) \le \beta^{\ast}(q)$ and $\beta^{\ast}(q) - \mathcal{B}(q)$ 
is small as much as possible.
Let us recall that for any $q > 1$, we have 
$\frac{9}{25} \le \beta^{\ast}(q) \le 1$ and $1/q < \beta^{\ast}(q)$ 
due to Lemmas \ref{lemma:9/25} and~\ref{lemma:nonempty}. Thus, we want to
achieve $\mathcal{B}(q) > \min\left\{\frac{9}{25}, 1/q\right\}$, i.e., the graph of 
$\mathcal{B}$ will be in the gray area in Figure \ref{fig:main_2}, left.

Similarly as in the proof of Lemma \ref{lemma:9/25}, 
we construct a function $v\in H^{2}(\R)$ such that
$J(\alpha,\beta,v) < 0$ with $(\alpha,\beta)$ close as much as possible
to the set $\mathscr{E}$ defined in \eqref{eq:envelope}.
Since $\mathscr{E}$ is the envelope 
of the family of all \fucik\ spectra $\Sigma^{\texttt{D}}_{r}$ of the problem \eqref{eq:fucik},
it is natural to look for inspiration directly in numerical solutions of the problem \eqref{eq:fucik}.

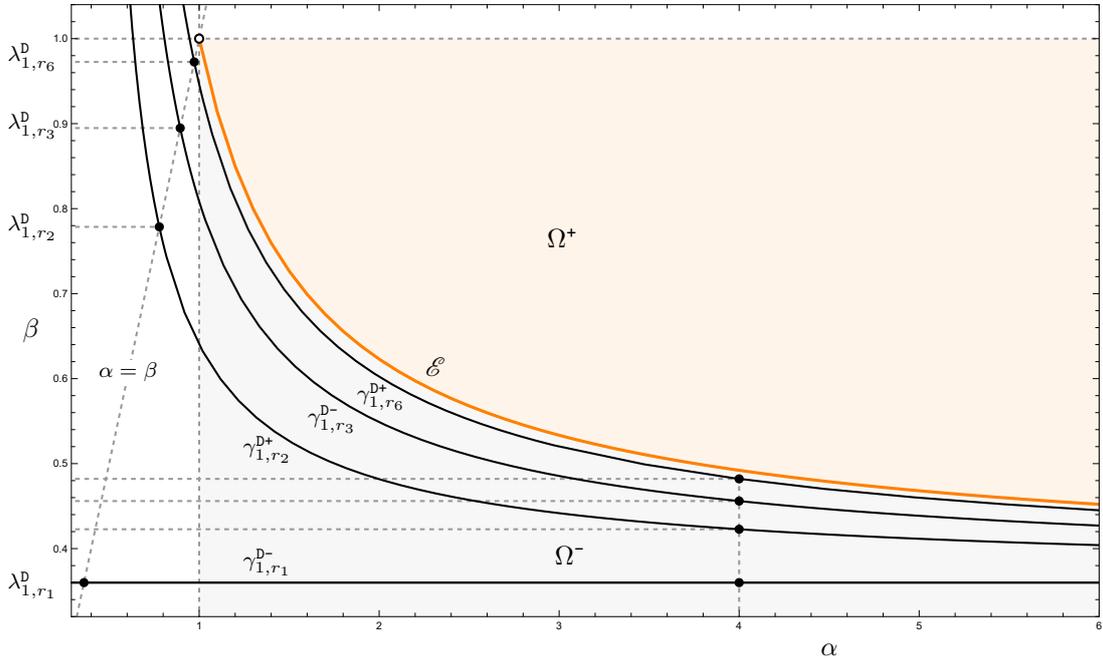
\begin{figure}[h]
\centerline{
  \setlength{\unitlength}{1mm}
  \begin{picture}(146, 89)(-6, -4)
    \put(0,0){\includegraphics[width=13.8cm]{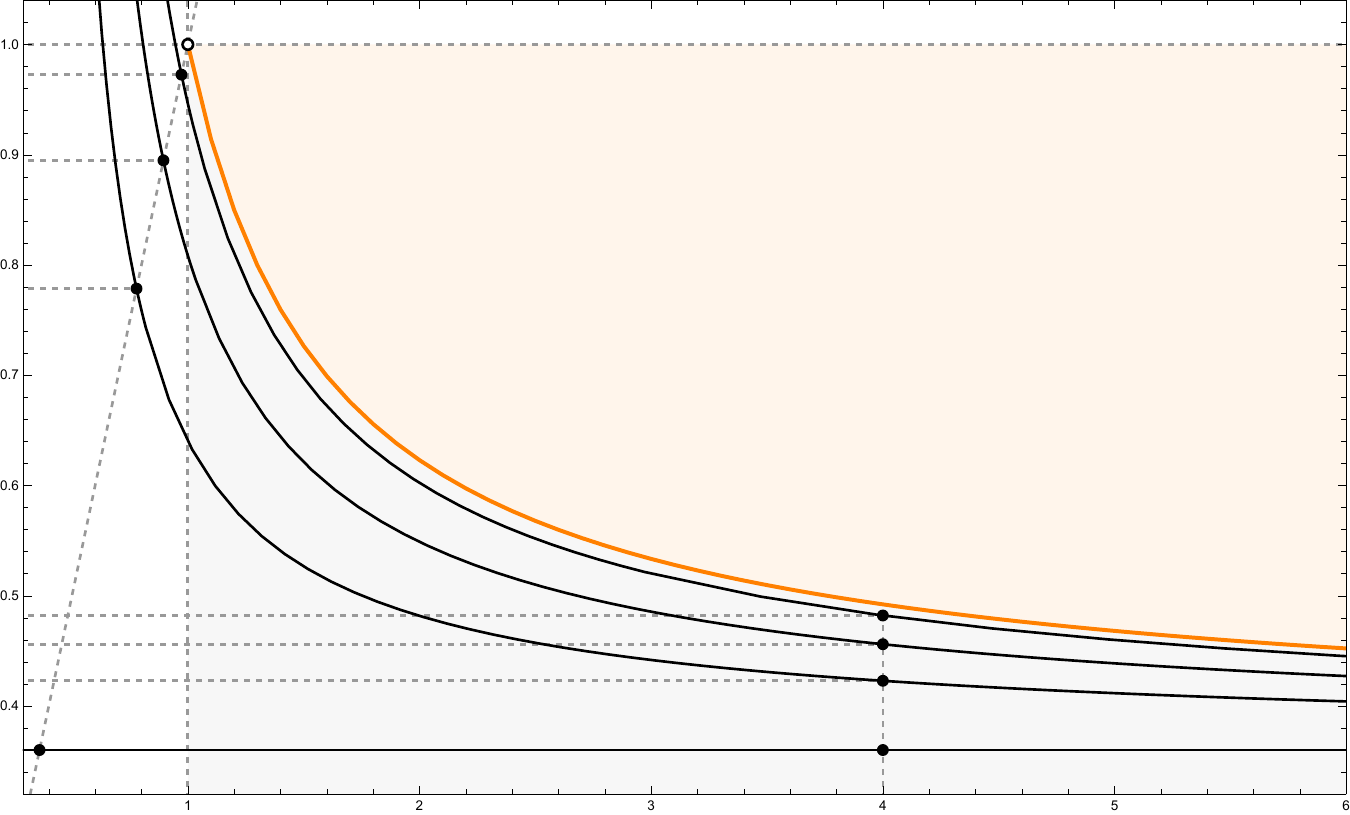}}
    \put(101,-3){\makebox(0,0)[lb]{$\alpha$}}         
    \put(-4,38.5){\makebox(0,0)[lb]{$\beta$}}             
    \put(49,34){\makebox(0,0)[lb]{$\mathscr{E}$}}        
    \put(5.3,34.0){\colorbox{White}{\makebox(7,1)[cm]{\footnotesize $\alpha = \beta$}}}            
    \put(-6,4.5){\makebox(0,0)[lb]{\footnotesize$\lambda^{\texttt{D}}_{1,r_{1}}$}}                          
    \put(-6,52.0){\makebox(0,0)[lb]{\footnotesize$\lambda^{\texttt{D}}_{1,r_{2}}$}}                          
    \put(-6,65.5){\makebox(0,0)[lb]{\footnotesize$\lambda^{\texttt{D}}_{1,r_{3}}$}}                          
    \put(-6,74.5){\makebox(0,0)[lb]{\footnotesize$\lambda^{\texttt{D}}_{1,r_{6}}$}}                                      
    \put(25.0,7.5){\makebox(0,0)[lb]{\footnotesize$\gamma^{\texttt{D}\texttt{-}}_{1,r_{1}}$}}                      
    \put(25.0,22.5){\makebox(0,0)[lb]{\footnotesize$\gamma^{\texttt{D}\texttt{+}}_{1,r_{2}}$}}                  
    \put(33.5,26.5){\makebox(0,0)[lb]{\footnotesize$\gamma^{\texttt{D}\texttt{-}}_{1,r_{3}}$}}                  
    \put(40.0,29.5){\makebox(0,0)[lb]{\footnotesize$\gamma^{\texttt{D}\texttt{+}}_{1,r_{6}}$}}                      
    \put(65,51){\makebox(3.3,2.2)[lb]{$\Omega^{\texttt{+}}$}}             
    \put(66,9){\makebox(3.3,2.2)[lb]{$\Omega^{\texttt{-}}$}}                 
  \end{picture}    
  }
\caption{
Four \fucik~curves 
$\gamma^{\texttt{D}\texttt{-}}_{1,r_{1}}$,
$\gamma^{\texttt{D}\texttt{+}}_{1,r_{2}}$,
$\gamma^{\texttt{D}\texttt{-}}_{1,r_{3}}$ and
$\gamma^{\texttt{D}\texttt{+}}_{1,r_{6}}$
that are parts of the \fucik~spectra $\Sigma^{\texttt{D}}_{r_{n}}$, $n=1,2,3,6$, 
and emanate from points 
$\big(\lambda^{\texttt{D}}_{1,r_{n}}, \lambda^{\texttt{D}}_{1,r_{n}}\big)$
on the diagonal $\alpha = \beta$,
where
$r_{1} = \frac{\pi}{2}\sqrt{5}$,
$r_{2} = \frac{\pi}{2}\sqrt{17}$,
$r_{3} = \frac{\pi}{2}\sqrt{37}$ and
$r_{6} = \frac{\pi}{2}\sqrt{145}$.}
\label{obr6}
\end{figure}

Using continuation method, we find numerically curves that belong
to $\Sigma_{r}^{\texttt{D}}$ and emanate from the diagonal $\alpha=\beta$
(see Figure \ref{obr6} for curves $\gamma_{1,r}^{\texttt{D\textpm}}$).
We proceed in the following way.
For fixed $r = r_{k} := \frac{\pi}{2}\sqrt{4k^2+1}$, $k\in\N$,
we start on the diagonal $\alpha=\beta=\lambda^{\texttt{D}}_{1,r_{k}}=(4k^{2} - 1)^2/(4k^2 + 1)^2$
with the solution of the problem \eqref{eq:fucik} 
in the form of $\pm v^{\texttt{D}}_{1,r_{k}}$, where the eigenfunction $v^{\texttt{D}}_{1,r_{k}}$ is 
given in \eqref{eq:eigenfunction} (see Figure \ref{obr19}).
Then we increase the value of $\alpha$ and consider the following initial value problem 
with $\delta\in\R$ as a parameter
\begin{equation}
\label{ivp}
\left\{ \begin{array}{l}
v^{(4)} + 2v'' + \alpha v^+ - \beta v^- = 0, \quad x \in (-r,r),\\[3pt]
v(0) = \pm 1,\
v'(0) = 0,\
v''(0) = \delta,\
v'''(0) = 0.
\end{array} \right.
\end{equation}
Now, using the shooting method, we find $\beta$ and $\delta$ such that the solution $v$ of \eqref{ivp}
satisfies also the Dirichlet boundary conditions in \eqref{eq:fucik}.
Following these steps, it is possible to find curves 
$\gamma_{1,r_{2}}^{\texttt{D+}}$, 
$\gamma_{1,r_{3}}^{\texttt{D-}}$,
$\gamma_{1,r_{6}}^{\texttt{D+}}$
as in Figure~\ref{obr6}, where $r_{2} < r_{3} < r_{6}$. 
See also Figures~\ref{obr21} and \ref{obr3}
for the corresponding nontrivial solutions 
$v_{1,r_{2}}^{\texttt{D+}}$, 
$v_{1,r_{3}}^{\texttt{D-}}$,
$v_{1,r_{6}}^{\texttt{D+}}$
of the Dirichlet problem \eqref{eq:fucik} for $\alpha = 4$.
Let us note that we have  
$v_{1,r_{k}}^{\texttt{D\textpm}}(0) = \pm 1$.
We can observe that for increasing $k$, 
the number of positive and negative semi-waves of $v_{1,r_{k}}^{\texttt{D\textpm}}$ 
on the interval $(-r_{k}, r_{k})$ increases 
and the corresponding
Fučík curves $\gamma_{1,r_{k}}^{\texttt{D\textpm}}$ are getting closer to $\mathscr{E}$ 
(see Figure \ref{obr6}).

\begin{figure}[t]
\centerline{
  \setlength{\unitlength}{1mm}
  \begin{picture}(44, 32)(0, 0)
    \put(0,0){\includegraphics[width=4.5cm]{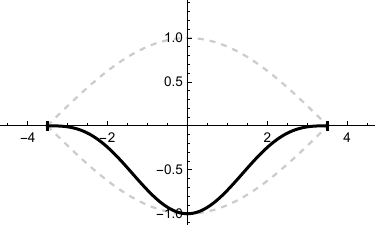}}
    \put(1,28){\makebox(0,0)[lb]{\footnotesize$(\alpha,\beta)\in \gamma_{1,r_{1}}^{\texttt{D-}}$}}    
    \put(1,24.5){\makebox(0,0)[lb]{\footnotesize$\alpha = 4$}}
    \put(1,19.0){\makebox(0,0)[lb]{\footnotesize$\beta = \tfrac{9}{25}$}}
    \put(43,9.0){\makebox(0,0)[lb]{\footnotesize$x$}}  
    \put(1.5,14.0){\makebox(0,0)[lb]{\footnotesize$-r_{1}$}}            
    \put(38.5,14.0){\makebox(0,0)[lb]{\footnotesize$r_{1}$}}           
    \put(24.0,23.5){\makebox(0,0)[lb]{\footnotesize$y = v_{1,r_{1}}^{\texttt{D-}}(x)$}}            
  \end{picture}    
  \hspace{6pt}
  \begin{picture}(44, 32)(0, 0)
    \put(0,0){\includegraphics[width=4.5cm]{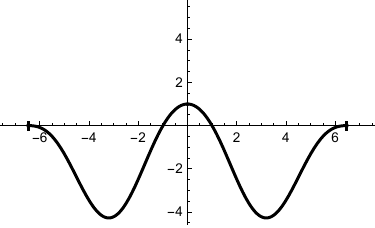}}
    \put(1,28){\makebox(0,0)[lb]{\footnotesize$(\alpha,\beta)\in \gamma_{1,r_{2}}^{\texttt{D+}}$}}
    \put(1,24.5){\makebox(0,0)[lb]{\footnotesize$\alpha = 4$}}    
    \put(1,19.5){\makebox(0,0)[lb]{\footnotesize$\beta \doteq 0.42278$}}
    \put(43,9.0){\makebox(0,0)[lb]{\footnotesize$x$}}  
    \put(0.0,14.0){\makebox(0,0)[lb]{\footnotesize$-r_{2}$}}            
    \put(40.5,14.0){\makebox(0,0)[lb]{\footnotesize$r_{2}$}}           
    \put(24.0,23.5){\makebox(0,0)[lb]{\footnotesize$y = v_{1,r_{2}}^{\texttt{D+}}(x)$}}            
  \end{picture}    
  \hspace{6pt}  
  \begin{picture}(44, 32)(0, 0)
    \put(0,0){\includegraphics[width=4.5cm]{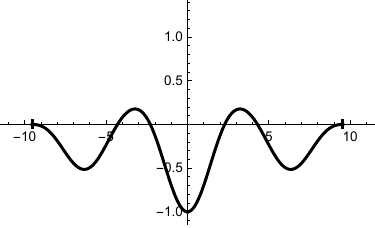}}
    \put(1,28){\makebox(0,0)[lb]{\footnotesize$(\alpha,\beta)\in \gamma_{1,r_{3}}^{\texttt{D-}}$}}
    \put(1,24.5){\makebox(0,0)[lb]{\footnotesize$\alpha = 4$}}
    \put(1,19.5){\makebox(0,0)[lb]{\footnotesize$\beta \doteq 0.45597$}}    
    \put(43,9.0){\makebox(0,0)[lb]{\footnotesize$x$}}  
    \put(0.5,14.0){\makebox(0,0)[lb]{\footnotesize$-r_{3}$}}            
    \put(40.0,14.0){\makebox(0,0)[lb]{\footnotesize$r_{3}$}}  
    \put(24.0,23.5){\makebox(0,0)[lb]{\footnotesize$y = v_{1,r_{3}}^{\texttt{D-}}(x)$}}            
  \end{picture}    
  }
\caption{The \fucik~eigenfunctions 
$v^{\texttt{D-}}_{1,r_{1}}$,
$v^{\texttt{D+}}_{1,r_{2}}$ and 
$v^{\texttt{D-}}_{1,r_{3}}$
of the Dirichlet problem \eqref{eq:fucik}
for $\alpha = 4$ and $\beta\in(0,1)$, 
where 
$r_{1}=\frac{\pi}{2}\sqrt{5}$,
$r_{2}=\frac{\pi}{2}\sqrt{17}$, 
$r_{3}=\frac{\pi}{2}\sqrt{37}$.
}
\label{obr21}
\end{figure}
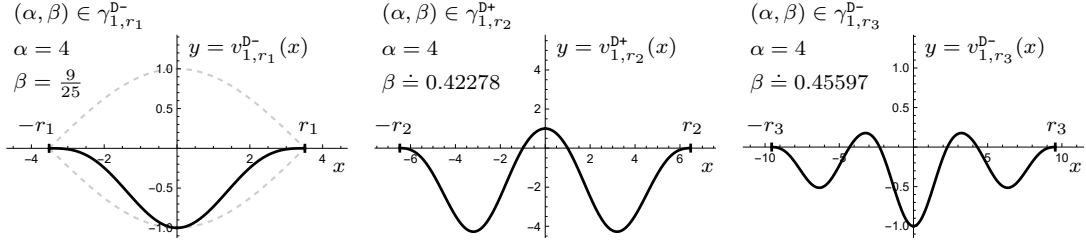

As we have stated above, our plan is to analytically construct a function $v\in H^{2}(\R)$ such that
$J(\alpha,\beta,v) < 0$ with $(\alpha,\beta)$ close to $\mathscr{E}$.
One possible way is to take $v$ consisting of damped positive and negative
semi-waves represented by fourth order polynomials 
(see Figure~\ref{obr1} for the graph of the analytically constructed function $v_{3}$ 
and compare it with the graph of the \fucik~eigenfunction $v_{1,r_{6}}^{\texttt{D+}}$
in Figure \ref{obr3} obtained numerically). 
This choice leads to the following result.

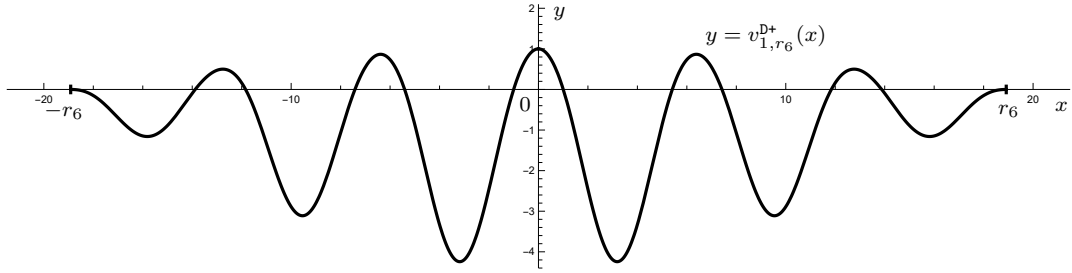
\begin{figure}[h]
\centerline{
  \setlength{\unitlength}{1mm}
\begin{minipage}{14cm}
  \begin{picture}(140, 36)(0, 0)
    \put(0,0){\includegraphics[width=14.0cm]{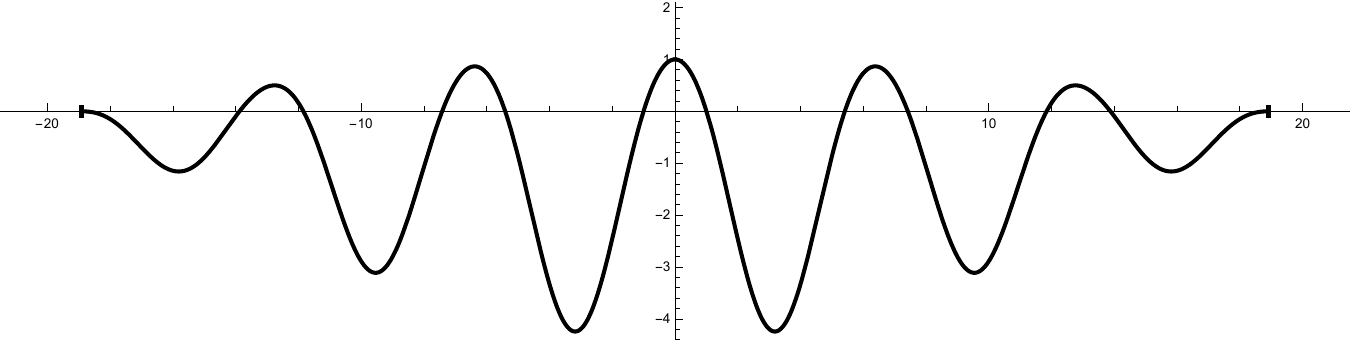}}
    \put(67.5,22){\makebox(0,0)[lm]{\footnotesize$0$}}                 
    \put(92,29){\makebox(0,0)[lb]{\footnotesize$y=v^{\texttt{D+}}_{1,r_{6}}(x)$}}     
    \put(138,21){\makebox(0,0)[lb]{\footnotesize$x$}}    
    \put(72.0,34){\makebox(0,0)[lm]{\footnotesize$y$}}      
    \put(130.5,21.0){\makebox(0,0)[lm]{\footnotesize$r_{6}$}}                     
    \put(5.0,21.0){\makebox(0,0)[lm]{\footnotesize$-r_{6}$}}         
  \end{picture}    
\end{minipage}  
  }
\caption{The \fucik~eigenfunction $v^{\texttt{D+}}_{1,r_{6}}$
of the Dirichlet problem \eqref{eq:fucik}
for $r = r_{6} = \frac{\pi}{2}\sqrt{145}$ and
$(\alpha,\beta)\in \gamma_{1,r_{6}}^{\texttt{D+}}\subset\Sigma_{r_{6}}^{\texttt{D}}$
such that $\alpha = 4$ and $\beta \doteq 0.48204$.}
\label{obr3}
\end{figure}

\begin{figure}[h]
\centerline{
  \setlength{\unitlength}{1mm}
\begin{minipage}{14cm}
  \begin{picture}(140, 36)(0, 0)
    \put(0,0){\includegraphics[width=14.0cm]{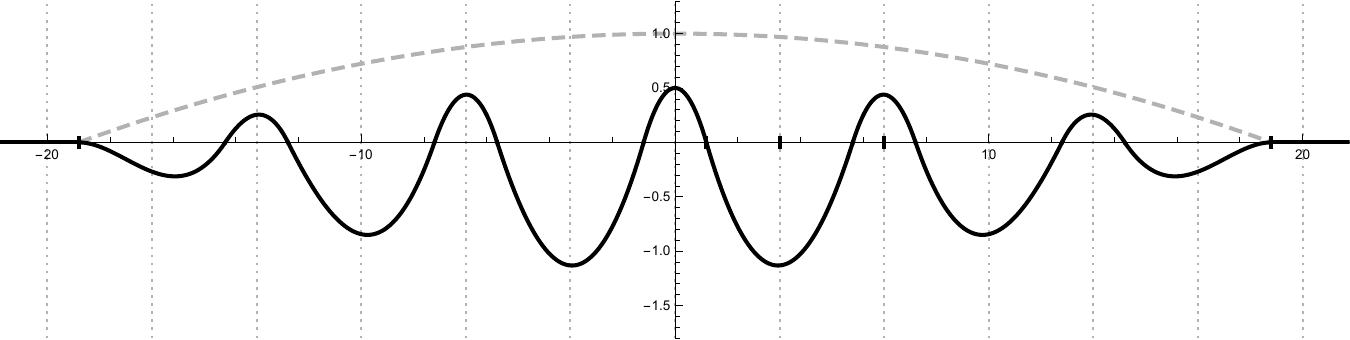}}
    \put(67.5,18.5){\makebox(0,0)[lm]{\footnotesize$0$}}             
    \put(74,23){\makebox(0,0)[lm]{\footnotesize$x_{1}$}}         
    \put(76.0,17.5){\makebox(0,0)[lm]{\footnotesize$x_{1} + x_{2}$}}             
    \put(90.8,17.5){\makebox(0,0)[lm]{\footnotesize$T$}}                 
    \put(130.5,17.5){\makebox(0,0)[lm]{\footnotesize$s_{3}$}}                     
    \put(5.0,17.5){\makebox(0,0)[lm]{\footnotesize$-s_{3}$}}                      
    \put(138.0,18.0){\makebox(0,0)[lm]{\footnotesize$x$}}                            
    \put(100,7){\makebox(0,0)[lb]{\footnotesize$y=v_{3}(x)$}}     
    \put(100,30){\makebox(0,0)[lb]{\footnotesize$y=w(x)$}}         
    \put(72.0,34){\makebox(0,0)[lm]{\footnotesize$y$}}                                
  \end{picture}    
\end{minipage}  
  }
\caption{The graph of analytically constructed function $v_{3}$ 
with 11 semi-waves on the interval $[-s_{3}, s_{3}] = [-19, 19]$
for $x_{1}=1$ and $x_{2}=\frac{7}{3}$. 
Each semi-wave is represented by a fourth order polynomial.}
\label{obr1}
\end{figure}
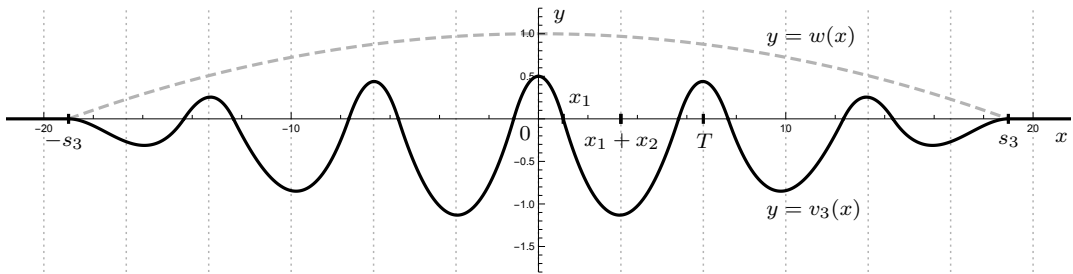

\begin{lemma} \label{veta1}
Let $x_{1}, x_{2} > 0$ be arbitrary and let $\alpha_{0}, \beta_{0} > 0$ be such that 
\begin{equation}
  \label{rce1}
  \alpha_{0} \cdot 2x_{1}^{3} + \beta_{0} \cdot 2x_{2}^{3} = 
  10(x_{1} + x_{2}) - 15\left(\tfrac{1}{x_{1}} + \tfrac{1}{x_{2}}\right).
\end{equation}
Then for all $(\alpha,\beta)\in(0,\alpha_{0}]\times(0,\beta_{0}]$,
 $(\alpha,\beta)\neq(\alpha_{0},\beta_{0})$, 
there exists $v\in H^{2}(\R)$ 
such that $J(\alpha, \beta, v) < 0$.
\end{lemma}
\begin{proof}
We find the required function $v$ as a suitable member of a sequence $(v_{n})_{n=1}^{+\infty}$ formed
by even functions $v_{n}$ with exactly $(4n-1)$ semi-waves (see Figure \ref{obr1}).
More precisely, we define $v_{n}$ as 
\begin{align}  \label{rce11}
  v_{n}(x) := \tilde{v}(x) \cdot w_{n}(x),\quad n\in\N, 
\end{align}
where $\tilde v$ is an even $T$-periodic function with the basic period $T:=2(x_{1} + x_{2}) > 0$
and $w_{n}$ is an even weight function $w_{n} = w_{n}(x) \ge 0$ that is nonzero only on
the interval $(-s_{n}, s_{n})$, where $s_{n} := n T - x_{1}$.
In particular, we consider $w_{n}$ in the following form
\begin{equation} \label{rce80}
  w_{n}(x) := 
    \left\{
  \begin{array}{lll}
    \displaystyle
    1 -\frac{x^{2}}{s_{n}^{2}}
    && \text{for } x \in (-s_{n}, s_{n}), \\[3pt]
    0
    && \text{for } |x| \ge s_{n},
  \end{array}
  \right.
\end{equation}
and $\tilde{v}$ on an interval of length $T$ as
\begin{equation} \label{rce81}
  \tilde{v}(x) := 
  \left\{
  \begin{array}{lll}
    \vpt(x)
    && \text{for } x \in [- x_{1}, x_{1}], \\[3pt]
    \vnt(x-(x_{1} + x_{2}))
    && \text{for } x \in (x_{1}, x_{1} + 2x_{2}],
  \end{array}
  \right.
\end{equation}
where the positive and the negative semi-waves $\vpt$ and $\vnt$ are given by
\begin{equation} \label{rce82}
  \vpt(x) :=  \frac{x_{1}}{2}\left(1 - \frac{x^{2}}{x_{1}^{2}}\right),\qquad
  \vnt(x) := -\frac{x_{2}}{2}\left(1 - \frac{x^{2}}{x_{2}^{2}}\right).
\end{equation}
Let us note that the even function $v_{n}$ defined in \eqref{rce11} consists of exactly
$(2n-1)$ positive and $2n$ negative semi-waves.
We show that the required function $v$ such that $J(\alpha,\beta,v) < 0 $ is given by $v_{n}$ for sufficiently large $n$.

At first, it is straightforward to verify that $v_{n} \in C^{1}(\R)$
since it consists of a finite number of polynomials that are smoothly connected
at zeros of $v_{n}$ on $[-s_{n}, s_{n}]$.
Moreover, we have $v_{n}(x) = 0$ for all $|x|\ge s_{n}$ 
and $v'_{n}(\pm s_{n}) = 0$. 
We therefore conclude that $v_{n} \in H^{2}(\R)$.

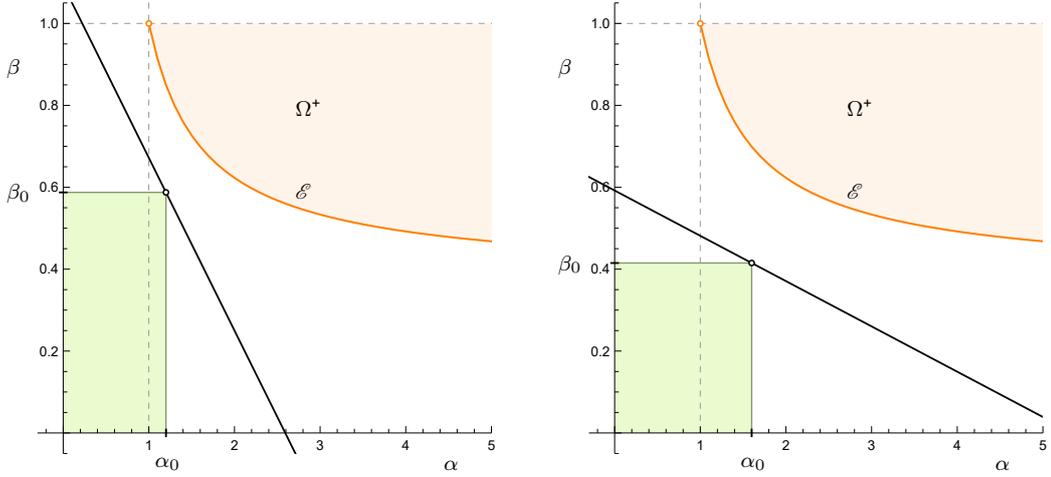
\begin{figure}[t]
\centerline{
  \setlength{\unitlength}{1mm}
  \begin{picture}(66, 62)(-7,-2)
    \put(0,0){\includegraphics[height=6.0cm]{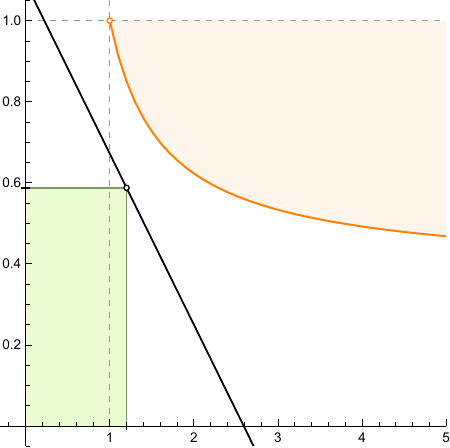}}
    \put(53.5,-2){\makebox(0,0)[lb]{\footnotesize$\alpha$}}         
    \put(-4,50){\makebox(0,0)[lb]{\footnotesize$\beta$}}        
    \put(34,45){\makebox(0,0)[lb]{\footnotesize$\Omega^{\texttt{+}}$}}            
    \put(34,34){\makebox(0,0)[lb]{\footnotesize$\mathscr{E}$}}                
    \put(15.5,-2){\makebox(0,0)[lb]{\footnotesize$\alpha_{0}$}}                     
    \put(-4,33.5){\makebox(0,0)[lb]{\footnotesize$\beta_{0}$}}      
  \end{picture}    
\hspace{12pt}  
  \begin{picture}(66, 62)(-7,-2)
    \put(0,0){\includegraphics[height=6.0cm]{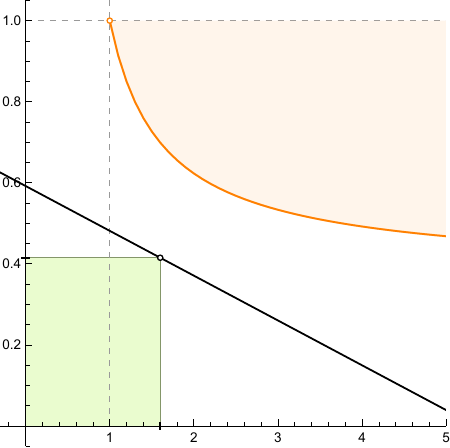}}
    \put(53.5,-2){\makebox(0,0)[lb]{\footnotesize$\alpha$}}         
    \put(-4,50){\makebox(0,0)[lb]{\footnotesize$\beta$}}        
    \put(34,45){\makebox(0,0)[lb]{\footnotesize$\Omega^{\texttt{+}}$}}            
    \put(34,34){\makebox(0,0)[lb]{\footnotesize$\mathscr{E}$}}              
    \put(20,-2){\makebox(0,0)[lb]{\footnotesize$\alpha_{0}$}}                     
    \put(-4,24){\makebox(0,0)[lb]{\footnotesize$\beta_{0}$}}                         
  \end{picture}    
  }
\caption{
Existence results provided by Lemma \ref{veta1} for 
$\alpha_{0} = 1.2$, $x_{1} = 1.5$, $x_{2} = 2$ (left)
and $\alpha_{0} = 1.6$, $x_{1} = 1.2$, $x_{2} = 2.5$ (right).
For each point $(\alpha,\beta)$ in the green rectangle, there exists 
$v\in H^{2}(\R)$ such that $J(\alpha, \beta, v) < 0$.
Note that the point $(\alpha_{0}, \beta_{0})$ is given by \eqref{rce1}
and lies on the black line
$\alpha \cdot 2x_{1}^{3} + \beta \cdot 2x_{2}^{3} = 
10(x_{1} + x_{2}) - 15\left(1/x_{1} + 1/x_{2}\right)$.
}
\label{fig:exist}
\end{figure}

At second, we claim that $J(\alpha, \beta, v_{n})$ can be expanded in terms of $n$ 
in the following way
(for the derivation see~\ref{sec:app3})
\begin{equation} \label{ref10}
  J(\alpha, \beta, v_{n}) 
  = 
  \tfrac{32}{225}
  \left(
  15\left(\tfrac{1}{x_{1}} + \tfrac{1}{x_{2}}\right) - 10\left(x_{1} + x_{2}\right)
  + \alpha\cdot 2 x_{1}^{3}
  + \beta\cdot 2 x_{2}^{3}
  \right) 
  n + \mathcal{O}(1),
\end{equation}
which can be written using \eqref{rce1} as
\begin{equation*}
  J(\alpha, \beta, v_{n}) 
  =
  n \left( 
  \tfrac{64}{225}
  \left(
  \left(\alpha - \alpha_{0}\right)\cdot x_{1}^{3}
  +
  \left(\beta - \beta_{0}\right)\cdot x_{2}^{3}
  \right) + \frac{\mathcal{O}(1)}{n}
  \right).
\end{equation*}
Finally, since $\alpha \le \alpha_{0}$, $\beta \le \beta_{0}$
and $(\alpha, \beta) \neq (\alpha_{0}, \beta_{0})$, we conclude that
$J(\alpha, \beta, v_{n}) < 0$
for sufficiently large~$n$.
\end{proof}

The statement of Lemma \ref{veta1} is illustrated in Figure~\ref{fig:exist}.
In the following theorem, we introduce a two-parametric family of functions 
$\mathcal{B}_{x_{1},x_{2}}$ 
such that each $\mathcal{B}_{x_{1},x_{2}}$ represents a lower bound for $\beta^{\ast}$, i.e.,
$\mathcal{B}_{x_{1},x_{2}}(q) \le \beta^{\ast}(q)$ for $q > 1$.
See also Figure \ref{obr55} for the graphs of $\mathcal{B}_{x_{1},x_{2}}$ 
as black dotted dashed curves
for two different settings of parameters:
$x_{1} = \frac{6}{5}$, $x_{2} = \frac{32}{15}$ (left) and 
$x_{1} = \frac{3}{2}$, $x_{2} = \frac{11}{6}$ (right).

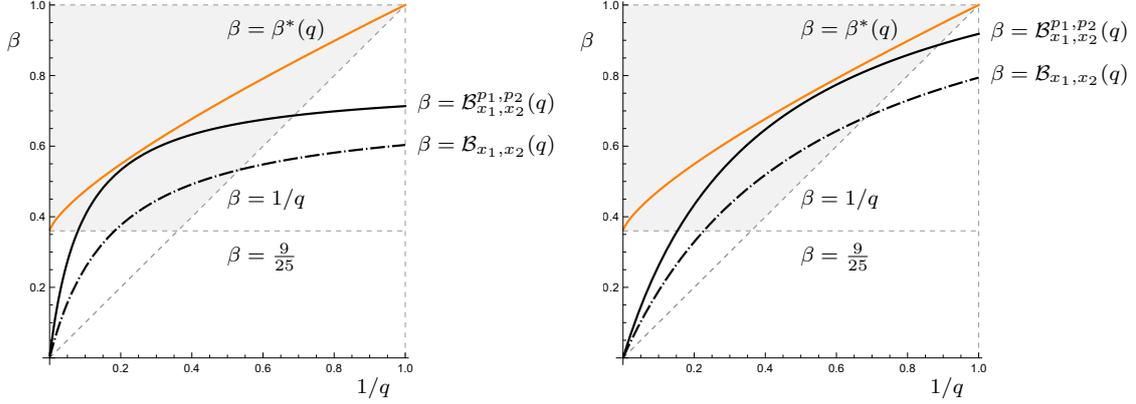
\begin{figure}[t]
\centerline{
  \setlength{\unitlength}{1mm}
  \begin{picture}(71, 53)(-3,-3)
    \put(0,0){\includegraphics[height=5.0cm]{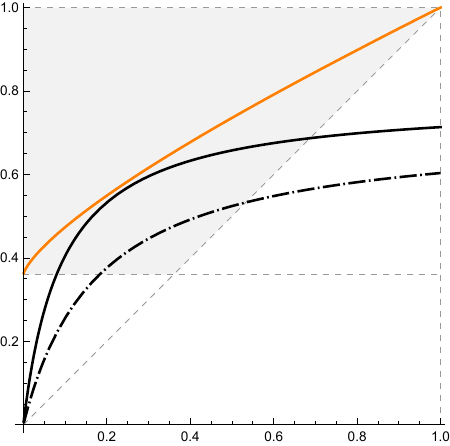}}
    \put(43,-3){\makebox(0,0)[lb]{\footnotesize$1/q$}}         
    \put(-3,43){\makebox(0,0)[lb]{\footnotesize$\beta$}}        
    \put(26,44.5){\makebox(0,0)[lb]{\footnotesize$\beta = \beta^{\ast}(q)$}}            
    \put(26,22){\makebox(0,0)[lb]{\footnotesize $\beta = 1/q$}}                
    \put(26,14){\makebox(0,0)[lb]{\footnotesize $\beta = \frac{9}{25}$}}                    
    \put(51,34.5){\makebox(0,0)[lb]{\footnotesize$\beta = \mathcal{B}_{x_{1},x_{2}}^{p_{1},p_{2}}(q)$}}                
    \put(51,29.0){\makebox(0,0)[lb]{\footnotesize$\beta = \mathcal{B}_{x_{1},x_{2}}(q)$}}                
  \end{picture}    
  \hspace{6pt}
  \begin{picture}(71, 53)(-3,-3)
    \put(0,0){\includegraphics[height=5.0cm]{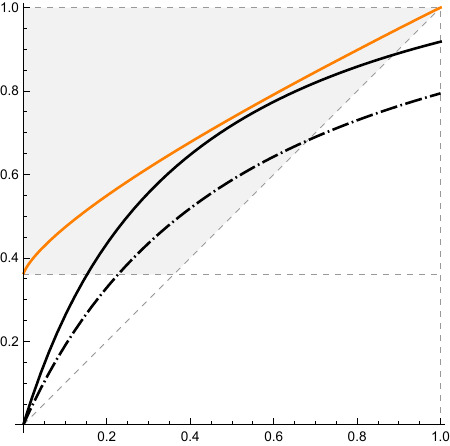}}
    \put(43,-3){\makebox(0,0)[lb]{\footnotesize$1/q$}}         
    \put(-3,43){\makebox(0,0)[lb]{\footnotesize$\beta$}}        
    \put(26,44.5){\makebox(0,0)[lb]{\footnotesize$\beta = \beta^{\ast}(q)$}}            
    \put(26,22){\makebox(0,0)[lb]{\footnotesize $\beta = 1/q$}}                
    \put(26,14){\makebox(0,0)[lb]{\footnotesize $\beta = \frac{9}{25}$}}                    
    \put(51,44.0){\makebox(0,0)[lb]{\footnotesize$\beta = \mathcal{B}_{x_{1},x_{2}}^{p_{1},p_{2}}(q)$}}                
    \put(51,38.5){\makebox(0,0)[lb]{\footnotesize$\beta = \mathcal{B}_{x_{1},x_{2}}(q)$}}                
  \end{picture}    
  }
\caption{Graphs of functions $\mathcal{B}_{x_{1},x_{2}}$ 
and $\mathcal{B}_{x_{1},x_{2}}^{p_{1},p_{2}}$
for $x_{1} = \frac{6}{5}$ (left), $x_{1} = \frac{3}{2}$ (right)
and for $x_{2} = \frac{T}{2} - x_{1}$, $T = \frac{20}{3}$ and
$p_{1} = \frac{1}{10}$, $p_{2} = \frac{3}{20}$.}
\label{obr55}
\end{figure}

\begin{theorem} \label{veta11}
Let $a > b > 0$ be arbitrary but fixed and let 
$x_{1},x_{2} > 0$ be such that 
$\mathcal{B}_{x_{1},x_{2}}(\frac{a}{b}) > \frac{b}{a}$, where
\begin{equation} 
\label{eq:bx1x2}
  \mathcal{B}_{x_{1},x_{2}}(q) :=
  \frac{10(x_{1} + x_{2})-15\left(\frac{1}{x_{1}} + \frac{1}{x_{2}}\right)}{2 x_{1}^{3}\cdot q + 2 x_{2}^{3}}, \quad q > 1.
\end{equation}
Then the travelling wave solution of \eqref{eq:original_problem} exists for any wave speed $c$ satisfying
\begin{equation} \label{eq:ccc}
|c| \in \left( \sqrt[4]{\frac{4b}{\mathcal{B}_{x_{1},x_{2}}(a/b)}}, \sqrt[4]{4a} \right).
\end{equation}
\end{theorem}
\begin{proof}
Let us define
$$
  q_{0}:= \frac{a}{b} > 1,\quad
  \beta_{0} := \mathcal{B}_{x_{1},x_{2}}(q_{0}) > \frac{b}{a} > 0, \quad
  \alpha_{0} := q_{0}\beta_{0} > 1.
$$
Then the equation \eqref{rce1} is satisfied and using Lemma \ref{veta1},
we obtain that for all $(\alpha,\beta)\in(0,\alpha_{0}]\times(0,\beta_{0}]$,
$(\alpha,\beta)\neq(\alpha_{0},\beta_{0})$, 
there exists $v\in H^{2}(\R)$ such that $J(\alpha, \beta, v) < 0$.
Thus, using \eqref{mrce1} in Remark \ref{rem:omega}, we may conclude that $\beta_{0} \le 1$.
Now, let us take
$$
  \alpha = \frac{4 a}{c^4} > 0, \quad
  \beta = \frac{4 b}{c^4} > 0,
$$
and using \eqref{eq:ccc}, we get
$$
  0 < \beta < \beta_{0} \le 1 < \alpha = q_{0}\beta < q_{0}\beta_{0} = \alpha_{0},
$$
which implies that $(\alpha,\beta)\in\Omega^{\texttt{-}}$.
Our statement then follows from Theorem \ref{th:main1}.
\end{proof}

\begin{remark}
Theorem \ref{veta11} is a weaker version of Theorem \ref{th:main2} 
in the sense that the interval in \eqref{eq:ccc} is a proper subset of the wave speed range in \eqref{eq:c}.
See Figure \ref{fig:main_3} to compare the length of intervals in \eqref{eq:ccc} and \eqref{eq:c}.
On the other hand, the assumption \eqref{eq:ccc} in Theorem \ref{veta11} is easy 
to verify compared to the optimal condition \eqref{eq:c} in Theorem \ref{th:main2}.
\end{remark}

\begin{figure}[t]
\centerline{
  \setlength{\unitlength}{1mm}
  \begin{picture}(67, 63)(-7, -3)
    \put(0,0){\includegraphics[height=6.0cm]{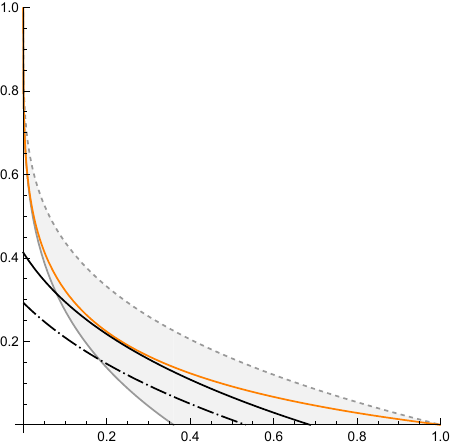}}
    \put(52,-3){\makebox(0,0)[lb]{\footnotesize$b/a$}}         
    \put(-7,50){\makebox(0,0)[lb]{\footnotesize$\cfrac{l}{\sqrt[4]{4a}}$}}             
  \end{picture}     
\hspace{12pt}  
  \begin{picture}(73, 63)(-7, -3)
    \put(0,0){\includegraphics[height=6.0cm]{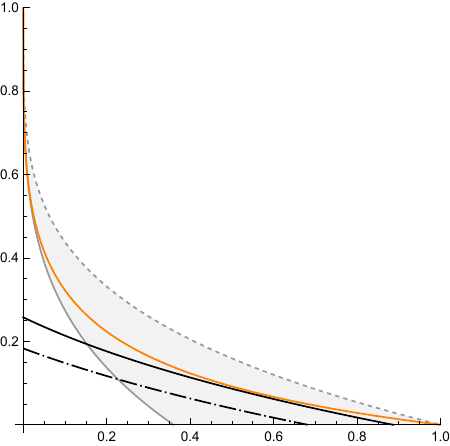}}
    \put(52,-3){\makebox(0,0)[lb]{\footnotesize$b/a$}}         
    \put(-7,50){\makebox(0,0)[lb]{\footnotesize$\cfrac{l}{\sqrt[4]{4a}}$}}             
  \end{picture}    
  }
\caption{
Dependence of the length $l$ of 
the interval in \eqref{eq:c} (orange curves),
the interval in \eqref{eq:ccc} (black dotted dashed curves)
and
the interval in \eqref{eq:cc} (black solid curves) 
on $a$ and $b$ 
in the case of the following setting of free parameters:
$x_{1} = \frac{6}{5}$ (left), $x_{1} = \frac{3}{2}$ (right)
and for $x_{2} = \frac{T}{2} - x_{1}$, $T = \frac{20}{3}$ and
$p_{1} = \frac{1}{10}$, $p_{2} = \frac{3}{20}$.
}
\label{fig:main_3}
\end{figure}
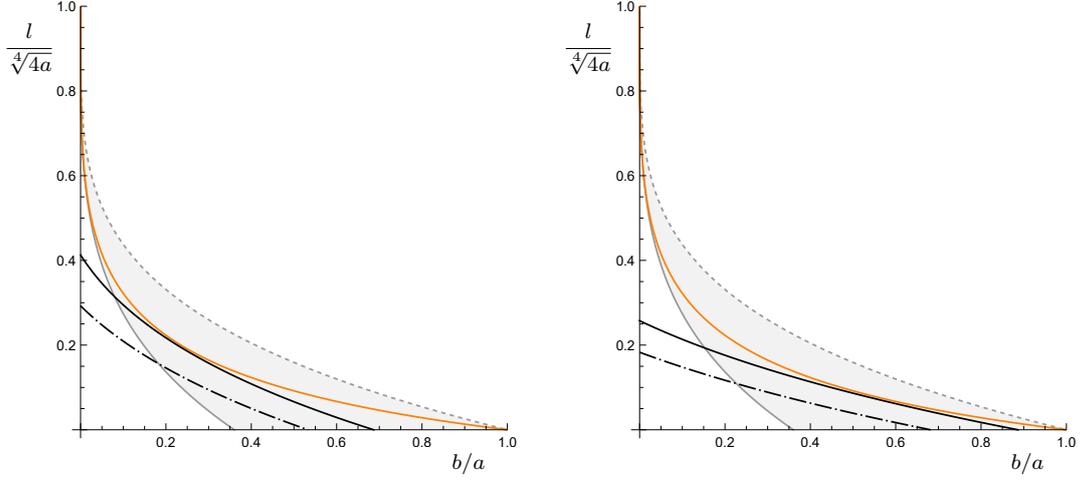

The price we pay for the simplicity of formulas \eqref{eq:bx1x2}, \eqref{eq:ccc} 
is a significant difference between $\beta^{\ast}$ and $\mathcal{B}_{x_{1},x_{2}}$ 
illustrated in Figure \ref{obr55}. 
To obtain a better analytical approximation of 
$\beta^{\ast}$ (and hence larger interval for $|c|$) requires 
a better analytical approximation of \fucik~eigenfunctions $v^{\texttt{D+}}_{1,r}$.

The following lemma extends the existence result of Lemma \ref{veta1} 
using a more general description of positive and negative semi-waves represented
by rational functions with two free parameters $p_{1}$ and $p_{2}$
(see Figure \ref{obr11} for the graph of the analytically constructed function 
$\breve{v}_{3}$ and compare it with the graph of the \fucik~eigenfunction $v_{1,r_{6}}^{\texttt{D+}}$
in Figure \ref{obr3} obtained numerically).
Let us note that this construction is strongly inspired 
by Padé approximations of trigonometric functions (see~\cite{Baker_1996}). 

\begin{lemma} \label{veta2}
Let $x_{1},x_{2},p_{1},p_{2} > 0$ be arbitrary and let 
$\alpha_{0}, \beta_{0} > 0$ be such that 
\begin{equation}
  \label{rce2}
  G(p_{1},x_{1})\cdot\alpha_{0} + 
  G(p_{2},x_{2})\cdot\beta_{0} = 
  2 M(p_{1},x_{1}) + 
  2 M(p_{2},x_{2}) - 
  N(p_{1},x_{1}) - 
  N(p_{2},x_{2}),
\end{equation}
where functions $G$, $M$ and $N$ are defined as
\begin{align*}
  G(p,x)
  & := 
  \frac{\left(p x^{2} + 1\right)^{2}
  \left(
  \sqrt{p}x\left(p x^{2} + 3\right) + 
  \left(p x^{2} - 3\right)\left(p x^{2} + 1\right)\arctan\left(\sqrt{p}x\right)
  \right)  
  }{4\sqrt{p^{5}} x^{2}},
\\[6pt]
  M(p,x)
  & := 
  \frac{\left(p x^{2} + 1\right)
  \left(
  \sqrt{p}x\left(p x^{2} + 3\right)\left(3p x^{2} - 1\right) + 
  3\left(p x^{2} + 1\right)^{3}\arctan\left(\sqrt{p}x\right)
  \right)  
  }{24 \sqrt{p^{3}} x^{2}},
\\[6pt]
  N(p,x)
  & := 
  \frac{
  \sqrt{p}x\left(15 p^{4} x^{8} + 70 p^{3} x^{6} + 128 p^{2} x^{4} + 10 p x^{2} + 65 \right) 
  + 15\left(p x^{2} + 1\right)^{5}\arctan\left(\sqrt{p}x\right) 
  }{40 \sqrt{p} x^{2}\left(p x^{2} + 1\right)}.
\end{align*}
Then for all $(\alpha,\beta)\in(0,\alpha_{0}]\times(0,\beta_{0}]$,
 $(\alpha,\beta)\neq(\alpha_{0},\beta_{0})$, 
there exists $v\in H^{2}(\R)$ 
such that $J(\alpha, \beta, v) < 0$.
\end{lemma}

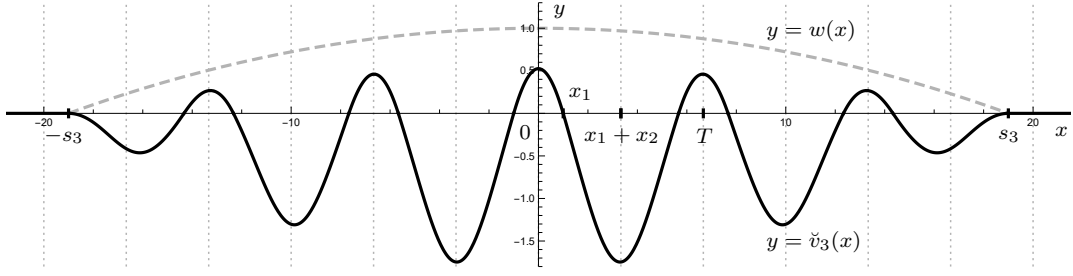
\begin{figure}[h]
\centerline{
  \setlength{\unitlength}{1mm}
\begin{minipage}{14cm}
  \begin{picture}(140, 35)(0, 0)
    \put(0,0){\includegraphics[width=14.0cm]{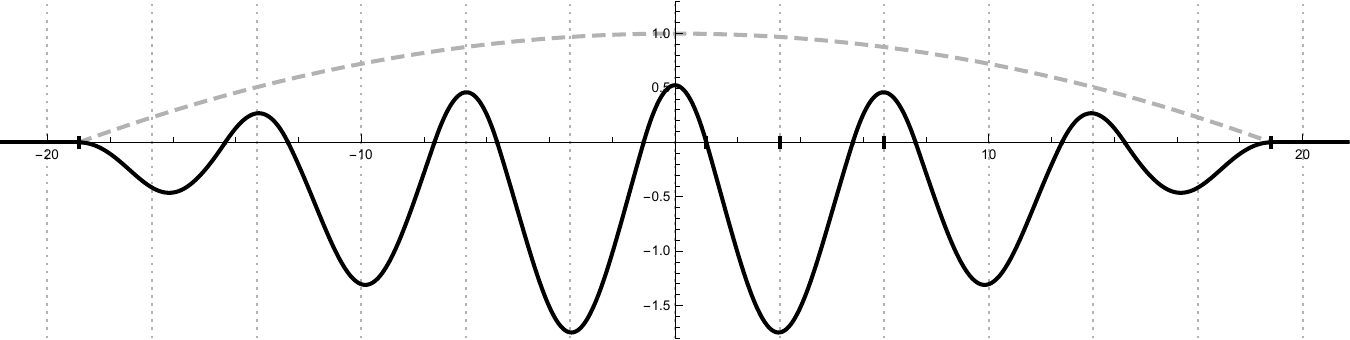}}
    \put(67.5,18){\makebox(0,0)[lm]{\footnotesize$0$}}             
    \put(74,23){\makebox(0,0)[lm]{\footnotesize$x_{1}$}}         
    \put(76.0,17.5){\makebox(0,0)[lm]{\footnotesize$x_{1} + x_{2}$}}             
    \put(90.8,17.5){\makebox(0,0)[lm]{\footnotesize$T$}}                 
    \put(130.5,17.5){\makebox(0,0)[lm]{\footnotesize$s_{3}$}}                     
    \put(5.0,17.5){\makebox(0,0)[lm]{\footnotesize$-s_{3}$}}                      
    \put(138.0,18.0){\makebox(0,0)[lm]{\footnotesize$x$}}                            
    \put(100,2){\makebox(0,0)[lb]{\footnotesize$y=\breve{v}_{3}(x)$}}     
    \put(100,30){\makebox(0,0)[lb]{\footnotesize$y=w(x)$}}         
    \put(72.0,34){\makebox(0,0)[lm]{\footnotesize$y$}}    
  \end{picture}    
\end{minipage}  
  }
\caption{The graph of 
analytically constructed function $\breve{v}_{3}$ with 11 semi-waves on 
the interval $[-s_{3}, s_{3}] = [-19, 19]$
for $x_{1}=1$, $x_{2}=\frac{7}{3}$ and 
$p_{1} = \frac{1}{20}$, $p_{2} = \frac{1}{10}$.
Each semi-wave is represented by a rational function of order $(4,2)$.}
\label{obr11}
\end{figure}

\begin{proof}
We proceed similarly as in the proof of Lemma \ref{veta1}. The main difference is that
now we consider the $T$-periodic function $\tilde{v}$ given by \eqref{rce81} 
with a more general form of positive and negative semi-waves
\begin{equation} \label{rce90}
  \vpt(x) =  \frac{x_{1}}{2}\left(1 - \frac{x^{2}}{x_{1}^{2}}\right)\frac{1 + p_{1} x_{1}^{2}}{1 + p_{1} x^{2}},\qquad
  \vnt(x) = -\frac{x_{2}}{2}\left(1 - \frac{x^{2}}{x_{2}^{2}}\right)\frac{1 + p_{2} x_{2}^{2}}{1 + p_{2} x^{2}},
\end{equation}
where $p_{1},p_{2} \ge 0$ are new parameters.
Let us note that for $p_{1} = p_{2} = 0$, functions $\vpt$ and $\vnt$ 
in \eqref{rce90} reduce to the form in \eqref{rce82}. 
For this reason, in what follows, we take $p_{1},p_{2} > 0$.
 
Now, let us consider the sequence $(\breve{v}_{n})_{n=1}^{+\infty}$, 
where $\breve{v}_{n}(x) = \tilde{v}(x)\cdot w_{n}(x)$ and the weight function $w_{n}$ 
is given by \eqref{rce80}.
If we proceed similarly as in \ref{sec:app3}, we obtain expansions of the following
integrals with respect to $n$
\begin{align*}
\int_{\breve{v}_{n} > 0} \left(\breve{v}_{n}\right)^{2}\,\dd x & = \tfrac{16}{15}G(p_{1},x_{1}) n + \mathcal{O}(1), &
\int_{\breve{v}_{n} < 0} \left(\breve{v}_{n}\right)^{2}\,\dd x & = \tfrac{16}{15}G(p_{2},x_{2}) n + \mathcal{O}(1), \\
\int_{\breve{v}_{n} > 0} \left(\breve{v}_{n}'\right)^{2}\,\dd x & = \tfrac{16}{15}M(p_{1},x_{1}) n + \mathcal{O}(1), &
\int_{\breve{v}_{n} < 0} \left(\breve{v}_{n}'\right)^{2}\,\dd x & = \tfrac{16}{15}M(p_{2},x_{2}) n + \mathcal{O}(1), \\
\int_{\breve{v}_{n} > 0} \left(\breve{v}_{n}''\right)^{2}\,\dd x & = \tfrac{16}{15}N(p_{1},x_{1}) n + \mathcal{O}(1), &
\int_{\breve{v}_{n} < 0} \left(\breve{v}_{n}''\right)^{2}\,\dd x & = \tfrac{16}{15}N(p_{2},x_{2}) n + \mathcal{O}(1),
\end{align*}
which, using \eqref{rce2}, leads to
\begin{align*} 
  J(\alpha, \beta, v_{n}) 
  = &\
  \tfrac{16}{15}
  \big(
  N(p_{1},x_{1}) + N(p_{2},x_{2}) - 2 M(p_{1},x_{1}) - 2 M(p_{2},x_{2}) + \\
  &\
  \alpha\cdot G(p_{1},x_{1}) + \beta\cdot G(p_{2},x_{2})
  \big)\cdot n + \mathcal{O}(1), \\
  = &\
  n\left(
  \tfrac{16}{15}
  \left(
  (\alpha - \alpha_{0})\cdot G(p_{1},x_{1})+
  (\beta - \beta_{0})\cdot G(p_{2},x_{2})
  \right)
  + \frac{\mathcal{O}(1)}{n}\right).
\end{align*}
Finally, it remains to show that $G(p,x) > 0$ for $p,x > 0$.
Since we have 
$$
  G(p,x) = \frac{(p x^{2} + 1)^{2}}{4 \sqrt{p^{5}}x^{2}}\cdot g\left(\sqrt{p} x\right),
$$
where $g(u) := u(u^{2} + 3) + (u^{2} - 3)(u^{2} + 1)\arctan u$, it is enough to justify that
$g(u) > 0$ for $u > 0$. Indeed, we have $g(0)=0$ and it is straightforward to verify that $g'(u) > 0$ for $u>0$.
Hence, we conclude that $J(\alpha, \beta, v_{n}) <0$ for sufficiently large $n$.
\end{proof}

The following theorem provides a significant improvement of 
the existence result in Theorem~\ref{veta11} in terms of
the length of the interval in \eqref{eq:ccc}. 
It uses a four-parametric family of functions $\mathcal{B}_{x_{1},x_{2}}^{p_{1},p_{2}}$,
each of these functions can be used as a lower bound for $\beta^{\ast}$, 
i.e., $\mathcal{B}_{x_{1},x_{2}}^{p_{1},p_{2}}(q) \le \beta^{\ast}(q)$ for $q > 1$.
See Figure \ref{obr55} for the graphs of $\mathcal{B}_{x_{1},x_{2}}^{p_{1},p_{2}}$ 
as black solid curves for two different settings of parameters
and notice their position with respect to graphs of $\beta^{\ast}$ and $\mathcal{B}_{x_{1},x_{2}}$. 
See also Figure \ref{fig:main_3} to compare the length of the corresponding intervals for $|c|$.

\begin{theorem}
\label{veta12}
Let $a > b > 0$ be arbitrary but fixed and let 
$x_{1},x_{2},p_{1},p_{2} > 0$ be such that 
$\mathcal{B}_{x_{1},x_{2}}^{p_{1},p_{2}}(\frac{a}{b}) > \frac{b}{a}$, where
\begin{equation} 
  \mathcal{B}_{x_{1},x_{2}}^{p_{1},p_{2}}(q) :=
  \frac{2 M(p_{1},x_{1}) + 2 M(p_{2},x_{2}) - N(p_{1},x_{1}) - N(p_{2},x_{2})}{G(p_{1},x_{1})\cdot q + G(p_{2},x_{2})}, \quad q > 1,
\end{equation}
with $M$, $N$ and $G$ being functions given in Lemma \ref{veta2}.
Then the travelling wave solution of \eqref{eq:original_problem} exists for any wave speed $c$ satisfying
\begin{equation} \label{eq:cc}
|c| \in \left( \sqrt[4]{\frac{4b}{\mathcal{B}_{x_{1},x_{2}}^{p_{1},p_{2}}(a/b)}}, \sqrt[4]{4a} \right).
\end{equation}
\end{theorem}
\begin{proof}
The proof is based on Lemma \ref{veta2} and is an analogue of the proof of Theorem \ref{veta11}
with $\mathcal{B}_{x_{1},x_{2}}^{p_{1},p_{2}}$ instead of $\mathcal{B}_{x_{1},x_{2}}$.
\end{proof}

\begin{remark}~
Let us note that for functions $G$, $M$ and $N$ 
in Lemma \ref{veta2}, we have that
\begin{equation} \label{rce88}
  \lim_{p\to 0+} G(p,x) = \tfrac{4 x^{3}}{15},\qquad
  \lim_{p\to 0+} M(p,x) = \tfrac{2x}{3},\qquad
  \lim_{p\to 0+} N(p,x) = \tfrac{2}{x}.
\end{equation}
Thus, the relation \eqref{rce2} reduces to \eqref{rce1} for $p_{1} \to 0+$ and $p_{2} \to 0+$
and moreover, for $x_{1},x_{2} > 0$ and $q>1$, we have
\begin{align*}
  \lim_{\substack{p_{1} \to 0+\\ p_{2} \to 0+}} \mathcal{B}_{x_{1},x_{2}}^{p_{1},p_{2}}(q) = \mathcal{B}_{x_{1},x_{2}}(q).
\end{align*}
\end{remark}


\section{Analytical approximations of $\mathscr{E}$}
\label{sec:approx2}

The previous Section \ref{sec:approx1} provided two different families of subintervals
in \eqref{eq:ccc} and \eqref{eq:cc} 
of the theoretical optimal interval in \eqref{eq:c} for the wave speed $c$. 
To describe unions of these families, it is more convenient to express them as approximations 
of the set $\Omega^{\texttt{-}}$ in $\alpha\beta$ plane. 
Consequently, we reverse the point of view, fix the wave speed $c$ and provide a characterization
of original parameters $a$ and $b$ for which the problem \eqref{eq:original_problem} possesses 
a travelling wave solution.

We start with  
the set of lines \eqref{rce1} introduced in Lemma \ref{veta1} containing
two parameters $x_{1}$ and $x_{2}$, and remove them step by step by constructing suitable envelopes
(see Figure~\ref{obr5} for curves $\mu_{T}$ and $\mu$).

\begin{figure}[h]
\centerline{
  \setlength{\unitlength}{1mm}
  \begin{picture}(140, 85)(0, 0)
    \put(0,0){\includegraphics[width=14.0cm]{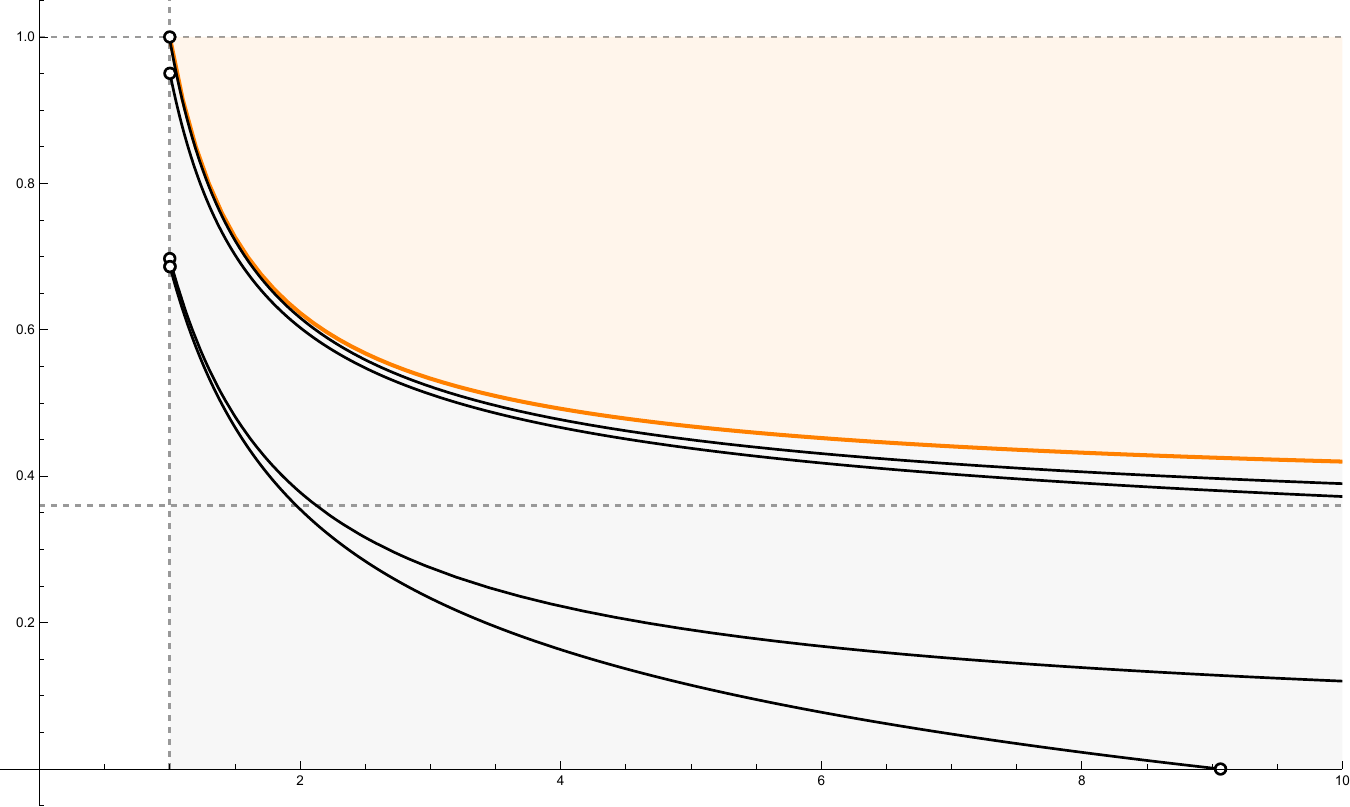}}
    \put(118.5,-1){\makebox(0,0)[lb]{$\alpha$}}         
    \put(-2,70){\makebox(0,0)[lb]{$\beta$}}             
    \put(-2,29.1){\makebox(0,0)[lb]{$\frac{9}{25}$}}                 
    \put(48,15){\makebox(0,0)[lb]{$\mu_{T}$}}     
    \put(48,25){\makebox(0,0)[lb]{$\mu$}}         
    \put(48,35){\makebox(0,0)[lb]{$\eta_{T}^{p_{1}, p_{2}}$}}     
    \put(141.0,32.3){\makebox(0,0)[lb]{$\eta$}}             
    \put(80,56){\makebox(0,0)[lb]{$\Omega^{\texttt{+}}$}}             
    \put(80,23){\makebox(0,0)[lb]{$\Omega^{\texttt{-}}$}}            
    \put(80,40){\makebox(0,0)[lb]{$\mathscr{E}$}}              
  \end{picture}    
  }
\caption{Comparison of the existence results provided 
by Lemmas \ref{col1}, \ref{col3} and \ref{col2} for $T = \frac{20}{3}$,
$p_{1} = \frac{1}{10}$ and $p_{2} = \frac{3}{20}$.}
\label{obr5}
\end{figure}
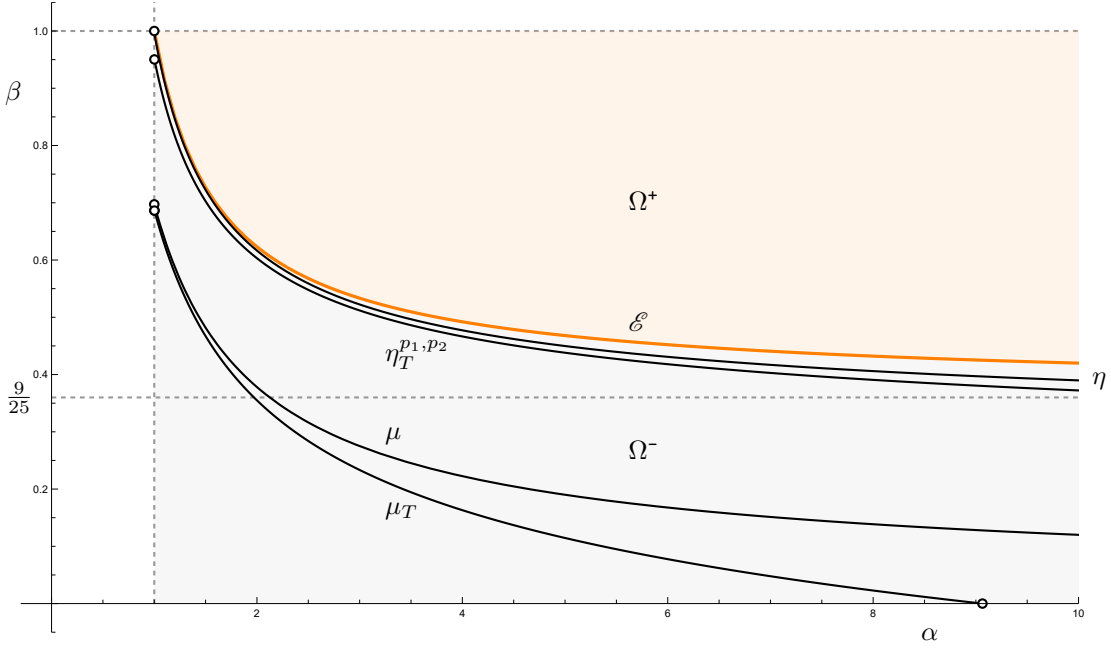

\begin{lemma} \label{col1}
Let $T>0$ be fixed and let $(\alpha_{0},\beta_{0})\in(0,+\infty)\times(0,+\infty)$ 
be an arbitrary point of the curve $\mu_{T} = \mu_{T}(x_{1})$
given by
\begin{equation} \label{rce55}
  \mu_{T} : 
  \left\{
  \begin{array}{l}
  \alpha = Q\big(x_{1}, \tfrac{T}{2} - x_{1}\big), \\[3pt]
  \beta  = Q\big(\tfrac{T}{2} - x_{1}, x_{1}\big),   
  \end{array}
  \right.
  \quad
  0 < x_{1} < \tfrac{T}{2},\qquad
  Q(x_{1}, x_{2}) := \cfrac{10 x_{1}^{2} x_{2} - 20 x_{1} + 5 x_{2}}{2 x_{1}^{4} x_{2}}.
\end{equation}
Then for all $(\alpha,\beta)\in(0,\alpha_{0}]\times(0,\beta_{0}]$,
 $(\alpha,\beta)\neq(\alpha_{0},\beta_{0})$, 
there exists $v\in H^{2}(\R)$ 
such that $J(\alpha, \beta, v) < 0$.
\end{lemma}
\begin{proof}
Let us consider the family of lines given by \eqref{rce1}
as a set of curves in $\alpha\beta$ plane with one parameter $x_{1}\in\left(0, \frac{T}{2}\right)$ 
and one condition $x_{2} = \frac{T}{2} - x_{1}$.
Let us denote by 
 $\mu_{T}$  the envelope of this set of lines,
i.e., $\mu_{T}$ consists of all points $(\alpha, \beta)$ for which
\begin{align} \label{rce12}
  F\big(\alpha, \beta, x_{1}, \tfrac{T}{2} - x_{1}\big) = 0,\qquad
  \tfrac{\partial F}{\partial x_{1}}\big(\alpha, \beta, x_{1}, \tfrac{T}{2} - x_{1}\big) - 
  \tfrac{\partial F}{\partial x_{2}}\big(\alpha, \beta, x_{1}, \tfrac{T}{2} - x_{1}\big) = 0,  
\end{align}
where $F(\alpha, \beta, x_{1}, x_{2}) :=
\alpha \cdot 2x_{1}^{3} + \beta \cdot 2x_{2}^{3} - 
10(x_{1} + x_{2}) + 15\left(1/x_{1} + 1/x_{2}\right)$.  
The system \eqref{rce12} can be directly solved for ($\alpha, \beta)$ 
and results in the parametrization \eqref{rce55} of $\mu_{T}$.
\end{proof}

\begin{lemma} \label{col3}
Let $\alpha_{0}, \beta_{0} > 0$ be such that 
\begin{align}
  \label{rce77}
  \beta_{0} = \frac{675\, \alpha_{0} + 125 + \sqrt{125(18\,\alpha_{0} + 5)^{3}}}{2916\, \alpha_{0}^{2}}.
\end{align}
Then for all $(\alpha,\beta)\in(0,\alpha_{0}]\times(0,\beta_{0}]$,
$(\alpha,\beta)\neq(\alpha_{0},\beta_{0})$, 
there exists $v\in H^{2}(\R)$ 
such that $J(\alpha, \beta, v) < 0$.
\end{lemma}
\begin{proof}
Let us consider the set of curves $\mu_{T}$ 
in $\alpha\beta$ plane with one parameter $T>0$,
introduced in \eqref{rce55} in Lemma \ref{col1}.
It is straightforward to check that the envelope of this set is
the following curve
\begin{equation} \label{krivka}
  \mu : 
  \left\{
  \begin{array}{l}
  \alpha = Q\left(x_{1}, \frac{3}{x_{1}}\right), \\[6pt]
  \beta  = Q\left(\frac{3}{x_{1}}, x_{1}\right),   
  \end{array}
  \right.
  \quad
  x_{1} > 0,
\end{equation}
where $Q$ is the rational function defined in \eqref{rce55}.
Indeed, the curve $\mu$ is tangent to all curves $\mu_{T}$ since
for $x_{1} > 0$ and $T = \frac{6}{x_{1}} + 2 x_{1}$, we have that 
\begin{align*}
\mu(x_{1}) = \mu_{T}(x_{1}),\quad
\mu'(x_{1}) = \tfrac{9 + 3 x_{1}^{2}}{9 + x_{1}^{4}}\cdot \mu_{T}'(x_{1}).
\end{align*}
Finally, for $\alpha_{0}, \beta_{0} > 0$, the relation \eqref{rce77} is equivalent to 
$(\alpha_{0}, \beta_{0}) = \mu(x_{1})$ with 
some $x_{1} > 0$.
\end{proof}


Lemma \ref{col3} enables us to state the following alternative existence result 
for the original problem~\eqref{eq:original_problem}.
For the illustration, see Figure \ref{fig:main_222}.

\begin{theorem} \label{veta22}
Let $c\neq 0$ be arbitrary but fixed.
Then the travelling wave solution of \eqref{eq:original_problem}
with the wave speed $c$ exists for any $(a,b)$ satisfying
\begin{equation} \label{podm8}
   a > \frac{c^{4}}{4},\quad
   0 < b < \frac{c^{4}}{4} \cdot 
   P\left(\frac{4a}{c^{4}}\right)
\end{equation}
with $P$ given by
\begin{equation} \label{funkceF}
P(\alpha) := \frac{675 \, \alpha + 125 + \sqrt{125 \, (18 \, \alpha + 5)^{3}}}{2916 \, \alpha^{2}}.
\end{equation}
\end{theorem}

\begin{proof}
Let us denote $\alpha_{0} := 4a/c^{4}$ and $\beta_{0} := P(\alpha_{0})$.
Then conditions in \eqref{podm8} imply that 
\begin{equation} \label{rce99}
  0 < \frac{4 b}{c^{4}} < \beta_{0} \quad \text{and} \quad 1 < \frac{4a}{c^{4}} = \alpha_{0}
\end{equation}
and using Lemma \ref{col3}, we get that  for all $(\alpha,\beta)\in(0,\alpha_{0}]\times(0,\beta_{0}]$,
$(\alpha,\beta)\neq(\alpha_{0},\beta_{0})$, 
there exists $v\in H^{2}(\R)$ such that $J(\alpha_{0},\beta, v) < 0$.
Thus, using \eqref{mrce1} in Remark \ref{rem:omega}, we may conclude that $\beta_{0} \le 1$.
Now, let us take $\beta := 4b/c^{4}$ and using \eqref{rce99}, we have that
$0 < \beta < \beta_{0} \le 1 < \alpha_{0}$,
which implies that $(\alpha_{0},\beta)\in\Omega^{\texttt{-}}$.
Our statement then follows from Theorem \ref{th:main1}.
\end{proof}


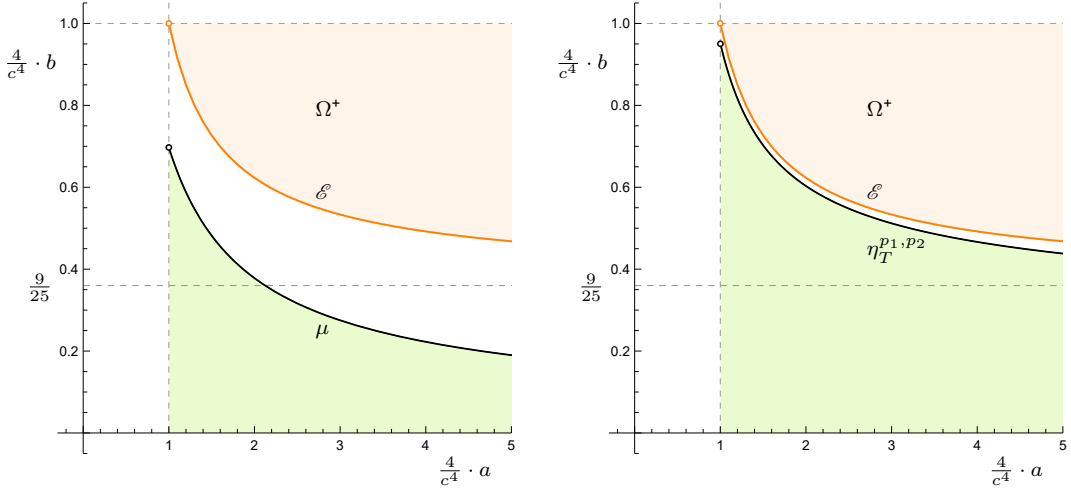
\begin{figure}[t]
\centerline{
  \setlength{\unitlength}{1mm}
  \begin{picture}(66, 66)(-7,-6)
    \put(0,0){\includegraphics[height=6.0cm]{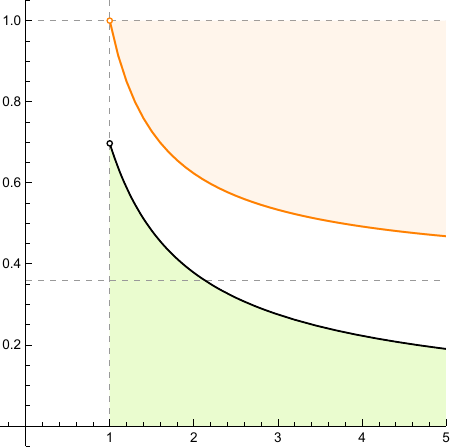}}
    \put(50,-4){\makebox(0,0)[lb]{\footnotesize$\tfrac{4}{c^{4}}\cdot a$}}         
    \put(-7,50){\makebox(0,0)[lb]{\footnotesize$\tfrac{4}{c^{4}}\cdot b$}}        
    \put(34,45){\makebox(0,0)[lb]{\footnotesize$\Omega^{\texttt{+}}$}}            
    \put(34,34){\makebox(0,0)[lb]{\footnotesize$\mathscr{E}$}}                
    \put(34,15.5){\makebox(0,0)[lb]{\footnotesize$\mu$}}                    
    \put(-4,20.5){\makebox(0,0)[lb]{\footnotesize$\frac{9}{25}$}}                    
  \end{picture}    
\hspace{12pt}  
  \begin{picture}(66, 66)(-7,-6)
    \put(0,0){\includegraphics[height=6.0cm]{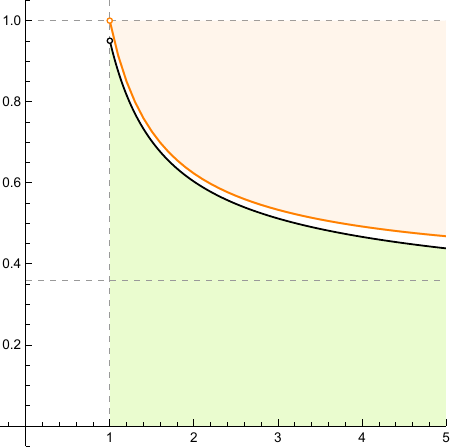}}
    \put(50,-4){\makebox(0,0)[lb]{\footnotesize$\tfrac{4}{c^{4}}\cdot a$}}         
    \put(-7,50){\makebox(0,0)[lb]{\footnotesize$\tfrac{4}{c^{4}}\cdot b$}}        
    \put(34,45){\makebox(0,0)[lb]{\footnotesize$\Omega^{\texttt{+}}$}}            
    \put(34,34){\makebox(0,0)[lb]{\footnotesize$\mathscr{E}$}}              
    \put(34,25.5){\makebox(0,0)[lb]{\footnotesize$\eta_{T}^{p_{1},p_{2}}$}}                          
    \put(-4,20.5){\makebox(0,0)[lb]{\footnotesize$\frac{9}{25}$}}                    
  \end{picture}    
  }
\caption{
Comparison of the existence results provided 
by Theorem \ref{veta22} (left) and Theorem \ref{veta33} (right) 
for $T = \frac{20}{3}$,
$p_{1} = \frac{1}{10}$ and $p_{2} = \frac{3}{20}$.
For a pair $\left(4a/c^{4}, 4b/c^{4}\right)$ in the green area, 
the existence of the travelling wave solution of \eqref{eq:original_problem} is guaranteed.
}
\label{fig:main_222}
\end{figure}

\begin{remark}~
\begin{enumerate}
\item
Let us note that in $\Omega$, the graph of the function $P$ defined in \eqref{funkceF}  
and the curve $\mu$ given by \eqref{krivka} coincide (see Figure \ref{fig:main_222}, left)
and can be interpreted as a lower approximation of $\mathscr{E}$, which is rough but simple to evaluate.
\item
The graph of the function $P$ can be described using the function 
$\mathcal{B}_{x_{1},x_{2}}$ defined by \eqref{eq:bx1x2} in Theorem \ref{veta11}.
For this purpose, let us consider the curve $\mu$ with the parametrization given by
\eqref{krivka} with $x_{1} = \sqrt{3 s}$, $s > 0$, and define the function
$\mathcal{Q}: s \mapsto \alpha/\beta$, where $(\alpha,\beta) = \mu\big(\sqrt{3s}\big)$,
i.e.,
\begin{equation} \label{pos1}
  \mathcal{Q}(s) :=  
  \frac{Q\left(\sqrt{3s}, \frac{3}{\sqrt{3s}}\right)}{Q\left(\frac{3}{\sqrt{3s}}, \sqrt{3s}\right)} = 
  \frac{2s + 1}{s^{3}(s + 2)},
  \quad s > 0.
\end{equation}
The function $\mathcal{Q}$ is strictly decreasing and maps $(0,+\infty)$ onto $(0,+\infty)$.
Thus, for any $q > 0$, there exists exactly one $s > 0$ such that $\mathcal{Q}(s) = q$.
Now, for fixed $s > 0$, it is straightforward to verify that
$\mathcal{B}_{x_{1},x_{2}}(\mathcal{Q}(s))$ attains its maximum 
for $x_{1} = \sqrt{3s}$ and $x_{2}=\frac{3}{\sqrt{3s}}$.
Finally, if we define
\begin{equation} \label{pos2}
  \mathcal{B}^{*}(s) :=
  \sup_{x_{1},x_{2} > 0}\mathcal{B}_{x_{1},x_{2}}(\mathcal{Q}(s)) = 
  \mathcal{B}_{\sqrt{3s},\frac{3}{\sqrt{3s}}}(\mathcal{Q}(s)) = \tfrac{5}{18}s(s + 2),\quad s > 0,
\end{equation}
then we may conclude that 
the graph of the function $P = P(\alpha)$ for $\alpha > 0$ is the set
$$\left\{\left( \mathcal{Q}(s)\mathcal{B}^{*}(s), \mathcal{B}^{*}(s) \right): s > 0\right\}.$$
Indeed, it is easy to check that $P(\mathcal{Q}(s)\mathcal{B}^{*}(s)) = \mathcal{B}^{*}(s)$ for all $s > 0$.
\item Using functions $\mathcal{Q}$ and $\mathcal{B}^{*}$ defined in 
\eqref{pos1} and \eqref{pos2}, respectively,
it is possible to characterize the union of the family of intervals in \eqref{eq:ccc}
in Theorem \ref{veta11} as
\begin{align*}
  \bigcup_{x_{1},x_{2} > 0}
  \left( \sqrt[4]{\frac{4b}{\mathcal{B}_{x_{1},x_{2}}(a/b)}}, \sqrt[4]{4a} \right)
  =
  \left( \sqrt[4]{\frac{4b}{\mathcal{B}^{*}\left(\mathcal{Q}^{-1}(a/b)\right)}}, \sqrt[4]{4a} \right)
  \subset
  \left( \sqrt[4]{\frac{4b}{\beta^{\ast}(a/b)}}, \sqrt[4]{4a} \right).
\end{align*}
\end{enumerate}
\end{remark}
Similarly as in Section \ref{sec:approx1}, a finer result in Lemma \ref{veta2} enables us to obtain a better approximation of $\mathscr{E}$. 
In particular, the following assertion describes the envelope of the family of lines given by \eqref{rce2}
(see Figure~\ref{obr5} for the curve $\eta_{T}^{p_{1},p_{2}}$).

\begin{lemma} \label{col2}
Let $T,p_{1},p_{2} > 0$ be fixed and let 
$(\alpha_{0},\beta_{0})\in(0,+\infty)\times(0,+\infty)$ 
be a point of the curve 
$\eta_{T}^{p_{1},p_{2}} = \eta_{T}^{p_{1},p_{2}}(x_1)$
given by
\begin{equation*}
  \eta_{T}^{p_{1},p_{2}} : 
  \left\{
  \begin{array}{l}
  \alpha = R^{p_{1},p_{2}}\big(x_{1}, \tfrac{T}{2} - x_{1}\big), \\[3pt]
  \beta  = R^{p_{2},p_{1}}\big(\tfrac{T}{2} - x_{1}, x_{1}\big),   
  \end{array}
  \right.
  \quad
  0 < x_{1} < \tfrac{T}{2},
\end{equation*}
where
\begin{align*}
  R^{p_{1},p_{2}}(x_{1}, x_{2}) := & 
  \frac{G_{x}(p_{2},x_{2})
  \left(2M(p_{1},x_{1}) + 2M(p_{2},x_{2}) - N(p_{1},x_{1}) - N(p_{2},x_{2})\right)}
  {G(p_{1},x_{1})G_{x}(p_{2},x_{2}) + G_{x}(p_{1},x_{1}) G(p_{2},x_{2})} + \\
  &
  \frac{G(p_{2},x_{2}) 
  \left(
  2M_{x}(p_{1},x_{1}) - 
  2M_{x}(p_{2},x_{2}) -
   N_{x}(p_{1},x_{1}) +
   N_{x}(p_{2},x_{2})\right)}  
  {G(p_{1},x_{1})G_{x}(p_{2},x_{2}) + G_{x}(p_{1},x_{1}) G(p_{2},x_{2})},
\end{align*}
functions $G$, $M$ and $N$ are given in Theorem \ref{veta2} 
and $G_{x}$, $M_{x}$ and $N_{x}$ denote their first-order 
partial derivatives with respect to the second variable $x$.

Then for all $(\alpha,\beta)\in(0,\alpha_{0}]\times(0,\beta_{0}]$,
$(\alpha,\beta)\neq(\alpha_{0},\beta_{0})$, 
there exists $v\in H^{2}(\R)$ 
such that $J(\alpha, \beta, v) < 0$.
\end{lemma}
\begin{proof}
We proceed similarly as in the proof of Lemma~\ref{col1}.
That is, we construct the envelope $\eta_{T}^{p_{1},p_{2}}$ of the set of lines given by \eqref{rce2} 
by solving the system \eqref{rce12} for $(\alpha,\beta)$ with $F$ given by
$$
F(\alpha, \beta, x_{1}, x_{2}) :=
\alpha\cdot G(p_{1},x_{1}) + \beta\cdot G(p_{2},x_{2})
- 2 M(p_{1},x_{1}) - 2 M(p_{2},x_{2}) + N(p_{1},x_{1}) + N(p_{2},x_{2}).
$$ 
\end{proof}

Finally, the following theorem substantially improves the existence 
result in Theorem \ref{veta22} (see Figure~\ref{fig:main_222}).

\begin{theorem} \label{veta33}
Let $T,p_{1},p_{2} > 0$ and let $c\neq 0$ be arbitrary but fixed.
Then the travelling wave solution of \eqref{eq:original_problem}
with the wave speed $c$ exists for any $(a,b)$ satisfying
\begin{equation}
   \frac{c^{4}}{4} < a = \frac{c^{4}}{4}\cdot R^{p_{1},p_{2}}\big(x_{1}, \tfrac{T}{2} - x_{1}\big),\quad
   0 < b < \frac{c^{4}}{4} \cdot R^{p_{2},p_{1}}\big(\tfrac{T}{2} - x_{1}, x_{1}\big),
\end{equation}
where $x_{1}\in\left(0,\frac{T}{2}\right)$ and 
the function $R^{p_{1},p_{2}}$ is given in Lemma \ref{col2}.
\end{theorem}
\begin{proof}
The proof is based on Lemma \ref{col2} and is analogous to the proof of Theorem \ref{veta22},
where we take 
$\alpha_{0} := R^{p_{1},p_{2}}\big(x_{1}, \frac{T}{2} - x_{1}\big)$
and
$\beta_{0} := R^{p_{2},p_{1}}\big(\frac{T}{2} - x_{1}, x_{1}\big)$.
\end{proof}

\begin{remark}
For $x_{1},x_{2} > 0$, we have that
\begin{align*}
  \lim_{\substack{p_{1} \to 0+\\ p_{2} \to 0+}} R^{p_{1},p_{2}}(x_{1},x_{2}) = Q(x_{1},x_{2})
\end{align*}
and thus, for fixed $T > 0$ and $0 < x_{1} < \frac{T}{2}$, we get that the point 
$\eta_{T}^{p_{1},p_{2}}(x_{1})$ converges to $\mu_{T}(x_{1})$ for 
$p_{1} \to 0+$ and 
$p_{2} \to 0+$.
\end{remark}

\begin{remark}
Let us consider lines given by \eqref{rce2} with $x_{2} = \frac{T}{2} - x_{1}$
as a set of curves in $\alpha\beta$~plane with four parameters $x_{1}, p_{1}, p_{2}, T$ represented by
$F(\alpha, \beta, x_{1}, p_{1}, p_{2}, T) = 0$,
where
\begin{align*}
  F(\alpha, \beta, x_{1}, p_{1}, p_{2}, T) :=\ & 
  \alpha\cdot G(p_{1}, x_{1}) + 
  \beta\cdot G\big(p_{2}, \tfrac{T}{2} - x_{1}\big) - \\
  & 2 M(p_{1}, x_{1}) - 2 M\big(p_{2}, \tfrac{T}{2} - x_{1}\big) + 
  N(p_{1}, x_{1}) + N\big(p_{2}, \tfrac{T}{2} - x_{1}\big).
\end{align*}
The envelope $\eta$ of this four-parameter set of lines is given by
\begin{align} \label{rce91}
  F(\alpha, \beta, x_{1}, p_{1}, p_{2}, T) = 0,\quad
  \tfrac{\partial F}{\partial x_{1}} = 0,\quad
  \tfrac{\partial F}{\partial p_{1}} = 0,\quad
  \tfrac{\partial F}{\partial p_{2}} = 0,\quad
  \tfrac{\partial F}{\partial T} = 0,         
\end{align}
where all partial derivatives are evaluated at $(\alpha, \beta, x_{1}, p_{1}, p_{2}, T)$.
For fixed $x_{1} > 0$, the system \eqref{rce91} can be solved numerically to get
the approximation of one point $(\alpha,\beta)$ of the envelope $\eta$ 
(see~Figure~\ref{obr5}).
Let us point out that any point $(\alpha_{0}, \beta_{0})$ of the envelope $\eta$ 
is also the point of a line in \eqref{rce2} and thus, using Lemma \ref{veta2},
for all $\beta \in(0, \beta_{0})$ there exists $v\in H^{2}(\R)$ such that 
$J(\alpha_{0}, \beta, v) < 0$. Thus, we may conclude that $\eta$ is a lower approximation of $\mathscr{E}$.
\end{remark}


\section{Characterization of $\beta^{\ast}$ using the periodic problem}
\label{sec:conjecture}

As we have stated above, the exact characterization \eqref{eq:beta_sup_max} of $\beta^{\ast}$ 
(and hence $\mathscr{E}$) is not suitable for computation, since 
it is not easy to obtain $\Sigma^{\texttt{D}}_r$ either. 
An alternative and much more straightforward approach to the description of $\beta^{\ast}$
is inspired by the properties of the eigenvalues $\lambda^{\texttt{D}}_{n,r}$ of the Dirichlet problem \eqref{eq:eigen}
and the eigenvalues $\lambda^{\texttt{P}}_{n,T}$ of the periodic problem
\begin{equation}
\label{eq:eigen_periodic}
\left\{ \begin{array}{l}
v^{(4)} + 2v'' + \lambda v = 0, \quad x \in \R,\\[3pt]
v(x + T) = v(x).
\end{array} \right.
\end{equation}
Namely, the following statement holds true (cf. \ref{sec:eigen}).
\begin{lemma}
\label{lemma:eigen_sup}
We have
$$
\lim_{q\to 1+} \beta^{\ast}(q) = \sup_{r>0} \max_{n\in\N} \lambda^\textup{\texttt{D}}_{n,r} = \lim_{r \to +\infty} \lambda^\textup{\texttt{D}}_{1,r} = 1 = \sup_{T>0} \max_{n \in \N_0} \lambda^\textup{\texttt{P}}_{n,T} = \lambda^\textup{\texttt{P}}_{1,2\pi}.
$$
Similarly, for the (even) eigenfunction $v^{\mathtt{D}}_{1,r}$ 
corresponding to $\lambda^\textup{\texttt{D}}_{1,r}$, we have
$\lim_{r \to +\infty} v^{\mathtt{D}}_{1,r}(x) = \cos x $
which coincides with the eigenfunction 
corresponding to $\lambda^\textup{\texttt{P}}_{1,2\pi}$.
\end{lemma}
That is, suprema of all Dirichlet and periodic eigenvalues coincide, moreover, in the case of the periodic problem, the supremum is attained. Besides that, the periodic eigenvalues are easier to compute. Description of both the eigenvalues and their properties are stated in \ref{sec:eigen}.

Motivated by this fact,
let us now consider the periodic problem
\begin{equation}
\label{eq:fucik_per}
\left\{ \begin{array}{l}
v^{(4)} + 2v'' + \alpha v^+ - \beta v^- = 0, \quad x\in \R,\\[3pt]
v(x + T) = v(x),
\end{array} \right.
\end{equation}
with $T>0$ and denote by $\Sigma^{\texttt{P}}_T$ the corresponding \fucik\ spectrum of \eqref{eq:fucik_per}, i.e.,
$$
\Sigma^{\texttt{P}}_T := \left\{ (\alpha,\beta) \in \R^2: ~ \mbox{\eqref{eq:fucik_per} has a nontrivial solution} \right\}.
$$
Unlike $\Sigma^{\texttt{D}}_r$, 
the \fucik~spectrum $\Sigma^{\texttt{P}}_T$ is easier to obtain. For the full analytic description of its first component see \ref{sec:fucik_per}. Denoting
\begin{equation}
\label{eq:beta_conjecture1}
\hat\beta=\hat\beta(q) := \sup\limits_{T>0} \sup \left\{ \beta > 0: (q \beta, \beta)\in\Sigma_{T}^\textup{\texttt{P}} \right\},
\end{equation}
we state the following.

\begin{conjecture}
\label{con:periodic1}
Let $\Sigma^\textup{\texttt{P}}_T$ be the \fucik\ spectrum of \eqref{eq:fucik_per} with a period $T > 0$ and $\hat\beta$ be given by \eqref{eq:beta_conjecture1}. Then for any $q > 1$, $\beta^{\ast}(q) = \hat\beta(q)$. Moreover, both suprema in \eqref{eq:beta_conjecture1} are attained, namely
\begin{equation}
\label{eq:beta_conjecture2}
\beta^{\ast}(q) = \hat\beta(q)
= \max_{2\pi\leq T \leq \sqrt{5}\pi} \max \left\{ \beta > 0: \, (q\beta,\beta) \in \Sigma^\textup{\texttt{P}}_T \right \}. 
\end{equation}
\end{conjecture}

\begin{figure}[h]
\centerline{
  \setlength{\unitlength}{1mm}
  \begin{picture}(60, 60)(0,0)
    \put(0,0){\includegraphics[height=6.0cm]{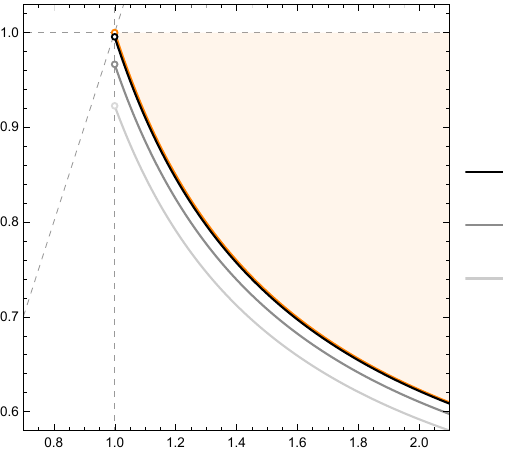}}
    \put(50.5,-2){\makebox(0,0)[lb]{\footnotesize$\alpha$}}         
    \put(-3,50.0){\makebox(0,0)[lb]{\footnotesize$\beta$}}        
    \put(68,36.0){\makebox(0,0)[lb]{\scriptsize$T = T_{1}$}}                  
    \put(68,29.0){\makebox(0,0)[lb]{\scriptsize$T = T_{2}$}}                  
    \put(68,22.0){\makebox(0,0)[lb]{\scriptsize$T = T_{3}$}}   
    \put(38,38){\makebox(0,0)[lb]{\footnotesize$\hat{\Omega}^{\texttt{+}}$}}            
    \put(38,21){\makebox(0,0)[lb]{\footnotesize$\mathscr{\hat{E}}$}}                                
  \end{picture}    
  \hspace{1.8cm}
  \begin{picture}(60, 60)(0,0)
    \put(0,0){\includegraphics[height=6.0cm]{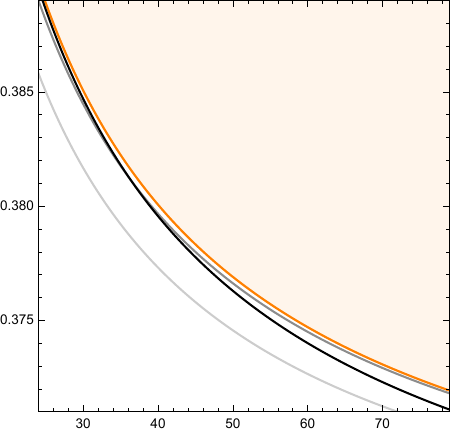}}
    \put(47,-2){\makebox(0,0)[lb]{\footnotesize$\alpha$}}         
    \put(-3,50.0){\makebox(0,0)[lb]{\footnotesize$\beta$}}             
    \put(38,38){\makebox(0,0)[lb]{\footnotesize$\hat{\Omega}^{\texttt{+}}$}}            
    \put(38,19){\makebox(0,0)[lb]{\footnotesize$\mathscr{\hat{E}}$}}                                    
  \end{picture}    
  }
\caption{The first \fucik~curve of $\Sigma_{T}^{\texttt{P}}$ (see \ref{sec:fucik_per})
for three different values of $T$
($T_{1} = \sqrt{4.2}\pi$, $T_{2} = \sqrt{4.6}\pi$ and $T_{3} = \sqrt{5}\pi$)
close to the diagonal $\alpha=\beta$ (left) 
and slightly further off the diagonal (right). 
The orange curve corresponds to the envelope $\mathscr{\hat{E}}$.
}
\label{obr16}
\end{figure}
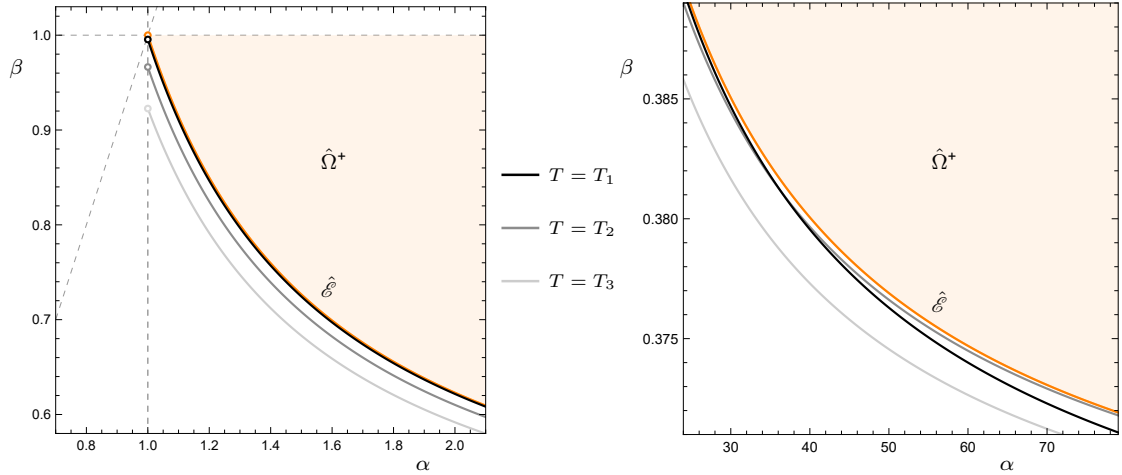

Unfortunately, all the statements of Conjecture \ref{con:periodic1} remain in the form of hypotheses and we are not able to prove or disprove them at the moment. Their probable validity is supported by numerical experiments. If we denote
\begin{align*}
\mathscr{\hat{E}} := \left\{(q\hat\beta(q), \hat\beta(q)): q > 1\right\},
\end{align*}
and
$$
\hat{\Omega}^{\texttt{-}} := \left\{(q\beta, \beta)\in\Omega: q > 1,\ \beta < \hat{\beta}(q)\right\},\quad
\hat{\Omega}^{\texttt{+}} := \left\{(q\beta, \beta)\in\Omega: q > 1,\ \beta > \hat{\beta}(q)\right\},
$$
then the set $\mathscr{\hat{E}}$ is the envelope of all {\fucik} curves of 
$\Sigma_{T}^{\texttt{P}}$ for all $T > 0$ and Conjecture \ref{con:periodic1} states that $\mathscr{E}=\mathscr{\hat{E}}$ and $\Omega^{\texttt{-}} = \hat{\Omega}^{\texttt{-}}$, $\Omega^{\texttt{+}} = \hat{\Omega}^{\texttt{+}}$. 

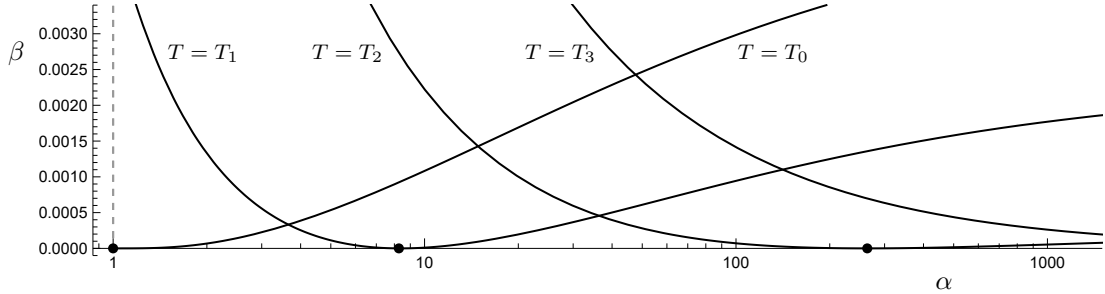
\begin{figure}[h]
\centerline{
  \setlength{\unitlength}{1mm}
  \begin{picture}(148, 50)(-6, -4)
    \put(0,0){\includegraphics[width=14.0cm]{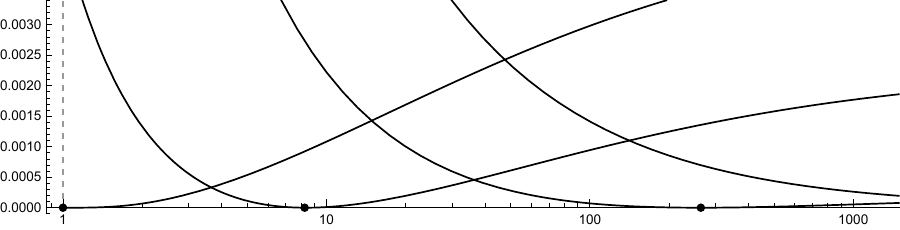}}
    \put(118,-2){\makebox(0,0)[lb]{$\alpha$}}         
    \put(-4,27.6){\makebox(0,0)[lb]{$\beta$}}             
    \put(92,28){\makebox(0,0)[lb]{\footnotesize$T = T_{0}$}}             
    \put(17,28){\makebox(0,0)[lb]{\footnotesize$T = T_{1}$}}             
    \put(36,28){\makebox(0,0)[lb]{\footnotesize$T = T_{2}$}}             
    \put(64,28){\makebox(0,0)[lb]{\footnotesize$T = T_{3}$}}                         
  \end{picture}    
  }
\caption{
Difference of 
the curve $\mathscr{\hat{E}}$ and 
the first \fucik~curve of $\Sigma_{T}^{\texttt{P}}$
for four different values of $T$ 
($T_{0}=2\pi$, 
$T_{1} = \sqrt{4.2}\pi$, 
$T_{2} = \sqrt{4.6}\pi$,
$T_{3} = \sqrt{5}\pi$)
in case they are parametrized as 
$(\alpha, \beta(\alpha))$ 
and we plot the difference of $\beta$ coordinates with respect to~$\alpha$.
}
\label{obr17}
\end{figure}

Figure \ref{obr16} shows analyticaly described parts of $\Sigma_T^{\texttt{P}}$ 
for several values $T$ (cf. \ref{sec:fucik_per}) and 
the envelope $\mathscr{\hat{E}}$ (in orange color) which coincides with numerically obtained envelope $\mathscr{E}$. 
Notice that we can observe mutual intersections of periodic \fucik\ spectra 
with different periods $T$, which means 
that---unlike the eigenvalue case---we cannot say that the ``highest'' one corresponds to the period $T=2\pi$.
Difference of $\mathscr{\hat{E}}$ and particular \fucik\ curves is visualized in Figure \ref{obr17}.


\section{Open problems}
\label{sec:final}

Although we have provided an exact theoretical description of the optimal travelling 
wave speed range (cf. Theorem \ref{th:main2}) as well as families of its approximations 
which are easy to evaluate (cf. Theorems \ref{veta11} and \ref{veta12}) together 
with their envelopes (cf. Theorems \ref{veta22} and \ref{veta33}), several questions still remain open.

First of all, our methods are based on variational approach, 
namely Mountain Pass Theorem, and Theorems \ref{th:main1} and \ref{th:main2} provide 
the maximal possible wave speed range for which the mountain pass geometry is preserved. 
So the following natural questions arise.

\begin{description}
\item{$Q_1$:} \emph{Is it possible to enlarge the interval for admissible values of the wave speed up to $\left(\sqrt[4]{4b},\sqrt[4]{4a}\right)$ by using some other methods? Or are there no travelling wave solutions with the wave speed close enough to $\sqrt[4]{4b}$?}
\end{description}

Moreover, obtained results can be interpreted so that for given $0<b<a$, there exist infinitely many travelling wave solutions of \eqref{eq:original_problem} with different wave speeds from the interval given in \eqref{eq:c}. Another open question concerns the number of travelling wave solutions with the same wave speed, their stability and interaction properties. Champneys and McKenna in \cite{champneys_mckenna} proved multiplicity in the case $a = 1$, $b = 0$ for given wave speeds from certain nested intervals. As far as we know, the following questions for a general setting of parameters are not answered yet.

\begin{description}
\item{$Q_2$:} \emph{For which fixed values of wave speed $c$ from the admissible interval 
do there exist at least two different travelling wave solutions? 
Or for which fixed values of $c$ the uniqueness result holds true?}
\end{description}

Finding the appropriate answers can be quite difficult, but the following task should be more feasible.

\begin{description}
\item{$T_1$:}
\emph{For some fixed values of $a,b$ and $c$, construct two different travelling wave solutions of \eqref{eq:original_problem}.}
\end{description}
 
Further, our description of the wave speed range is characterized by the envelope of 
a collection of all \fucik\ spectra $\Sigma_r^{\texttt{D}}$ of the Dirichlet boundary value 
problems \eqref{eq:fucik} with $r>0$. 
The problem is that the description of \fucik\ spectra of the fourth-order operators are, 
in general, unknown. Some partial results were obtained, e.g., by Krej\v{c}\'{i} in \cite{Krejci_1983}, 
however, the following question is left open.

\begin{description}
\item{$Q_3$:}  \emph{What is the analytical description of the \fucik\ spectrum $\Sigma_r^\textup{\texttt{D}}$ for $r>0$? Or can we analytically describe their envelope?}
\end{description}

In Section \ref{sec:approx1}, 
we show that suitable polynomial and/or rational approximations of Dirichlet \fucik~eigenfunctions 
lead to explicit analytical approximations of the admissible wave speed range
given in \eqref{eq:c}
(cf. Theorems \ref{veta11} and \ref{veta12}). 
Moreover, in Section \ref{sec:approx2}, we show for which pairs 
$\left(4a/c^{4},4b/c^{4}\right)$ there exists the travelling wave solution of
the problem \eqref{eq:original_problem} with the wave speed $c$ 
(cf. Theorems \ref{veta22} and \ref{veta33}).
Let us point out that those pairs are below the curve $\mu$ or below $\eta_{T}^{p_{1},p_{2}}$
(see Figure \ref{fig:main_222})
and that there is still a gap between $\eta_{T}^{p_{1},p_{2}}$ and the border line $\mathscr{E}$
under which $J(\alpha,\beta,v)$ possesses negative
values for some $v \in H^2(\R)$.
Thus, it is reasonable to formulate the following task. 

\begin{description}
\item{$T_2$:} \emph{
Let $\alpha_{0} > 1$ be fixed and denote
$$
  \beta_{0}^{\bullet} := 
  \sup_{\substack{x_{1},x_{2} > 0 \\ p_{1},p_{2} > 0}}
  \left\{
  \beta_{0} > 0:\ \textup{the equation \eqref{rce2} is satisfied}
  \right\}.
$$
For $\beta > \beta_{0}^{\bullet}$, construct $v\in H^{2}(\R)$ such that $J(\alpha_{0}, \beta, v) < 0$.}
\end{description}
 
Unlike $\Sigma_r^{\texttt{D}}$, description of the \fucik\ spectrum $\Sigma_T^{\texttt{P}}$ 
of the periodic problem \eqref{eq:fucik_per} with $T>0$ is much more straightforward. 
Motivated by the fact that suprema of Dirichlet eigenvalues correspond to suprema 
(or maxima, respectively) of periodic eigenvalues, we expect that the same holds 
true also for the envelopes of Dirichlet and periodic \fucik\ spectra and that this 
fact can be used for a better characterization of the optimal travelling wave speed range. 
Unfortunately, this statement remains in the form of hypothesis and leads to the following question. 

\begin{description}
\item{$Q_4$:}  \emph{Is it possible to confirm or refute Conjecture \ref{con:periodic1}?}
\end{description}


\appendix
\section{Eigenvalues of Dirichlet and periodic problems}
\label{sec:eigen}

In this appendix, we provide a description of periodic and Dirichlet eigenvalues of  \eqref{eq:eigen} and \eqref{eq:eigen_periodic}. For simplicity, we consider both problems with additional evenness assumption, i.e., the Dirichlet eigenvalue problem
\begin{equation}
\label{eq:eigen_even}
\left\{ \begin{array}{l}
v^{(4)} + 2v'' + \lambda v = 0, \quad x \in (-r,r),\\[3pt]
v(\pm r) = v'(\pm r) = 0, \, v(-x) = v(x),
\end{array} \right.
\end{equation}
and the periodic eigenvalue problem
\begin{equation}
\label{eq:eigen_per_even}
\left\{ \begin{array}{l}
v^{(4)} + 2v'' + \lambda v = 0, \quad x \in \R,\\[3pt]
v(x + T) = v(x) = v(-x).
\end{array} \right.
\end{equation}
As we mentioned in Section \ref{sec:connection}, the equation in both \eqref{eq:eigen_even}, \eqref{eq:eigen_per_even} can be rewritten with Swift-Hohenberg operator as $(\partial^2 + 1)^2 v = (1-\lambda) v$. 
By a direct calculation we easily obtain the following statement.


\begin{lemma}
\label{lemma:eigen_per}
For any $T>0$, the eigenvalues of \eqref{eq:eigen_per_even} form a sequence 
$(\lambda^\textup{\texttt{P}}_{n,T})_{n=0}^{+\infty}$: 
$$
\lambda^\textup{\texttt{P}}_{n,T} = 1 - \left( \left( {2\pi n}/{T}\right)^2 - 1 \right)^2
$$
with corresponding eigenfunctions in the form
$v_{n,T}^\textup{\texttt{P}}(x)= \cos (2\pi n x/T)$, $n \in \N_0$.
Consequently,
$$
\sup_{T>0,n\in\N_0} \lambda^\textup{\texttt{P}}_{n,T} = \lambda^\textup{\texttt{P}}_{k,2\pi k} = 1, \quad k \in \N,
$$
with  
$v_{k,2\pi k}^\textup{\texttt{P}}(x) = \cos x$.
\end{lemma}
Dependence of $\lambda^{\texttt{P}}_{n,T}$ on $T$ is illustrated in Figure \ref{fig:eigen}, left. As for the Dirichlet eigenvalues, the situation is more complicated.
We split its description into several partial assertion.

\begin{figure}[htb]
\centerline{
  \setlength{\unitlength}{1mm}
  \begin{picture}(65, 65)(0,0)
    \put(0,5){\includegraphics[height=6.0cm]{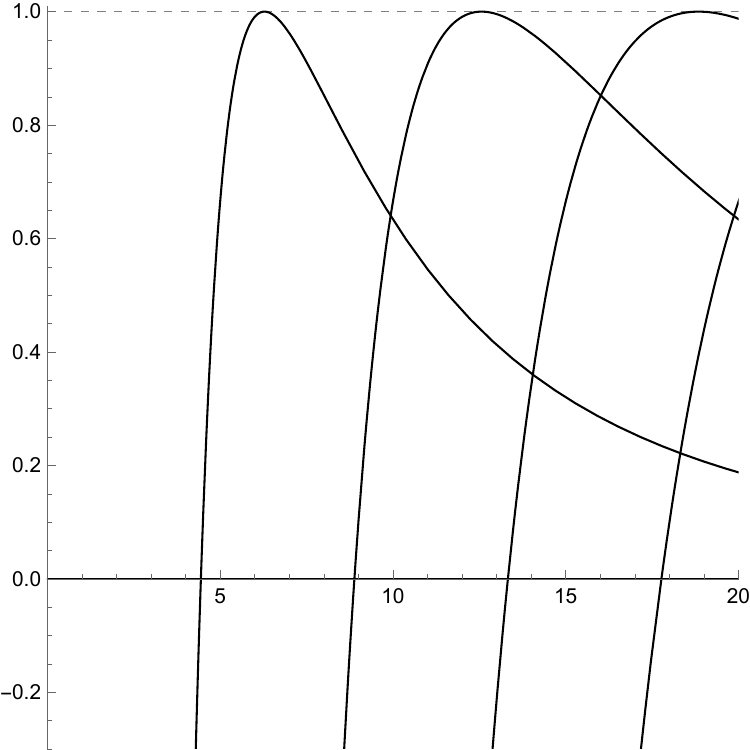}}
  	\put(5,14){\makebox(0,0)[lb]{\footnotesize$\lambda^{\texttt{P}}_{0,T}$}} 
    \put(12,0){\makebox(0,0)[lb]{\footnotesize$\lambda^{\texttt{P}}_{1,T}$}}     
    \put(25,0){\makebox(0,0)[lb]{\footnotesize$\lambda^{\texttt{P}}_{2,T}$}}     
    \put(37,0){\makebox(0,0)[lb]{\footnotesize$\lambda^{\texttt{P}}_{3,T}$}}         
    \put(50,0){\makebox(0,0)[lb]{\footnotesize$\lambda^{\texttt{P}}_{4,T}$}}
    \put(65,18){\makebox(0,0)[lb]{\footnotesize$T$}}
    \put(-5,62){\makebox(0,0)[lb]{\footnotesize$\lambda$}}
  \end{picture}    
\hspace{1.4cm} 
  \begin{picture}(65, 65)(0,0)
    \put(0,5){\includegraphics[height=6.0cm]{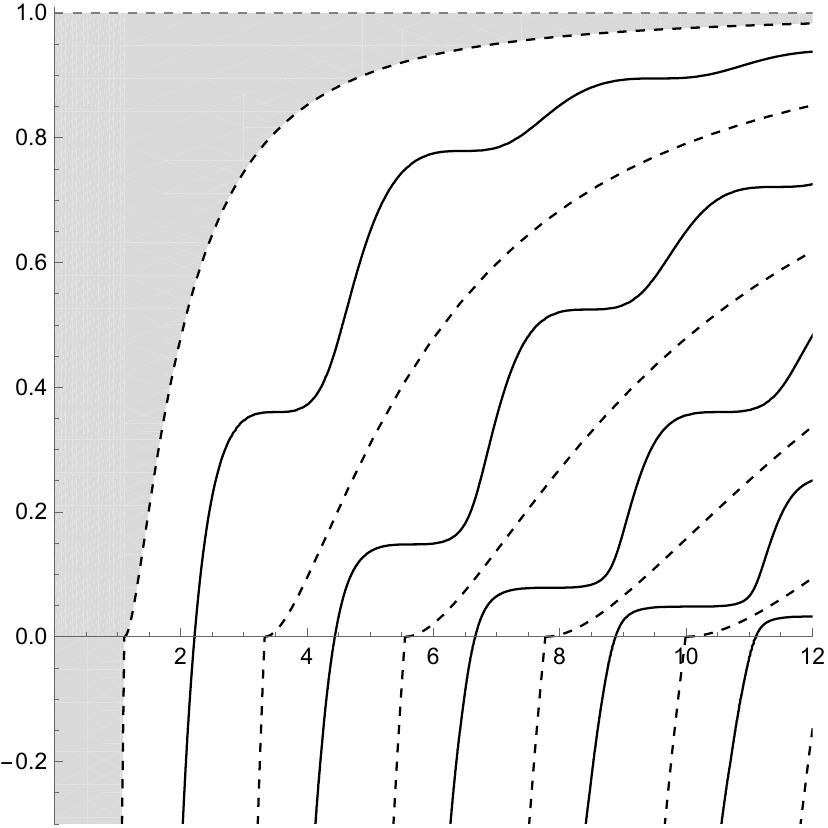}}
    \put(61,62){\makebox(0,0)[lb]{\footnotesize$\bar\lambda_1$}}
    \put(11,0){\makebox(0,0)[lb]{\footnotesize$\lambda^{\texttt{D}}_{1,r}$}}     
    \put(61,56){\makebox(0,0)[lb]{\footnotesize$\bar\lambda_2$}}     
    \put(21,0){\makebox(0,0)[lb]{\footnotesize$\lambda^{\texttt{D}}_{2,r}$}}         
    \put(61,46){\makebox(0,0)[lb]{\footnotesize$\bar\lambda_3$}}
    \put(30,0){\makebox(0,0)[lb]{\footnotesize$\lambda^{\texttt{D}}_{3,r}$}}
    \put(61,35){\makebox(0,0)[lb]{\footnotesize$\bar\lambda_4$}}
    \put(40,0){\makebox(0,0)[lb]{\footnotesize$\lambda^{\texttt{D}}_{4,r}$}}
    \put(61,24){\makebox(0,0)[lb]{\footnotesize$\bar\lambda_5$}}
    \put(50,0){\makebox(0,0)[lb]{\footnotesize$\lambda^{\texttt{D}}_{5,r}$}}
    \put(61,12){\makebox(0,0)[lb]{\footnotesize$\bar\lambda_6$}}  
    \put(65,18){\makebox(0,0)[lb]{\footnotesize$r$}}    
    \put(-5,62){\makebox(0,0)[lb]{\footnotesize$\lambda$}}         
  \end{picture}    
  }
\caption{Eigenvalues $\lambda^{\texttt{P}}_{n,T}$ of \eqref{eq:eigen_per_even} 
for $n=0,1,2,3,4$ on the left, and components of the zero level set of $F$ given by \eqref{eq:F_a_lambda} 
(solid curves) and their bounds (dashed curves) on the right.}
\label{fig:eigen}
\end{figure}
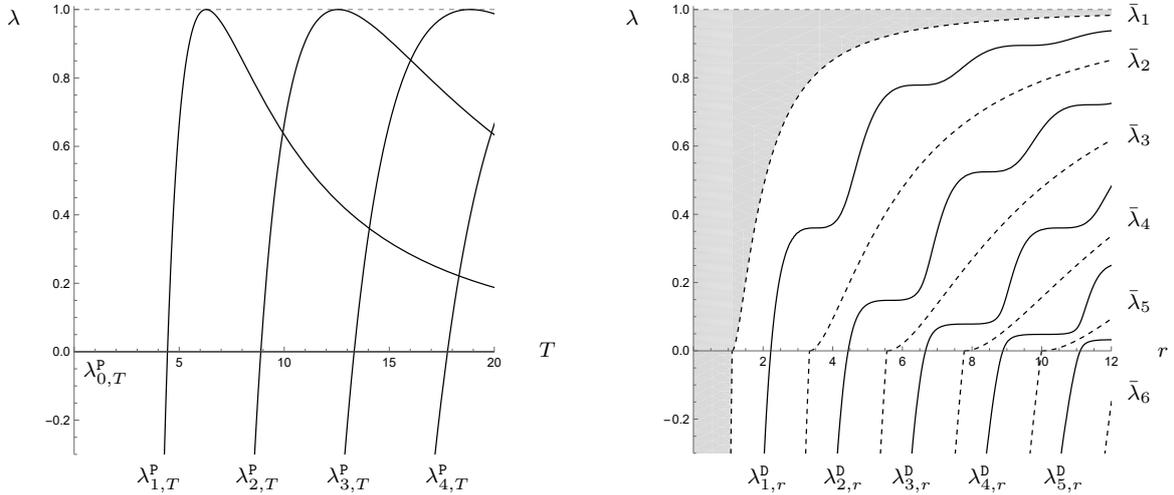
\begin{lemma}
\label{lemma:lambda>=1}
The problem \eqref{eq:eigen_even} has only a trivial solution for any $\lambda \geq 1$.
\end{lemma}

\begin{proof}
We move on to the weak formulation of \eqref{eq:eigen_even} with $v \in H_0^2(-r,r)$ and take $\varphi = v$ as the test function to get
\begin{equation}
\nonumber
\label{eq:weak_form_eigen}
\int\limits_{-r}^{r} \left( (v'')^2 -2 (v')^2 + \lambda v^2 \right) \d x = 0.
\end{equation}
Now, we extend $v$ by zero onto the whole real line.  If we denote $\hat{v}(\omega) = \mathcal{F}(v(x))$ the Fourier transform of $v=v(x)$ and use $\mathcal{F}(v^n(x)) = (\ii \omega)^n \hat{v}(\omega)$ with $\ii^2 = -1$, we can write
\begin{equation}
\nonumber
\label{eq:weak_form_eigen_R}
\int\limits_{\R} \left( (v'')^2 -2 (v')^2 + \lambda v^2 \right) \d x = \int\limits_{\R} \left(\omega^4-2\omega^2+\lambda \right)\hat{v}^2 \d \omega = 0.
\end{equation}
Since $\omega^4-2\omega^2+\lambda \geq 0$ for all $\lambda \geq 1$ and the equality occurs only for $\lambda = 1$ at $\omega = \pm 1$, we obtain trivially $\hat{v}(\omega) = v(x) \equiv 0$.
\end{proof}

\begin{lemma}
\label{lemma:eigen1}
For any $r>0$, the eigenvalues of \eqref{eq:eigen_even} form a strictly decreasing sequence $\left(\lambda_{n,r}^\textup{\texttt{D}}\right)_{n=1}^{+\infty} \subset (-\infty,1)$ with $\lim\limits_{n\rightarrow + \infty}\lambda_{n,r}^\textup{\texttt{D}}=-\infty$. In particular, $\lambda_{n,r}^\textup{\texttt{D}}$ are solutions of $F(r,\lambda)=0$ where
\begin{equation}
\label{eq:F_a_lambda}
F(r,\lambda) =
\left\{ 
\begin{array}{lll} 
\nu_2 \sin(\nu_2 r) \cos(\nu_1 r)-\nu_1 \cos(\nu_2 r) \sin(\nu_1 r)   & \textup{for}~\lambda \in (0,1),\\[3pt]
\sqrt{2}\sin(\sqrt{2}r) & \textup{for}~\lambda = 0,\\[3pt]
\nu_2 \sin(\nu_2 r)\cosh(\nu_1 r) + \nu_1 \cos(\nu_2 r)\sinh(\nu_1 r)  & \textup{for}~\lambda <0,
\end{array} 
\right.
\end{equation}
with $\nu_1 = \sqrt{1 - \sqrt{1-\lambda}}$ for $\lambda \in (0,1)$, $\nu_1 = \sqrt{-1+\sqrt{1-\lambda}}$ for $\lambda < 0$, and $\nu_2 = \sqrt{1+\sqrt{1-\lambda}}$.
\end{lemma}

\begin{proof}
The problem \eqref{eq:eigen_even} is the standard Sturm-Liouville problem of the fourth order and the existence of a monotone sequence of its eigenvalues $\left(\lambda^{\texttt{D}}_{n,r}\right)_{n=1}^{+\infty}$ bounded from above was proved, e.g., in \cite{greenberg}. Due to the evenness assumption, all the eigenvalues are simple.
Moreover, Lemma~\ref{lemma:lambda>=1} implies $\lambda^{\texttt{D}}_{n,r}<1$ for all $n \in \N$. 
By a direct calculation we easily obtain that $\lambda = \lambda^{\texttt{D}}_{n,r}$ is an eigenvalue of \eqref{eq:eigen_even} if and only if $F(r,\lambda)=0$ with $F$ given by \eqref{eq:F_a_lambda}. 
\end{proof}

The zero level set of $F$ is depicted in Figure \ref{fig:eigen}, right (solid curves). 
Now, we discuss the dependence of $\lambda^{\texttt{D}}_{n,r}$ on $r$. 
 
\begin{lemma}
\label{lemma:eigen2}
For all $n \in \N$, $\lambda_{n,r}^\textup{\texttt{D}}$ are increasing $C^1$ functions of $r > 0$ satisfying 
$$
\sup_{r>0} \lambda_{n,r}^\textup{\texttt{D}} = \lim_{r \to +\infty} \lambda_{n,r}^\textup{\texttt{D}} = 1 \quad \mbox{and} \quad 
\lim_{r \to 0+} \lambda_{n,r}^\textup{\texttt{D}} = -\infty.
$$ 
\end{lemma}

\begin{proof}
We start with the region $r>0$, $\lambda \in (0,1)$ and examine properties of the set $F(r,\lambda) = 0$
(cf. Lemma \ref{lemma:eigen1} and \eqref{eq:F_a_lambda}). Using standard calculations we obtain $F_{\lambda} = \partial F / \partial {\lambda} \neq 0$ and $\d \lambda/\d r = -F_r/F_{\lambda} \geq 0$ for all $r>0$ and $\lambda \in (0,1)$ satisfying $F(r,\lambda) = 0$. 
The Implicit Function Theorem implies that the zero level set of $F$ for $r>0$, $\lambda \in (0,1)$ can be locally described as an increasing function $\lambda = \lambda(r)$. To see the global description, we use the following transform.

Let us rewrite $F(r,\lambda) = (\nu_2-\nu_1)\sin r(\nu_2+\nu_1) + (\nu_2+\nu_1)\sin r(\nu_2-\nu_1)$, set $\xi := r(\nu_2+\nu_1)$ and $\eta := r(\nu_2-\nu_1)$ and define 
\begin{equation}
\nonumber
\label{eq:transform2}
\widehat{F}(\xi,\eta) := \dfrac{\sin \xi}{\xi} + \dfrac{\sin \eta}{\eta}
\end{equation}
with $\xi > \eta > 0$.
Thus $F(r,\lambda) = 0$ if and only if $\widehat{F}(\xi,\eta) = 0$. 
The separated form of $\widehat{F}$ enables to show easily that the zero level set of $\widehat{F}$ consists of countably many functions $\eta_n = \eta_n(\xi)$, $n \in \N$ defined for $\xi > n\pi$ and satisfying $(2n - 1)\pi/2 < \eta_n(\xi) < (2n + 1)\pi/2$, $n \in \N$ (see the solid curves $\eta_n(\xi)$ separated by the dashed ones in Figure \ref{fig:eigen_dir_tr}). We also have $\widehat{F}(\xi,\eta) > 0$ for any $\xi > \eta \in (0,\pi/2]$, i.e., no component of $\widehat{F}=0$ lies in the gray region in Figure \ref{fig:eigen_dir_tr}. 

Moving back to the original variables, we get that the zero level set of $F$ in the region $r>0$, $\lambda \in (0,1)$ can be represented by countably many (increasing) functions $\lambda_n = \lambda_n(r)$, $n \in \N$, defined for $r > \sqrt{2}n\pi/2$ and separated by curves
\begin{equation}
\label{eq:curves_lambda+}
\bar\lambda_n = \left(1 - {\pi^2 (2n-1)^2}/({8 r^2})\right)^2, \quad r \in \left({\sqrt{2}\pi (2n-1)}/{4},+\infty\right), ~ n \in \N,
\end{equation}
(see the dashed curves in Figures \ref{fig:eigen}, right).
Furthermore, $F(r,\lambda)>0$ for all $\lambda \geq (1-\pi^2/(8r^2))^2$ and there are no eigenvalues in the gray region in Figures \ref{fig:eigen}, right.

We can proceed analogously for $\lambda < 0$ to get functions $\lambda_n = \lambda_n(r)$, $r \in (0, \sqrt{2}n\pi/2)$, describing the zero level set of $F$ for $r > 0$ and $\lambda < 0$, separated by curves
\begin{equation}
\label{eq:curves_lambda-}
\bar\lambda_n = 1 - \left(1-{\pi^2 (2n-1)^2}/({4r^2})\right)^2, \quad r \in \left(0,{\sqrt{2}\pi (2n-1)}/{4}\right], ~ n \in \N.
\end{equation}
Additionally, functions $\lambda_n(r)$ are smoothly connected at points $(r, \lambda) = (\sqrt{2}n\pi/2,0)$. The asymptotic behavior of the bounds $\bar\lambda_n$ (cf. \eqref{eq:curves_lambda+}, \eqref{eq:curves_lambda-}) implies that, for all $n \in \N$, $\lambda_n(r) \rightarrow 1$ for $r \rightarrow +\infty$ and $\lambda_n(r) \rightarrow -\infty$ for $r \rightarrow 0+$.
\end{proof}

\begin{figure}[h]
\centerline{
  \setlength{\unitlength}{0.95mm}
  \begin{picture}(146, 72)(0, -4)
    \put(4,0){\includegraphics[width=13cm]{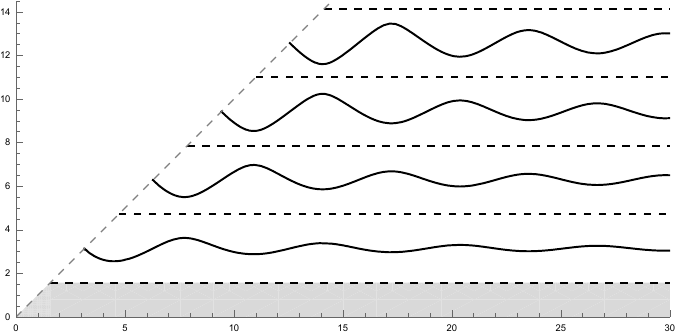}}
    \put(0,60){\makebox(0,0)[lb]{$\eta$}}
    \put(52,61){\makebox(0,0)[lb]{$\eta = \xi$}}
    \put(128,-3){\makebox(0,0)[lb]{$\xi$}}    
    \put(143,7){\makebox(0,0)[lb]{$\frac{\pi}{2}$}}
    \put(143,15){\makebox(0,0)[lb]{$\eta_1$}}     
    \put(143,21){\makebox(0,0)[lb]{$\frac{3\pi}{2}$}}     
    \put(143,30){\makebox(0,0)[lb]{$\eta_2$}}         
    \put(143,36){\makebox(0,0)[lb]{$\frac{5\pi}{2}$}}
    \put(143,45){\makebox(0,0)[lb]{$\eta_3$}}
    \put(143,50){\makebox(0,0)[lb]{$\frac{7\pi}{2}$}}
    \put(143,57){\makebox(0,0)[lb]{$\eta_4$}}
    \put(143,64){\makebox(0,0)[lb]{$\frac{9\pi}{2}$}}
  \end{picture}    
  }
\caption{Components of the zero level set of $\widehat{F}$ (solid curves) and their bounds (dashed curves) in the region $\xi > \eta > 0$, i.e., for $r>0$ and $\lambda \in (0,{1})$.}
\label{fig:eigen_dir_tr}
\end{figure}
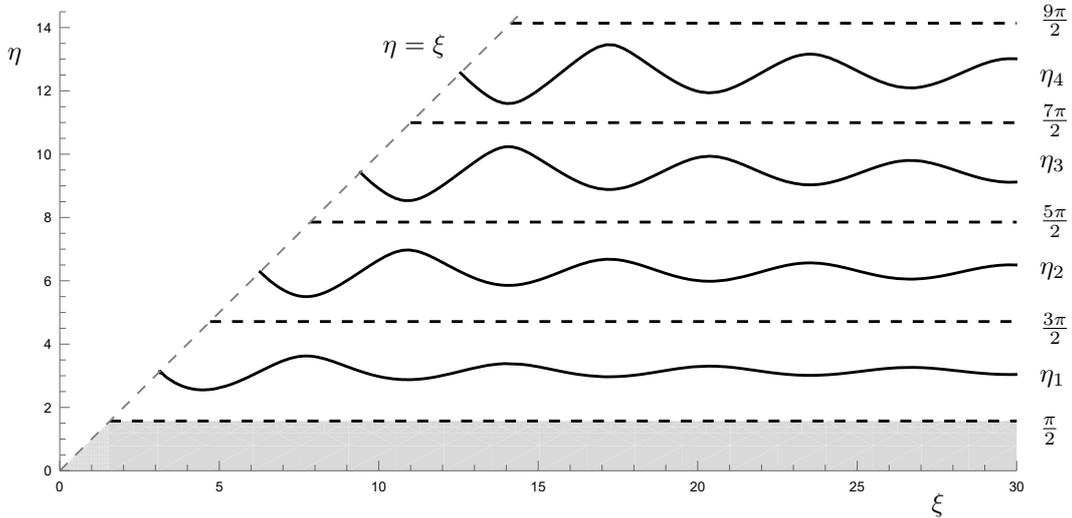

\begin{remark}
\label{rem:eigen}
The properties of $F(r,\lambda)$ also imply that $\d \lambda_n(r)/\d r = 0$ if and only if 
$r \nu_2 = (2k-1)\pi/2$ and $r\nu_1 = (2j -1)\pi/2$ for $k,j \in \N$. 
In $r\lambda$ plane these points correspond to mutual intersections of curves
$\lambda = 1 - \left(1-{\pi^2 (2k-1)^2}/({4r^2})\right)^2$, $k \in \N$. 

In particular, on the $\lambda_1(r)$ component of 
$F(r,\lambda) = 0$, $\d \lambda_1/\d r = 0$ if and only if 
$$(r,\lambda_1) = \left(\pi\sqrt{4k^2+1}/2,(1 - 4 k^{2})^2/(1+4k^2)^2\right),\quad k\in \N.$$
Importance of these points lies in the fact that the nodal properties 
of the corresponding eigenfunctions  $v_{1,r}^{\texttt{D}}$ change therein. 
Namely, for $r_k = \pi\sqrt{4k^2+1}/2$, $k \in \N$, we have
(see Figure \ref{obr19})
\begin{eqnarray}
\label{eq:eigenfunction}
v_{1,r_k}^{\texttt{D}}(x) = 
  \tfrac{2k+1}{4k}\cos\left(\tfrac{(2k-1)}{\sqrt{4k^2+1}}x\right)+ 
  \tfrac{2k-1}{4k}\cos\left(\tfrac{(2k+1)}{\sqrt{4k^2+1}}x\right).
\end{eqnarray}
That is, the first eigenfunction corresponding to 
$\lambda_{1,r}^{\texttt{D}} = \lambda_1(r)$ is of constant sign only 
for $0< r \le \pi\sqrt{5}/2$.
\end{remark}

\begin{figure}[h]
\centerline{
  \setlength{\unitlength}{1mm}
  \begin{picture}(44, 28)(0, 0)
    \put(0,0){\includegraphics[width=4.5cm]{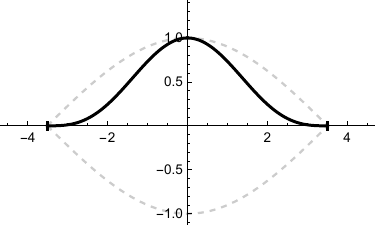}}
    \put(43,9.0){\makebox(0,0)[lb]{\footnotesize$x$}}  
    \put(1.5,14.0){\makebox(0,0)[lb]{\footnotesize$-r_{1}$}}            
    \put(38.5,14.0){\makebox(0,0)[lb]{\footnotesize$r_{1}$}}        
    \put(24.0,24.0){\makebox(0,0)[lb]{\footnotesize$y = v_{1,r_{1}}^{\texttt{D}}(x)$}}            
  \end{picture}    
  \hspace{6pt}
  \begin{picture}(44, 28)(0, 0)
    \put(0,0){\includegraphics[width=4.5cm]{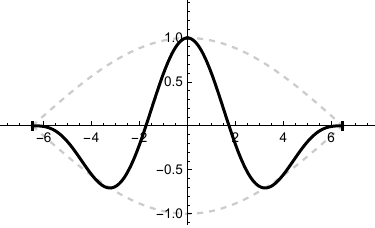}}
    \put(0.0,14.0){\makebox(0,0)[lb]{\footnotesize$-r_{2}$}}            
    \put(40.0,14.0){\makebox(0,0)[lb]{\footnotesize$r_{2}$}}        
    \put(43,9.0){\makebox(0,0)[lb]{\footnotesize$x$}}    
    \put(24.0,24.0){\makebox(0,0)[lb]{\footnotesize$y = v_{1,r_{2}}^{\texttt{D}}(x)$}}            
  \end{picture}    
  \hspace{6pt}  
  \begin{picture}(44, 28)(0, 0)
    \put(0,0){\includegraphics[width=4.5cm]{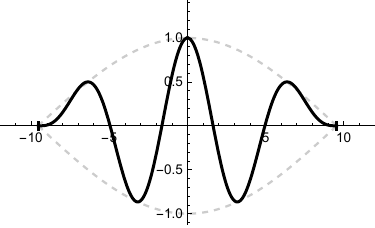}}
    \put(0.5,14.0){\makebox(0,0)[lb]{\footnotesize$-r_{3}$}}            
    \put(39.5,14.0){\makebox(0,0)[lb]{\footnotesize$r_{3}$}}        
    \put(43,9.0){\makebox(0,0)[lb]{\footnotesize$x$}}    
    \put(24.0,24.0){\makebox(0,0)[lb]{\footnotesize$y = v_{1,r_{3}}^{\texttt{D}}(x)$}}            
  \end{picture}    
  }
\caption{The eigenfunctions $v_{1,r}^{\texttt{D}}$
of the Dirichlet problem \eqref{eq:eigen} for
$r = r_{1}=\pi\sqrt{5}/2$ (left), 
$r = r_{2}=\pi\sqrt{17}/2$ (middle), 
$r = r_{3}=\pi\sqrt{37}/2$ (right).
}
\label{obr19}
\end{figure}
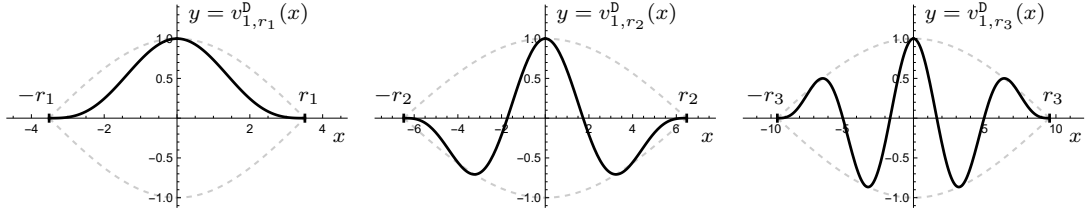

\section{\fucik~spectrum $\Sigma^{\texttt{P}}_T$ of the periodic problem}
\label{sec:fucik_per}

In this part, we show the way how to
describe the periodic \fucik\ spectrum $\Sigma^{\texttt{P}}_T$ of \eqref{eq:fucik_per}. 
For simplicity, we omit our considerations only to the ``first even'' component $\Sigma^{\texttt{P}}_{1,T}$ of $\Sigma^{\texttt{P}}_T$. 
To be more specific, by $\Sigma^{\texttt{P}}_{1,T}$ we mean the set of
pairs $(\alpha,\beta) \in \Sigma^{\texttt{P}}_T$ such that
there exist a nontrivial solution of \eqref{eq:fucik_per} which is 
$T$ periodic, even, sign-changing and has exactly one node on $\big(0,\frac{T}{2}\big)$.
The following statement provides the full description of $\Sigma^{\texttt{P}}_{1,T}$ in $\Omega$, other parts 
of $\Sigma^{\texttt{P}}_T$ could be treated in the same way. 
Curves $\Sigma^{\texttt{P}}_{1,T} \cap \Omega$ for several values of $T$ are illustrated in Figure~\ref{obr16}. 
Notice also that $\left(\lambda_1^{\texttt{P}}(T),\lambda_1^{\texttt{P}}(T)\right) \in \Sigma^{\texttt{P}}_{1,T}$. 

\begin{lemma}
\label{lemma:fucik_per}
For  $(\alpha,\beta) \in \Omega$, let us denote
$$
\mu_1 := \sqrt[4]{\alpha}\sin \left(\tfrac{1}{2} \arctan \sqrt{\alpha - 1} \right), \quad 
\mu_2 := \sqrt[4]{\alpha}\cos \left(\tfrac{1}{2} \arctan \sqrt{\alpha - 1} \right), 
$$
$$
\nu_1 := \sqrt{1-\sqrt{1-\beta}}, \quad \nu_2 := \sqrt{1+\sqrt{1-\beta}},
$$
and introduce
\begin{align*}
F_{ij}(\alpha,\beta,\tau,T) := & -2\mu_1 \mu_2 (\cosh 2\mu_1 \tau + \cos 2\mu_2 \tau) \nu_i \sin \nu_i\big(\tau-\tfrac{T}{2}\big) + \\
& \left( -\mu_1(\mu_1^2 - 3\mu_2^2 + \nu_j^2)\sin 2\mu_2 \tau + \mu_2(\mu_2^2 - 3\mu_1^2-\nu_j^2) \sinh 2 \mu_1 \tau \right)\cos \nu_i\big(\tau-\tfrac{T}{2}\big).
\end{align*}
Then
$(\alpha,\beta) \in \Sigma^\textup{\texttt{P}}_{1,T} \cap \Omega$ 
with $T>0$ if and only if there exist $\tau \in (0,\pi/(2\mu_2))$ such that 
$$
F_{12}(\alpha,\beta,\tau,T) = 0 \quad \wedge \quad F_{21}(\alpha,\beta,\tau,T) = 0.
$$
\end{lemma}

\begin{proof}
A couple $(\alpha,\beta)$ belongs to $\Sigma^\textup{\texttt{P}}_{1,T} \cap \Omega$ if and only if there exists a nontrivial even solution of \eqref{eq:fucik_per} in the form  
$$
v(x) = \left\{\begin{array}{ll}
v_{\texttt{pw}}(x) = A_1 \cosh \mu_1 x \cos \mu_2 x + A_2 \sinh \mu_1 x \sin \mu_2 x, \quad & x \in (-\tau,\tau),\\[3pt]
v_{\texttt{nw}}(x) = B_1 \cos \nu_1\big(x-\frac{T}{2}\big) + B_2 \cos \nu_2\big(x-\frac{T}{2}\big), \quad & x \in (\tau,T-\tau)
\end{array}
\right.
$$
with the only node $\tau \in \big(0,\frac{T}{2}\big)$ and real constants 
$A_{i}$,  $B_{i}$, $i=1,2$.
To ensure smoothness of $v$, we must have 
$v_{\texttt{pw}}^{(i)}(\tau-) = v_{\texttt{nw}}^{(i)}(\tau+)$ for $i = 0,1,2,3$.
By a direct substitution into these conditions and after suitable simplifications we obtain $F_{12}(\alpha,\beta,\tau,T) = 0$ and $F_{21}(\alpha,\beta,\tau,T)=0$ with $\tau \in (0,\pi/(2\mu_2))$.
\end{proof}

Notice that for $T>0$ fixed, both expressions 
$F_{12}(\alpha,\beta,\tau,T) = 0$, $F_{21}(\alpha,\beta,\tau,T)=0$ 
represent surfaces in $(\alpha,\beta,\tau)$ space. That is, $\Sigma^{\texttt{P}}_{1,T}$ 
is a curve which corresponds to the projection of their intersection into 
$\alpha\beta$ plane and the connection point $\tau$ can be seen as a parameter of this curve. 
Moreover, the envelope 
$\hat{\cal E}_1$ of $\Sigma_{1,T}^{\texttt{P}}$
can be characterized in the following way. 

\begin{lemma}
\label{lemma:per_envelope}
Let $\mu_{1}$, $\mu_{2}$ and $\nu_{1}$, $\nu_{2}$ be as in Lemma \ref{lemma:fucik_per} and let us set
$$
G(\alpha,\beta,\tau) := \nu_2 \arctan G_{12}(\alpha,\beta,\tau) - \nu_1 (\arctan G_{21}(\alpha,\beta,\tau) - \pi)
$$
with
$$
G_{ij}(\alpha,\beta,\tau) := \frac{\mu_2 (\mu_2^2 - 3\mu_1^2 - \nu_j^2)\sinh 2\mu_1 \tau - \mu_1 (\mu_1^2 - 3\mu_2^2 + \nu_j^2) \sin 2\mu_2 \tau}{2\mu_1 \mu_2 \nu_i (\cosh 2 \mu_1 \tau + \cos 2 \mu_2 \tau)}.
$$
Then 
$(\alpha,\beta) \in \hat{\cal E}_1 \cap \Omega$ 
if and only if 
there exist $\tau \in (0,\pi/(2\mu_2))$ such that
$$
G(\alpha,\beta,\tau) = 0 \quad \wedge \quad \frac{\partial G}{\partial \tau}(\alpha,\beta,\tau) = 0.
$$
\end{lemma}

\begin{proof}
Using Lemma \ref{lemma:fucik_per} and expressing $\big(\tau - \frac{T}{2}\big)$ from both equations $F_{12}=0$, $F_{21}=0$, we obtain $G =0$. Hence, adding a condition $\partial G/ \partial{\tau} = 0$ leads to the desired envelope description.
\end{proof}


\section{Expansion of $J(\alpha,\beta,v_{n})$ in terms of $n$}
\label{sec:app3}

In this part, we justify the formula \eqref{ref10} in the proof of Lemma~\ref{veta1}.
Let us consider functions $v_{n} \in H^2(\R)$ defined in \eqref{rce11}, i.e., we have
\begin{equation} \label{ref11}
  v_{n}(x) = 
  \left\{
  \begin{array}{lll}
    \vpt(x - p_{k})w_{n}(x)                
    && \text{for } x \in [p_{k} - x_{1}, p_{k} + x_{1}],\ k\in\{1-n,\dots,n-1\}, \\[3pt]
    \vnt(x - q_{k})w_{n}(x) 
    && \text{for } x \in (q_{k} - x_{2}, q_{k} + x_{2}),\ k\in\{-n,\dots,n-1\}, \\[3pt]
    0 && \text{for } |x| \ge s_{n},
  \end{array}
  \right.
\end{equation}
where
$p_{k} = 2k(x_{1} + x_{2})$ and $q_{k} = (2k+1)(x_{1} + x_{2})$.
By a direct computation, we obtain
\begin{align*}
  \int_{p_{k} - x_{1}}^{p_{k} + x_{1}} \left(v_{n}\right)^{2} \dd x 
  & = 
  G_{0}(x_{1})
  \left(\tfrac{p_{k}^{2}}{s_{n}^{2}} - 1\right)^{2} +  
  G_{1}(x_{1})\cdot\tfrac{p_{k}^{2}}{s_{n}^{4}} + 
  G_{2}(x_{1},s_{n}),
\\
  \int_{p_{k} - x_{1}}^{p_{k} + x_{1}} \left(v_{n}'\right)^{2} \dd x 
  & = 
  M_{0}(x_{1})\left(\tfrac{p_{k}^{2}}{s_{n}^{2}} - 1\right)^{2} +    
  M_{1}(x_{1})\cdot\tfrac{p_{k}^{2}}{s_{n}^{4}} +
  M_{2}(x_{1},s_{n}),
\\
  \int_{p_{k} - x_{1}}^{p_{k} + x_{1}} \left(v_{n}''\right)^{2} \dd x 
  & = 
  N_{0}(x_{1})\left(\tfrac{p_{k}^{2}}{s_{n}^{2}} - 1\right)^{2} +     
  N_{1}(x_{1})\cdot\tfrac{p_{k}^{2}}{s_{n}^{4}} +
  N_{2}(x_{1},s_{n}),
\\
  \int_{q_{k} - x_{2}}^{q_{k} + x_{2}} \left(v_{n}\right)^{2} \dd x 
  & = 
  G_{0}(x_{2})\left(\tfrac{q_{k}^{2}}{s_{n}^{2}} - 1\right)^{2} +  
  G_{1}(x_{2})\cdot\tfrac{q_{k}^{2}}{s_{n}^{4}} + 
  G_{2}(x_{2},s_{n}),
\\
  \int_{q_{k} - x_{2}}^{q_{k} + x_{2}} \left(v_{n}'\right)^{2} \dd x 
  & = 
  M_{0}(x_{2})\left(\tfrac{q_{k}^{2}}{s_{n}^{2}} - 1\right)^{2} +    
  M_{1}(x_{2})\cdot\tfrac{q_{k}^{2}}{s_{n}^{4}} +
  M_{2}(x_{2},s_{n}),
\\
  \int_{q_{k} - x_{2}}^{q_{k} + x_{2}} \left(v_{n}''\right)^{2} \dd x 
  & = 
  N_{0}(x_{2})\left(\tfrac{q_{k}^{2}}{s_{n}^{2}} - 1\right)^{2} +     
  N_{1}(x_{2})\cdot\tfrac{q_{k}^{2}}{s_{n}^{4}} +
  N_{2}(x_{2},s_{n}),
\end{align*}
where functions $G_{i}$, $M_{i}$ and $N_{i}$ for $i=0,1,2$ are given by
\begin{align*}
\begin{array}{rclrclrcl}
  G_{0}(x) &\hspace{-6pt}=\hspace{-6pt}& \displaystyle\tfrac{4 x^{3}}{15}, \qquad
  &
  G_{1}(x) &\hspace{-6pt}=\hspace{-6pt}& \displaystyle\tfrac{8 x^{5}}{35}, \qquad
  &
  G_{2}(x,s) &\hspace{-6pt}=\hspace{-6pt}& \displaystyle\tfrac{4 x^{7} - 24 s^{2} x^{5}}{315 s^{4}},
  \\[12pt]
  M_{0}(x) &\hspace{-6pt}=\hspace{-6pt}& \displaystyle\tfrac{2 x}{3}, \qquad
  &
  M_{1}(x) &\hspace{-6pt}=\hspace{-6pt}& \displaystyle\tfrac{28 x^{3}}{15}, \qquad
  &
  M_{2}(x,s) &\hspace{-6pt}=\hspace{-6pt}& \displaystyle\tfrac{22 x^{5} - 28 s^{2} x^{3}}{105 s^{4}},
  \\[12pt]
  N_{0}(x) &\hspace{-6pt}=\hspace{-6pt}& \displaystyle\tfrac{2}{x}, \qquad
  &
  N_{1}(x) &\hspace{-6pt}=\hspace{-6pt}& \displaystyle 28 x, \qquad
  &
  N_{2}(x,s) &\hspace{-6pt}=\hspace{-6pt}& \displaystyle\tfrac{42 x^{3} - 20 s^{2} x}{5 s^{4}}.
\end{array}  
\end{align*}
Due to \cite{Dattoli_1999}, we have that
\begin{align*}
\sum_{k=0}^{n-1}(x + k y)^r = 
\tfrac{y^{r}}{r+1}\left(B_{r+1}\left(n+\tfrac{x}{y}\right) - B_{r+1}\left(\tfrac{x}{y}\right)\right),
\end{align*}
where $B_{r+1}$ are Bernoulli polynomials, which allows us to simplify
\begin{align*}
  \sum_{k=1}^{n-1} p_{k}^{2} 
  & = T^{2}\sum_{k=1}^{n-1} k^{2} =
  \tfrac{T^{2}}{3}\left(B_{3}(n) - B_{3}(0)\right) =
  T^{2}\left(\tfrac{n^{3}}{3} - \tfrac{n^{2}}{2} + \tfrac{n}{6}\right),\\
  \sum_{k=1}^{n-1} p_{k}^{4} 
  & = T^{4}\sum_{k=1}^{n-1} k^{4} =
  \tfrac{T^{4}}{5}\left(B_{5}(n) - B_{5}(0)\right) =
  T^{4}\left(\tfrac{n^{5}}{5} - \tfrac{n^{4}}{2} + \tfrac{n^{3}}{3} - \tfrac{n}{30}\right).  
\end{align*}
Moreover, since $s_{n}=n T - x_{1}$, we have
\begin{align*}
  \sum_{k=1}^{n-1} \left(\tfrac{p_{k}^{2}}{s_{n}^{2}} - 1\right)^{2} =
  \sum_{k=1}^{n-1} \left(\tfrac{p_{k}^{4}}{s_{n}^{4}} - 2\cdot\tfrac{p_{k}^{2}}{s_{n}^{2}} + 1\right) =
  \tfrac{n}{5} - 2\cdot\tfrac{n}{3} + n + \mathcal{O}(1) = \tfrac{8}{15} n + \mathcal{O}(1)
\end{align*}
and thus, we get
\begin{align}
  \int_{v_{n} > 0} \left(v_{n}\right)^{2}\,\dd x
  & =  
  \int_{-x_{1}}^{x_{1}} \left(v_{n}\right)^{2}\,\dd x +
  2 \sum_{k=1}^{n-1}
  \int_{p_{k} - x_{1}}^{p_{k} + x_{1}} \left(v_{n}\right)^{2}\,\dd x 
  = 
  \tfrac{16}{15}G_{0}(x_{1}) n + \mathcal{O}(1), \label{ref21}
\\
  \int_{v_{n} > 0} \left(v_{n}'\right)^{2}\,\dd x
  & = 
  \int_{-x_{1}}^{x_{1}} \left(v_{n}'\right)^{2}\,\dd x +
  2 \sum_{k=1}^{n-1}
  \int_{p_{k} - x_{1}}^{p_{k} + x_{1}} \left(v_{n}'\right)^{2}\,\dd x
  =
  \tfrac{16}{15}M_{0}(x_{1}) n + \mathcal{O}(1), \label{ref22}
\\
  \int_{v_{n} > 0} \left(v_{n}''\right)^{2}\,\dd x
  & =  
  \int_{-x_{1}}^{x_{1}} \left(v_{n}''\right)^{2}\,\dd x +
  2 \sum_{k=1}^{n-1}
  \int_{p_{k} - x_{1}}^{p_{k} + x_{1}} \left(v_{n}''\right)^{2}\,\dd x 
  = 
  \tfrac{16}{15}N_{0}(x_{1}) n + \mathcal{O}(1). \label{ref23}
\end{align}
Similarly, we obtain 
\begin{align*}
  \sum_{k=0}^{n-1} \left(\tfrac{q_{k}^{4}}{s_{n}^{4}} - 2\cdot\tfrac{q_{k}^{2}}{s_{n}^{2}} + 1\right) =
  \tfrac{n}{5} - 2\cdot\tfrac{n}{3} + n + \mathcal{O}(1) = \tfrac{8}{15} n + \mathcal{O}(1),
\end{align*}
since we have
\begin{align*}
  \sum_{k=0}^{n-1} q_{k}^{2} & = 
  T^{2}\sum_{k=0}^{n-1} \left(k + \tfrac{1}{2}\right)^{2} =
  \tfrac{T^{2}}{3}\left(B_{3}\left(n + \tfrac{1}{2}\right) - B_{3}\left(\tfrac{1}{2}\right)\right) =
  T^{2}\left(\tfrac{n^{3}}{3} - \tfrac{n}{12} \right),\\
  \sum_{k=0}^{n-1} q_{k}^{4} & = 
  T^{4}\sum_{k=0}^{n-1} \left(k + \tfrac{1}{2}\right)^{4} =
  \tfrac{T^{4}}{5}\left(B_{5}\left(n + \tfrac{1}{2}\right) - B_{5}\left(\tfrac{1}{2}\right)\right) =  
  T^{4}\left(\tfrac{n^{5}}{5} - \tfrac{n^{3}}{6} + \tfrac{7n}{240}\right).  
\end{align*}
Moreover, we get
\begin{align}
  \int_{v_{n} < 0} \left(v_{n}\right)^{2}\,\dd x
  =  
  2 \sum_{k=0}^{n-1}
  \int_{q_{k} - x_{2}}^{q_{k} + x_{2}} \left(v_{n}\right)^{2}\,\dd x
  & = 
  \tfrac{16}{15}G_{0}(x_{2}) n + \mathcal{O}(1), \label{ref31}
\\
  \int_{v_{n} < 0} \left(v_{n}'\right)^{2}\,\dd x
  =  
  2 \sum_{k=0}^{n-1}
  \int_{q_{k} - x_{2}}^{q_{k} + x_{2}} \left(v_{n}'\right)^{2}\,\dd x
  & = 
  \tfrac{16}{15}M_{0}(x_{2}) n + \mathcal{O}(1), \label{ref32}
\\
  \int_{v_{n} < 0} \left(v_{n}''\right)^{2}\,\dd x
  =  
  2 \sum_{k=0}^{n-1}
  \int_{q_{k} - x_{2}}^{q_{k} + x_{2}} \left(v_{n}''\right)^{2}\,\dd x 
  & = 
  \tfrac{16}{15}N_{0}(x_{2}) n + \mathcal{O}(1). \label{ref33}
\end{align}
Finally, using \eqref{ref22}, \eqref{ref23}, \eqref{ref32} and \eqref{ref33}, we obtain
\begin{align*}
  \int_{\R} \left(
  \left(v_{n}''\right)^{2} - 2\left(v_{n}'\right)^{2} 
  \right) \dd x 
  & =
  \int_{v_{n} > 0} \left(v_{n}''\right)^{2} \dd x + \int_{v_{n} < 0} \left(v_{n}''\right)^{2} \dd x 
  -2\left(
  \int_{v_{n} > 0} \left(v_{n}'\right)^{2} \dd x + \int_{v_{n} < 0} \left(v_{n}'\right)^{2} \dd x   
  \right)\\[6pt]
  & = 
  \tfrac{16}{15}
  \left(
  N_{0}(x_{1}) + N_{0}(x_{2}) - 2M_{0}(x_{1}) - 2M_{0}(x_{2})
  \right)n
  + \mathcal{O}(1),
\end{align*}
and using \eqref{ref21} and \eqref{ref31}, we get
\begin{align*}
  J(\alpha, \beta, v_{n}) 
& = 
  \tfrac{16}{15}
  \left(
  N_{0}(x_{1}) + N_{0}(x_{2}) - 2M_{0}(x_{1}) - 2M_{0}(x_{2})
  + \alpha\cdot G_{0}(x_{1})
  + \beta\cdot G_{0}(x_{2})
  \right) 
  n + \mathcal{O}(1), \\
& = 
  \tfrac{16}{15}
  \left(
  \tfrac{2}{x_{1}} + \tfrac{2}{x_{2}} - \tfrac{4 x_{1}}{3} - \tfrac{4 x_{2}}{3}
  + \alpha\cdot \tfrac{4 x_{1}^{3}}{15}
  + \beta\cdot \tfrac{4 x_{2}^{3}}{15}
  \right) 
  n + \mathcal{O}(1),
\end{align*}
which justifies \eqref{ref10}.


\bigskip

\begin{acknowledgement}
	Authors H. Form\'{a}nkov\'{a} Lev\'{a} and G. Holubov\'{a} were supported by the Grant Agency of the Czech Republic, Grant No. 22--18261S.
\end{acknowledgement}
	\bibliographystyle{abbrv}
\bibliography{travelling_waves_references}
\end{document}